\documentclass[a4paper, english]{article}

\usepackage{version}

\usepackage[utf8]{inputenc}
\usepackage[T1]{fontenc}
\usepackage{babel}
\usepackage[babel]{csquotes}

\usepackage{amsmath}
\usepackage{amsfonts}
\usepackage{amsthm}
\usepackage{bbm}
\usepackage{amssymb}
\usepackage[top=3cm, bottom=3cm, left=3.3 cm, right=3.3 cm]{geometry}
\usepackage{mathrsfs}
\usepackage{stmaryrd}
\usepackage{hyperref}
\usepackage{cleveref}
\usepackage{lipsum}

\usepackage{thmtools}%
\usepackage{thm-restate}

\usepackage{float}
\usepackage{tikz}
\usetikzlibrary{shapes.misc}
\tikzset{cross/.style={cross out, draw, 
minimum size=2*(#1-\pgflinewidth), 
inner sep=0pt, outer sep=0pt}}
\usepackage{IEEEtrantools}
\usepackage{url}

\usepackage{enumitem}
\usepackage{accents}

\usepackage{sectsty}
\subsubsectionfont{\fontsize{11}{15}\selectfont}

\newcommand{\details}[1]{{\color{blue}#1}}

\newtheorem{theorem}{Theorem}[section]

\newtheorem{proposition}[theorem]{Proposition}
\newtheorem{fact}[theorem]{Fact}
\newtheorem{lemme}[theorem]{Lemma}
\newtheorem*{affirmation}{Claim}
\newtheorem{corollary}[theorem]{Corollary}
\newtheorem*{th*}{Theorem}
\newtheorem*{fact*}{Fact}
\newtheorem*{lemma*}{Lemma}
\theoremstyle{definition}
\newtheorem{definition}[theorem]{Definition}
\newtheorem{example}{Example}

\mathcode`l="8000
\begingroup
\makeatletter
\lccode`\~=`\l
\DeclareMathSymbol{\lsb@l}{\mathalpha}{letters}{`l}
\lowercase{\gdef~{\ifnum\the\mathgroup=\m@ne \ell \else \lsb@l \fi}}
\endgroup

\DeclareMathOperator{\GL}{\mathrm{GL}}

\DeclareMathOperator{\Cov}{\mathbb{C}ov}
\DeclareMathOperator{\Sym}{Sym}

\newcommand{\R}{\mathbb{R}}
\newcommand{\C}{\mathbb{C}}
\newcommand{\Z}{\mathbb{Z}}
\newcommand{\N}{\mathbb{N}}

\newcommand{\PP}{\mathbb{P}}

\newcommand{\E}{\mathbb{E}}
\newcommand{\1}{\mathbbm{1}}

\newcommand{\UN}{\mathcal{U}_{N}}
\newcommand{\DN}{\mathcal{D}_{N}}

\newcommand{\ku}{\mathfrak{u}}

\newcommand{\kh}{\mathfrak{h}}

\newcommand{\km}{\mathfrak{m}}

\newcommand{\kg}{\mathfrak{g}}
\newcommand{\g}{\mathfrak{g}}
\newcommand{\vv}{\underline{v}}

\newcommand{\w}{\omega}

\newcommand{\ut}{\underline{t}}

\newcommand{\xx}{\underline{x}}
\newcommand{\yy}{\underline{y}}
\newcommand{\zz}{\underline{z}}
\newcommand{\uu}{\underline{u}}
\newcommand{\rr}{\underline{r}}

\newcommand{\dkg}{\widehat{\kg}}
\newcommand{\tkg}{\widetilde{\kg}}
\newcommand{\tmu}{\widetilde{\mu}}
\newcommand{\teta}{\widetilde{\eta}}

\newcommand{\tkm}{\widetilde{\mathfrak{m}}}

\newcommand{\aalpha}{\underline{\alpha}}

\newcommand{\hmu}{\widehat{\mu}}

\newcommand{\hf}{\widehat{f}}

\newcommand{\hpsi}{\widehat{\psi}}
\newcommand{\hnu}{\widehat{\nu}}
\newcommand{\hrho}{\widehat{\rho}}

\newcommand{\eps}{\varepsilon}
\newcommand{\supp}{\text{supp}}

\newcommand{\kn}{\mathfrak{n}}

\newcommand{\thetaNgh}{\Theta^N_{p(g),p(h)}}
\newcommand{\uNgh}{\Upsilon^N_{g,h}}
\newcommand{\huNgh}{\widehat{\uNgh}}
\newcommand{\nn}{\underline{n}}
\newcommand{\ukappaN}{\underline{\kappa}_{N}}
\newcommand{\leb}{\text{leb}}

\newcommand{\XXab}{\overline{X}_{\mu}}
\newcommand{\XX}{X_{\mu}}

\newcommand{\dmu}{d_{\XXab}}
\newcommand{\dd}{d_{\Xab}}
\newcommand{\DilN}{D_{\sqrt{N}}}
\newcommand{\DilsN}{D_{\frac{1}{\sqrt{N}}}}
\newcommand{\Xm}{X}
\newcommand{\Xab}{\overline{X}}
\newcommand{\dg}{d_{\kg}}

\newcommand{\oo}{u}
\newcommand{\om}{\omega}

\newcommand{\norm}[1]{|#1|_{\kg} }  

\newcounter{namedthm}



\begin{document}

\title{Local limit theorems for random walks on nilpotent Lie groups}
\author{Timothée Bénard\thanks{The first author has received funding from the European Research
Council (ERC) under the European Union’s Horizon 2020 research and
innovation programme (grant agreement No. 803711).} \,and Emmanuel Breuillard}
\date{December 2023}

\maketitle

\large 

\begin{abstract}We establish the (non-lattice) local limit theorem for products of i.i.d. random variables on an arbitrary simply connected nilpotent Lie group $G$, where the variables are allowed to be non-centered. Our result also improves on the known centered case by proving uniformity for two-sided moderate deviations and allowing measures with a moment of order $2(\dim G)^2$ without further regularity assumptions. As applications we establish a Ratner-type equidistribution theorem for unipotent walks on homogeneous spaces and obtain a new proof of the Choquet-Deny property in our setting.
\end{abstract}

\bigskip

\addtocontents{toc}{\protect\setcounter{tocdepth}{0}}

\section{Introduction}

Given a Lie group $G$ and a probability measure $\mu$ on $G$, we may ask whether there are sequences $a_n>0$ and $g_n \in G$ such that the measures \begin{equation}\label{llconv} a_n \,\mu^{*n} * \delta_{g_n}\end{equation} converge weakly towards some non-trivial Radon measure on $G$. The measure $\mu^{*n} * \delta_{g_n}$ is the law of the random product $X_{1}\ldots X_{n} g_n$, where $X_1,\ldots,X_n$ are independent random variables with law $\mu$. An estimate of this kind is called a \emph{local (or local central) limit theorem} (LLT for short) on $G$. Such results are known for various kinds of Lie groups (see the historical remarks at the end of the introduction in \Cref{history}). In this paper we shall answer this question affirmatively for the class of simply connected nilpotent Lie groups. These groups are defined algebraically as those simply connected real Lie groups $G$ such that $G^{[s+1]}=\{1\}$ for some finite integer $s$, where $G^{[i+1]}:=[G,G^{[i]}]$, $G^{[1]}=G$, is the central descending series.  They can also be characterized precisely as the connected subgroups of unipotent upper triangular matrices (see \cite[Theorem 1.1.1]{corwin-greenleaf90}).

In the abelian case $G=\R^d$, the local limit theorem is a classical result \cite{gnedenko-kolmogorov68, breiman68}. Given a sum of $\R^d$-valued independent random variables $S_n=\sum_{k=1}^nX_{k}$ with common law $\mu$, drift $\chi=\E(X_1)$ and covariance the identity, we have as $n\to+\infty$, 
\begin{equation}\label{classicalLLT}
 (2\pi n)^{d/2} \,\mathbb{P}(S_{n}-n\chi \in B) =|B|+o_{\mu,B}(1),
 \end{equation} 
  where $B$ is any box  $B=\prod_{i=1}^d[s_i,t_i]$ of Euclidean volume $|B|$,
provided $\mu$ is non-lattice (i.e. the support of $\mu$ is not contained in a coset of a proper closed subgroup of $\R^d$). A uniform version of the local limit theorem is also available \cite{stone65}:
\begin{equation}\label{classicalLLT-uniform}
 \, \sup_{x \in \R^d} |\mathbb{P}(S_{n}-n\chi \in B+x) - (\sqrt{n}_{\star}\mathscr{N})(B+x)|  = o_{\mu,B}(n^{-d/2}),
 \end{equation} 
where $\sqrt{n}_{\star}\mathscr{N}$ is the push-forward by the homothety $x\mapsto \sqrt{n} \,x$ of  the standard  Gaussian distribution $\mathscr{N}:= (2\pi )^{-d/2}e^{-\frac{\|x\|^2}{2}}dx$ on $\R^d$.

\bigskip

In the context of non-commutative nilpotent Lie groups, the question of a local limit theorem has been investigated by the second-named author \cite{breuillard05} and by Diaconis and  Hough \cite{diaconis-hough21, hough19}, \emph{under the assumption that the measure $\mu$ is centered} in the abelianization $G/[G,G]$. Both \cite{breuillard05} and \cite{diaconis-hough21} deal with the case where the ambient group is the Heisenberg group, that is the simplest instance of a non-commutative nilpotent Lie group, while \cite{hough19} contains a generalization to arbitrary simply connected nilpotent Lie groups. As in \cite{breuillard05, hough19}, we will focus on the \emph{non-lattice} (equivalently aperiodic) case: we work under the assumption (see the remarks in \Cref{history}) that the projection $\mu_{ab}$ of $\mu$ under the abelianization map $G \to G/[G,G]$ is not supported on a coset of a proper closed subgroup of $G/[G,G] \simeq \R^d$.

 \begin{example}[Heisenberg group] \label{heis-ex} Let $G$ be the $3$-dimensional Heisenberg group, i.e.  $\R^3$ endowed with the group structure $x*y=(x_{1}+y_{1}, x_{2}+ y_{2}, x_{3}+y_{3}+\frac{1}{2}(x_{1}y_{2}-x_{2}y_{1})$. Let  $\mu$ be a probability measure on $G$ whose support is compact  and has  non-lattice projection in $\R^2$. Assume that $\mu$ is \emph{centered} in the sense that $\E_{\mu}(x_{1})=\E_{\mu}(x_{2})=0$.  It was shown in \cite{breuillard05} that for every  compactly supported  continuous function $f\in C_{c}(G)$, as $N\to+\infty$, 
 \begin{equation} \label{LLT-Heis-1}
 N^{2}\mu^{*N}(f) \rightarrow c_{\mu} \int_{G} f dx 
 \end{equation}
where $c_{\mu}>0$ only depends on $\mu$, and where $dx$ refers to the Lebesgue measure (or Haar measure) on $G$.

 \noindent As in the  abelian case, a uniform version of the LLT is also available. For $r>0$, set $D_r:x \mapsto (x_1 r, x_2 r, x_3 r^2)$. 
There exists a smooth probability measure $\nu$ on $G$ such that for every  $f\in C_{c}(G)$, for $N\geq1$, 
 \begin{equation} \label{LLT-Heis-2}
 \mu^{*N}(f)= D_{\sqrt{N}}(\nu)(f) + o_{\mu,f}(1)N^{-2}
 \end{equation}
where the term $o_{\mu,f}(1)$ goes to $0$ uniformly in the multiplicative translates of $f$, i.e. for the family of test functions $f_{g,h}: x \mapsto f(gxh)$ where $g,h\in G$. This was shown in \cite{breuillard05} (one sided uniformity only) where the two-sided uniformity required an additional Cram\'er assumption on $\mu$, an assumption subsequently lifted in \cite{diaconis-hough21}. 

\noindent Note that \cref{LLT-Heis-2} implies that $\nu$ is unique and coincides with the limiting measure in the central limit theorem, in other terms $D_{\frac{1}{\sqrt{N}}}(\mu^{*N}) \rightarrow \nu$ weakly (see \cite{crepel-raugi78, benard-breuillard-CLT}). Combining \eqref{LLT-Heis-1}  and \eqref{LLT-Heis-2} we see that  $c_{\mu}=u(0)$, where $u$ is the (smooth) density of $\nu$ with respect to $dx$. This density is nowhere-vanishing.  In fact $\nu$ corresponds to the law of the value at time $t=1$ of a hypoelliptic diffusion $(B^{(1)}_t,B^{(2)}_t,L_t)$ on $G$, where $(B^{(1)}_t,B^{(2)}_t)$ is a planar Brownian motion and $L_t$ is its L\'evy area (see e.g. \cite{friz-victoir10}). A well-known formula due to P. L\'evy \cite{levy51} allows to compute the Fourier transform of $\nu$ (see \cite[Th\'ero\`eme 1]{gaveau77}, \cite[Example 2]{benard-breuillard-CLT}), from which it follows that:
$$c_\mu= \frac{1}{4} \frac{1}{|\det{\Cov_\mu(x_1,x_2)}|^{\frac{1}{2}}}$$
where $\Cov_\mu(x_1,x_2)$ is the covariance matrix of the pair $(x_1,x_2)$ distributed according to $\mu_{ab}$. Note that this is insensitive to the mean of the central coordinate $\E_\mu(x_3)$.

\end{example}

The proof of the LLT for centered walks on the Heisenberg group obtained in \cite{breuillard05} relied on Fourier methods combined with certain spectral gap estimates for infinite dimensional irreducible unitary representations of the Heisenberg group. In a recent breakthrough \cite{diaconis-hough21}, Diaconis and Hough came up with a new idea that allowed them to bypass these estimates. They exploit the exchangeability of the random product under permutations of the variables in order to produce further smoothing in the central coordinate. This yields the desired decay of the Fourier transform.
With this new ``path-swapping'' technique they could obtain a new proof of the LLT for the Heisenberg group with improved (i.e. double-sided) uniformity. Quite remarkably, Hough \cite{hough19} was able to further these ideas and managed to prove the LLT for non-lattice centered walks on arbitrary nilpotent Lie groups. He also obtained the uniform LLT under the additional assumption that $\mu$ satisfies a strong form of aperiodicity (Cram\'er condition).

\bigskip

In the first part of this paper,  we push the Diaconis-Hough path-swapping method   one step further and show how it can be adapted to prove the local limit theorem on nilpotent Lie groups for \emph{non-centered} (we also say \emph{biased}) walks as well. Our result is the first of this kind, as all previous works on the nilpotent LLT assume the random walk to be centered.   It is worth pointing out  that the study of random walks on a non-abelian nilpotent  Lie group does not reduce to  the centered case. More precisely, after recentering, such a random walk may be seen as a process with centered but \emph{time-dependent} driving measure. There are at least two significant consequences. First,  the distribution $\mu^{*N}$ spreads at a faster rate around its average than  in the centered case. In fact, the large scale geometry of $\mu^{*N}$ is not captured by the descending central series (as for centered walks) but by some other filtration with smaller ideals,  which depends on the increment average.  We will need to work both with the descending central series (to exploit the bracket structure of $G$) and with this new filtration. The second novelty appearing in the non-centered case is that the limiting measure in the CLT is typically not of full support (even though it is absolutely continuous). This is a new phenomenon that came as a surprise in our previous work \cite{benard-breuillard-CLT} and forces some changes in the very formulation of the LLT in the non-centered case. In fact, the recentering sequence $g_n$ needs to be fine-tuned to remain within the range of the support of the limiting measure for the LLT to be meaningful, see \Cref{mu/Leb} below.


The limit theorems proven in this paper are the following.  \Cref{LLT} is the local limit theorem for non-lattice nilpotent random walks. It  allows the driving measure to be non-centered. It also provides two-sided moderate uniformity without a Cram\'er assumption, which is new even for centered walks, and useful in applications. An effective nilpotent LLT with better approximation rate but regularity conditions on the test function is also given in \Cref{eff-LLT}. We then deduce \Cref{mu/Leb}, which answers affirmatively the question raised above at \Cref{llconv} for non-lattice driving measures. For centered walks, we also manage to show a uniform  limit theorem for ratios, \Cref{mu/nu}. Such estimates are new even in the Heisenberg case.  The proof of  \Cref{LLT} and  \Cref{mu/Leb}   unifies both the centered and the non-centered cases and is valid for any non-lattice  probability measure on $G$ with a finite moment of high enough order, say $2(\dim G)^2$.  This moment condition also improves on \cite{breuillard05, diaconis-hough21, hough19}. A truncated version of the uniform LLT and its effective form is also established under a  moment of order $2+\eps$ (\Cref{LLT-tronc}, \Cref{TLLFouriercompact}).

\bigskip

The second part of the paper is devoted to two applications of our LLT for biased walks. The first, \Cref{Choquet-Deny}, gives a new proof of \emph{the Choquet-Deny theorem} for measures with finite moments on nilpotent Lie groups. This asserts that bounded $\mu$-harmonic functions are constant, a result proven by Guivarc'h in his thesis \cite{guivarch73} which extends the classical Choquet-Deny theorem \cite{choquet-deny60} on abelian groups. The second application, \Cref{equid-homogene}, concerns  the equidistribution of unipotent random walks on homogeneous spaces. In an unpublished chapter of the second author's thesis \cite[chapter 2]{breuillard-thesis04}, a \emph{probabilistic Ratner's theorem} was shown asserting that $\mu^{*n} * \delta_x$ weakly converges towards Ratner's invariant measure in any finite volume homogeneous space of a connected Lie group, under the assumption that $\mu$ is a finitely supported symmetric measure whose support generates a unipotent subgroup. Results of this type have since known a number of extensions to non-unipotent settings, in particular to subgroups whose Zariski-closure is semisimple, as in the breakthrough work of Benoist-Quint \cite{benoist-quint11} (see also \cite{benard22}). As a consequence of our LLT, we can now prove the equidistribution of random walks driven by  an arbitrary, possibly biased, probability measure with finite third moment and supported on a unipotent subgroup.


\bigskip

\begin{center}
*

*\,\,\,\,*
\end{center}


\bigskip

We now give a detailed presentation of  our main results. In the paper  $G$ will denote  a simply connected nilpotent Lie group  and $\kg $ its Lie algebra. The step (or nilpotency class) of $\kg$ is the smallest integer $s$ such that $G^{[s+1]}=\{1\}$, where $G^{[i]}:=[G,G^{[i-1]}]$ is the central descending series. In the sequel $G$ will be identified with $\kg$ via the exponential map $\exp: \kg \to G$. This allows to see the Lie product $*$ as a composition law on $\kg$.

Abusing notation we will continue to denote by $*$ the associated convolution product of measures on $\kg$.  Observe that $*$ is  in fact a polynomial map on $\kg$ as follows from the Baker-Campbell-Hausdorff formula \cite{dynkin47}:
$$x*y = x+y + \frac{1}{2}[x,y] + \frac{1}{12}[x,[x,y]] +\frac{1}{12}[y,[y,x]]+\dots $$
which contains only finitely many non-vanishing terms as $G$ is assumed nilpotent.

Given a probability measure $\mu$ on $(\kg,*)$ we consider the random walk  $S_N=X_1* \ldots * X_N$ where the $X_i$'s are independent with law $\mu$. We denote by $\mu_{ab}$ the projection of $\mu$ onto the abelianization $\kg_{ab}:=\kg/[\kg,\kg]$. Throughout the text, we assume that $\mu_{ab}$  has finite first moment and is not supported on an affine hyperplane. We also set $\XXab=\E(\mu_{ab}) \in \kg_{ab}$. Where further moment or aperiodicity assumptions are required they will be specified. 

\subsection{ The central limit theorem}
Before discussing the local limit theorem, we need to recall the central limit theorem (or CLT for short) for the random product $S_{N}$, in particular inasmuch as its formulation in the non-centered case requires the introduction of a subtle rescaling as we discovered in \cite{benard-breuillard-CLT}. The large scale behaviour of $S_N$ on the Lie group is captured by a certain decreasing sequence of ideals of $\kg$ whose definition depends only on $\XXab=\E(\mu_{ab}) \in \kg_{ab}$,
$$\kg=:\kg^{(1)} \supseteq \kg^{(2)} \supseteq \cdots  \supseteq \kg^{(t+1)} = \{0\}, $$
where $[\kg^{(i)},\kg^{(j)}]\subseteq \kg^{(i+j)}$ for all $i,j$. In  the centered case, i.e. when $\XXab=0$, the filtration coincides with the descending central  series, i.e. $\kg^{(i)}=\kg^{[i]}:=[\kg, \kg^{[i-1]}]$ and $t=s$. In general, the filtration is defined by induction, setting $\kg^{(0)} =\kg^{(1)}=\kg$ and $\kg^{(i)}=  [\kg, \kg^{(i-1)}] + [\XXab, \kg^{(i-2)}]$. 
We call $(\kg^{(i)})_{i\leq 2s-1}$ the \emph{weight filtration} associated to $\XXab$. 

It is shown in \cite{benard-breuillard-CLT} that for any choice of vector subspaces $\km^{(i)}\subseteq \kg$ such that $\kg^{(i)}= \km^{(i)} \oplus \kg^{(i+1)}$, the projection of $S_{N}$  to $\km^{(i)}$ (modulo the sum of the other $\km^{(j)}$) spreads at a rate $N^{i/2}$ around its average. Renormalizing accordingly, the law converges weakly and we obtain the central limit theorem (CLT). More precisely, let   \begin{align*}
 \pi^{(i)}: \kg &\to \km^{(i)}\\ x &\mapsto \pi^{(i)}(x)=x^{(i)}
 \end{align*} 
 be the coordinate projections induced by the decomposition $\kg=\oplus_{i} \km^{(i)}$. Set $\XX=\E_{\mu}(x^{(1)})$ and observe $\XX$ corresponds to $\XXab$ via the identification $\km^{(1)}\equiv \kg_{ab}$ induced by $\pi^{(1)}$. Introduce an associated one-parameter family of \emph{dilations} $(D_{r})_{r>0}$  by setting
  \begin{equation} \label{dilation} 
  D_r:\kg \to \kg, \,\,\,D_{r}(x)= \sum_{i \ge 1} r^{i}x^{(i)}.
  \end{equation} 
  The  central limit theorem  \cite[Theorem 1.1]{benard-breuillard-CLT} states that if $\mu$ has finite second moment for the weight filtration, then we have the  convergence in law
\begin{equation}\label{CLT}
\DilsN(\mu^{*N} * \delta_{-N \XX}) \to \nu 
\end{equation}
 where $\nu$ is a certain probability measure on $\kg$ with a smooth density. Here, one says that $\mu$ has finite $m$-th moment for the weight filtration if the map $x\mapsto \|x^{(i)}\|^{m/i}$ is in $L^1(\mu)$ for any $i$ (we take some norm, say a Euclidean norm, on each $\km^{(i)}$). We note in passing that since $\kg^{[i]}\subseteq \kg^{(i)}$ and the inclusion may be strict, this condition is \emph{weaker} than the ordinary notion of moment on a nilpotent Lie group. The ordinary notion can be defined in the same way via the central descending series rather than the weight filtration; equivalently it means that the distance to the identity (for some, or any, left invariant riemannian metric) has a finite $m$-th moment.
 
  The CLT also involves a description of the limiting measure $\nu$ as the value at time  $t=1$ of a certain explicit hypoelliptic diffusion process $(\sigma_{t})_{t\in \R^+}$ on $\kg$. The infinitesimal generator of $(\sigma_{t})_{t\in \R^+}$ is time-dependent, and left invariant for a new group structure  $*'$ on $\kg$, inherited from the choice of subspaces $\km^{(i)}$. This new group structure is defined by  the formula
\begin{equation}\label{*'-product}
a*'b=\lim_{t \to +\infty} D_{1/t}(D_t(a)*D_t(b)).
\end{equation} 
With this new Lie structure, $(\kg,*')$ is again an $s$-step nilpotent Lie group and dilations act as automorphisms. The semigroup  $(\sigma_{t})_{t>0}$ is then  defined (as in   \cite{hunt56}) by way of its infinitesimal generator
\begin{align} 
\mathscr{L}_{t}=\frac{1}{2}\sum_{i=1}^{q}  \left(\exp(ta_{\XX}) E_{i}\right)^2+ \exp(ta_{\XX})B_\mu
 \end{align}
where $E_{i}, B_\mu$ are left invariant vectors fields on $(\kg, *')$ such that the $E_i$'s  form a basis of $\km^{(1)}$ in which the covariance matrix of $\mu_{ab}$ is the identity,  and $B_\mu= E_\mu(x^{(2)}) \in \km^{(2)}$. Here 
 $a_{\XX}:\kg \to \kg$ is the nilpotent operator defined by
  \begin{equation}\label{axdef}a_{\XX}(y)=\pi^{(i+2)}([\XX,y])\, \textnormal{ for } y \in \km^{(i)}.\end{equation}
  
The properties of the limiting measure $\nu=\sigma_{1}$ are studied in \cite{varopoulos-saloff-coulhon92} in the centered case and \cite{benard-breuillard-CLT} in  general. It is proven that $\nu=v(x)dx$ where $v$ is a smooth and rapidly decaying (Schwartz) function on $\kg$. The  measure $\nu$ is an ordinary Gaussian distribution on $\kg$ (in the Euclidean sense) if and only if $[\XXab, \kg^{[i]}]=[\XXab, \kg^{[i+1]}]$ for every $i\leq s$, which corresponds exactly to situations where the eigenvalues of the renormalization $D_{\frac{1}{\sqrt{n}}}$ are odd negative powers of $\sqrt{n}$, see \cite[Theorem 1.4]{benard-breuillard-CLT}.  Usually there is no explicit formula for $v(x)$ nor its Fourier transform (the case of the Heisenberg group \cite[Th\'eor\`eme 1]{gaveau77} is an exception). However, satisfactory estimates hold when $\nu$ arises from a centered walk: there exist $A, \alpha>0$ such that 
$$\alpha \exp(-A |x|^2) \leq v(x)\leq A\exp(-\alpha |x|^2) $$
where $|x|$ denotes the distance to the origin for some (any) left-invariant Riemannian metric on $(\kg,*)$. 
Apart from the centered case, or when the group is abelian or of nilpotency class $2$ and certain exceptions (such as the full group of unipotent upper triangular matrices), it is usually the case that the support of the density $\nu$ is a \emph{proper} closed subset of $\kg$, see \cite[Section 3.3.2]{benard-breuillard-CLT}. This came as a surprise to us.

\subsection{The non-centered local limit theorems}

We are now in a position to state the main results of this paper.

\begin{definition} \label{hom-dim} Let $\dmu $ be the integer:   $$\dmu = \sum_{i \ge 1} \dim \kg^{(i)},$$ where $\kg^{(i)}$ is the \emph{weight filtration} introduced above.
\end{definition}

We call $\dmu$ the \emph{homogeneous dimension of the $\mu$-random walk}, or more precisely of the weight filtration  induced by $\XXab$. Note that $\dmu$ is related to the determinant of the dilations:  $\det D_{r}=r^{\dmu }$. When $\mu$ is centered, $\dmu $ coincides with the homogeneous dimension of $\g$, namely $d_\kg=\sum_{i \ge 1} \dim \kg^{[i]}$, which controls the volume growth of balls in $G$ by the Bass-Guivarc'h formula \cite{guivarch73, breuillard14}. We further note that $\dmu  \ge d_\kg$ with equality if and only if the weight filtration coincides with the descending central series, i.e. $\kg^{(i)}=\kg^{[i]}$ for all $i$. 

In the next theorem we  assume that the projection $\mu_{ab}$ is \emph{non-lattice}; this means that its support is not contained in a translate of a proper closed additive subgroup of $\kg_{ab}$.

\begin{theorem}[Non-centered uniform LLT]  \label{LLT} Let $G$ be a simply connected nilpotent Lie group with Lie algebra $\kg$.
Let $\mu$ be a probability measure on  $G$ with  finite  moment of order $m_{\mu}> \dmu +2$ for the weight filtration induced by $\XXab=\E(\mu_{ab})$, and such that $\mu_{ab}$ is non-lattice. Then for every continuous and compactly supported function $f$ on $\kg$,  for all $N\geq 1$, 
\begin{align} \label{non-centered-LLT}
|(\mu^{*N} *\delta_{-N\XX}) (f) -(\DilN \nu)(f)| =o_{\mu, f}(N^{-\frac{1}{2} \dmu }),
\end{align}
where $\nu$ is the limit measure from the central limit theorem \eqref{CLT}, and $\dmu $ is the homogeneous dimension of the $\mu$-random walk. 

Moreover for $\eps_{0}>0$ small enough, the error rate in \eqref{non-centered-LLT} is uniform in all the bideviations $\delta_g * f * \delta_h : x\mapsto f(g*x*h)$ for $g,h\in \kg$, provided that either $g$ or $h$ belongs to 
$$\mathcal{E}_{N,\eps_{0}}:=\{x\in \kg \,:\, \|x^{(i)}\| \leq N^{\frac{i}{2}+\eps_{0}}\} + [-N^{1+\eps_{0}}, N^{1+\eps_{0}}]X_{\mu}. $$
\end{theorem}

In \cite{breuillard05} this result was proven for the Heisenberg group in the special case of a centered $\mu$ of compact support. As discussed above, Diaconis and Hough \cite{diaconis-hough21} gave a new proof in the centered Heisenberg case.  Later, Hough \cite{hough19} obtained the LLT for centered walks on arbitrary nilpotent Lie groups. His result allows deviations of the test function $f$ only under the further assumption that $\mu_{ab}$ satisfies the Cramér condition, which is a strong form of aperiodicity  typically concerning spread out measures, or discrete measures with good diophantine properties (see \cite{breuillard-dioph, angst-poly}). We allow deviations without Cramér assumption and this is important for our applications in \Cref{Sec-applications}. 

\Cref{LLT} is the first of its kind to allow the measure $\mu$ to be non-centered. The  proof is again via characteristic functions.  It is inspired by the path swap technique used by Diaconis and Hough in \cite{diaconis-hough21, hough19}. Since their method was tailored to work for the bracket structure of the descending central series, we will need extra work to obtain results related to the weight filtration (they usually do not coincide in the non-centered case). 

We also note that we improve on the moment assumption compared to  \cite{breuillard05,diaconis-hough21,hough19} even in the centered case. This improvement is performed by truncating suitably the measure $\mu$. \Cref{LLT} will in fact be a corollary of a local limit theorem for a truncated version of $\mu$, namely \Cref{LLT-tronc},  available under the sole  condition that $\mu$ has a finite moment of order $2+\eps$ for some $\eps>0$. This  truncated formulation will turn out useful  for applications later.

\bigskip

\noindent\emph{Remark}. In the \emph{centered Heisenberg} case, the LLT proved by Diaconis and Hough  in \cite{diaconis-hough21} is uniform in all deviations $g,h\in G$, and also allows the measure to be lattice. We do not know whether this extra uniformity holds in the  setting of \Cref{LLT}. In the Heisenberg case, it relies on the observation that the $L^1$-norm  of the Fourier transform of a test function is constant under two-sided deviation, but this is not true in a general nilpotent group. The LLT for lattice measures on a general nilpotent Lie group is also still open. It requires a more delicate analysis that has not been worked out yet, even in the centered case (with the notable exception of centered walks on discrete nilpotent groups by work of Alexopoulos \cite{alexopoulos02-3}, see \Cref{history}).

\bigskip

If $f$ is in $L^1(\kg)$ and its Fourier transform decays fast enough, we also obtain an effective local limit theorem with power saving. In other words, we make  explicit the dependency on $f$ in the error term, and improve the rate of approximation from $N^{-\frac{1}{2}\dmu }$ to $N^{-\frac{1}{2}(\dmu +1-\eps)}$. To do so, 
given $c\in(0,1)$, $L>1$, we introduce the function $gap_{c}: \R_{>0}\rightarrow (0, c)$   given by 
$$gap_{c}(R)= c\inf \{1-|\hmu_{ab}(\xi')| \,:\, \xi' \in \dkg_{ab}, \, c\leq \|\xi'\|\leq c^{-1}(1+R)\}$$
and we define the norm 
\begin{align*}
\|f\|_{\mu_{ab}, c, L}:= \|f\|_{L^1}+ \int_{\dkg} |\hf(\xi)| (1+\|\xi\|)^L gap_{c}(\|\xi\|)^{-L} d\xi.   
\end{align*}
Here $\hmu_{ab}$, $\hf$ refer respectively to the Fourier transforms of $\mu_{ab}$ and $f$.  

\begin{theorem}[Effective LLT with power saving]  \label{eff-LLT} 
Keep the assumptions of \Cref{LLT}, suppose also  $m_{\mu}\geq (\dmu +4)$, and let $\eps>0$. 
There exist constants $c=c(\mu)\in (0, 1)$, $C=C(\mu, \eps)>1$, $L=L(\dim \kg, \eps)>1$ such that  
 for every continuous integrable function $f$ on $\kg$, every $N\geq 1$, we have
 $$|(\mu^{*N} *\delta_{-N\XX}) (f) -(\DilN \nu)(f)|\, \leq \, C  \|\delta_{g}* f* \delta_{h}\|_{\mu_{ab}, c, L} \,N^{ -\frac{1}{2}(\dmu+ 1-\eps)} $$
for every $g,h$ in  $\mathcal{E}_{N,1/L}$.
\end{theorem}

Following the terminology introduced in \cite{breuillard-dioph} in the case when $\kg=\R^d$, we will say that the measure $\mu$ on $\kg$ is \emph{$\kappa$-diophantine} if $\mu_{ab}$ is $\kappa$-diophantine in the sense of \cite[Definition 7.1]{breuillard-dioph}, namely if there is $c>0$ such that 
$$\inf_{t \in \R} \E_{x\sim \mu_{ab}}(\{\xi(x)+t\}^2) \ge \frac{c}{\|\xi\|^\kappa}$$
for all large enough $\xi \in \widehat{\kg_{ab}}$, where $\{\cdot\}$ denotes the distance to the nearest integer. This condition is easily seen to be equivalent to the requirement that for some (or each) $c>0$ there is $c'>0$ such that for all $R>0$
$$gap_c(R) \geq \frac{c'}{(1+R)^{\kappa}},$$
a condition studied in \cite{angst-poly} under the name of \emph{weak Cram\'er condition} in the classical setting when $\g=\R^d$. When this condition holds, Theorem \ref{eff-LLT} becomes:

\begin{corollary}[LLT under a diophantine condition] Keep the assumptions of \Cref{LLT}, suppose also  $m_{\mu}\geq (\dmu +4)$, and let $\eps>0$. Let $\kappa>0$ and assume that $\mu$ is $\kappa$-diophantine. Then there are  constants $C=C(\mu, \eps)>1$ and $L=L(\dim \kg, \eps)>1$ such that  
 for every continuous integrable function $f$ on $\kg$, every $N\geq 1$, we have
 $$|(\mu^{*N} *\delta_{-N\XX}) (f) -(\DilN \nu)(f)|\, \leq \, C  S_{L(1+\kappa)}(\delta_{g}* f* \delta_{h}) \,N^{ -\frac{1}{2}(\dmu+ 1-\eps)} $$
for every $g,h$ in  $\mathcal{E}_{N,1/L}$.
\end{corollary}
Here  $S_L$ denotes the Sobolev norm:
$$S_L(f):=\|f\|_{L^1} +  \int_{\dkg} |\hf(\xi)| (1+\|\xi\|)^L d\xi.  $$

It is worth observing that the diophantine condition on which the quality of the error depends is expressed only in terms of the projection of $\mu$ to the abelianization $\kg_{ab}$. This feature is common to many other situations regarding equidistribution in nilpotent Lie groups, such as in \cite{green-tao}.

\bigskip

The local limit theorem reduces the estimation of $\mu^{*N} *\delta_{-N\XX} (f)$ to that of $\DilN\nu(f)$. The latter is easier to understand, because the limit measure $\nu$ has a smooth and rapidly decaying density as we have discussed. However $\nu$ does not always have full support. In particular if the density of $\nu$ vanishes at the origin, the LLT in the  form of \eqref{non-centered-LLT} does not provide meaningful asymptotics for the decay of $\mu^{*N} *\delta_{-N\XX} (f)$, but only an upper bound in $o_{\mu, f}(N^{-\frac{1}{2}\dmu})$. One way to remedy this is to use the uniformity with respect to deviations in our LLT in order to change the recentering appropriately as follows. Recall that $\nu$ has a smooth density $v(x) \ge 0$ on $\kg$. We have:

\begin{corollary}[LLT with variable recentering] \label{mu/Leb}
Keep the assumptions of \Cref{LLT}, let $Y \in \kg$ and set $g_N=-N\XX*-\DilN Y$. Then for every continuous and compactly supported function $f$ on $\kg$ we have:
$$N^{\frac{1}{2}\dmu  } \,(\mu^{*N}*\delta_{g_{N}})(f) \underset{N\to +\infty}{\longrightarrow} v(Y) \int_{\kg}f(x) dx.$$ 
\end{corollary}

In particular this answers positively the question asked in the introduction as the quantity in \Cref{llconv} converges to a non-trivial multiple of the Haar measure on $G$ provided we choose $Y$ such that $v(Y)>0$. The following variant can also be derived in the centered case:

\begin{corollary}[Ratio LLT with controlled deviations] \label{mu/nu}
Assume that $\mu_{ab}$ is \emph{centered}, non lattice,  and $\mu$ has finite moment of order $m_{\mu}>d_{\kg}+2$. For all non-zero $f\in C_{c}(\kg)$ with $f\geq0$, one has 
$$\frac{\mu^{*N}(f)}{\DilN\nu(f)} \underset{N\to +\infty}{\longrightarrow} 1$$
and for any fixed $c>0$ sufficiently small, the convergence is uniform as we replace $f$ by the bitranslates $\delta_g * f * \delta_h$ with
$$g,h\in \{x\in \kg \,:\, \|x^{(i)}\|\leq c \,(N\log N)^{i/2}\}.$$
\end{corollary}

 Due to the possible vanishing of the limiting measure, the assumption that the walk is centered seems unavoidable to guarantee uniformity in the ratio LLT.

\subsection{Applications of the LLT}

We now present two applications of our LLT (\ref{LLT}, \ref{LLT-tronc}).

\bigskip
\noindent{\bf The Choquet-Deny property}.
Let $G$ be a locally compact group, and $\mu$ a probability measure on $G$. A function $f\in L^\infty(G)$ is said to be \emph{$\mu$-harmonic} if for almost-every $x\in G$, 
$$f(x)=\int_G f(xg)\,d\mu(g).$$
Note that every function $f$ that is right-invariant by the support of $\mu$, i.e. such that $f*\delta_g=f$ for every $g \in \supp(\mu)$, must be $\mu$-harmonic. 

 $G$ is said to have the \emph{Choquet-Deny property} if the converse holds: for any measure $\mu$, any bounded $\mu$-harmonic function $f\in L^\infty(G)$ is right-invariant by the support of $\mu$. True for $\Z^d$ \cite{blackwell55} and discrete nilpotent groups \cite{dynkin-malioutov61}, this property is by now completely understood when the group $G$ is discrete, thanks to the groundbreaking recent work of Frisch, Hartman, Tamuz and Vahidi-Ferdowsi \cite{frisch-hartman19} who showed that a finitely generated group is Choquet-Deny if and only if it is virtually nilpotent.  For totally disconnected groups of polynomial growth, several characterizations of the Choquet-Deny property are given in \cite{jaworski-raja07}. The case of continuous groups however remains more subtle.  Guivarc'h in his thesis \cite{guivarch73} established the property for all locally compact nilpotent groups under the assumption of a finite moment (of arbitrarily small but positive order) on the measure $\mu$, while the general case with no moment condition is still an open problem, despite \cite{raugi04}.  Guivarc'h's proof relied on unitary representations. Assuming that $G$ is a nilpotent Lie group and $\mu$ has a moment of order $2+\eps$, we provide a different argument, based on our LLT. 

\begin{corollary}[Bounded $\mu$-harmonic functions] \label{Choquet-Deny}
Let $G$ be a simply connected nilpotent Lie group. Let $\mu$ be a probability measure on $G$ with a finite moment of order strictly greater than $2$ for the weight filtration induced by $\XXab=\E(\mu_{ab})$, and whose support generates a dense subgroup of $G$. Then bounded $\mu$-harmonic functions on $G$ are constant. 
\end{corollary}

The condition that $G$ is \emph{simply} connected is not essential here, as one may always lift $\mu$ and $\mu$-harmonic functions to the universal cover of $G$. It is however convenient to express our moment condition on $\mu$.


\bigskip
\noindent{\bf Probabilistic Ratner theorem}.
Let $H$ be a real Lie group and $\Lambda \leq H$ a lattice (discrete subgroup with finite covolume). Let $U$ be a simply connected nilpotent Lie group and $\rho:U \to H$ a Lie group homomorphism such that $Ad \circ \rho$ is contained in a unipotent subgroup of $\GL(\kh)$, where $\kh=Lie(H)$ is the Lie algebra of $H$. Considering the resulting action of $U$ on $\Omega:=H/\Lambda$, Ratner established \cite{ratner91, morris05} that the closure of every $U$-orbit is itself a homogeneous subspace of finite volume: for each $\w \in \Omega$, the stabilizer $L_{\w}=\{h\in H, \, h.\overline{U\w}=\overline{U\w}\}$ acts transitively on  $\overline{U\w}$ while preserving a probability measure on $\overline{U\w}$. This measure is unique and we denote it by  $m_{\w}$. Ratner's proof is based on an equidistribution theorem for one-parameter unipotent flows, that was subsequently extended to the  multi-parameter case by Shah \cite{shah94}. Combining this with our LLT we obtain:

\begin{corollary}[Equidistribution in homogeneous spaces] \label{equid-homogene}
In the above setting, let $\mu$  be a probability measure on $U$ whose support generates a dense subgroup of $U$. Assume $\mu$ has a finite moment of order strictly greater than $2$ for the weight filtration induced by $\XXab=\E(\mu_{ab})$. Then for every  $\w\in \Omega$, we have the following weak-$\star$ convergence of probability measures:
$$\frac{1}{N}\sum_{k=0}^{N-1} \mu^k*\delta_{\w}\underset{N\to \infty}{\longrightarrow} m_{\w}. $$
If, additionally, $\mu$ is assumed centered and aperiodic, then $(\mu^{*N}*\delta_\w)_{N \ge 1}$ itself converges to $m_\w$.
\end{corollary}

We stress that the LLT is used not only through the Choquet-Deny property and its consequence   that $\mu$-stationary measures are invariant. Indeed, we also rely crucially on the LLT to establish the non-escape of mass of the above sequences of probability measures. It ensures that any limit measure  has mass $1$ and is absolutely continuous with respect to $m_\w$.

The  moment condition in  \Cref{equid-homogene} may not be optimal. However the other assumptions are tight. If $\mu$ fails to be centered, or aperiodic, then easy examples show that the convergence of individual powers fails and only the Cesaro convergence holds (see \cite[10.2.1]{breuillard05}).  

\Cref{equid-homogene} generalizes the unpublished chapter \cite[Chapter 2]{breuillard-thesis04}, where this result was proven by a different method (combining work of Guivarc'h and Varopoulos) in the special case of a symmetric and finitely supported measure $\mu$, as well as the case of centered walks on the Heisenberg group, treated in \cite{breuillard05}.

We also point out the analogous fact for measures supported on subgroups whose Zariski-closure is semisimple is a major result of Benoist-Quint \cite{benoist-quint11} in the Cesaro case, while convergence of individual powers was recently established by the first named author in \cite{benard22} in that setting, see also \cite{BFLM11, HLL16}.

\subsection{Historical remarks}\label{history}

Random walks and limit theorems on non-commutative Lie groups have a long history. We refer the reader to the early book \cite{grenander63} and survey \cite{tutubalin-sazonov66} for background on non-commutative analogues of the law of large numbers and central limit theorems, and more recent references such as \cite{saloff-coste01, breuillard-thesis04, benoist-quint-book, liao-book, varopoulos21}. The local limit theorem has been studied for compact groups in \cite{ito-kawada40, arnold-krylov63, bourgain-gamburd12, benoist-saxce16}, for the group of Euclidean rigid motions in \cite{kazhdan65, guivarch76, varju15, lindenstrauss-varju} and semisimple Lie groups in \cite{bougerol, kogler22}. On nilpotent Lie groups, the central limit theorem was studied in \cite{tutubalin64, virtser74, crepel-raugi78, raugi78, bentkus-pap96, benard-breuillard-CLT}. Random walks on nilpotent covers of compact manifolds or finite graphs have also attracted a lot of attention, see \cite{ishiwata03, ishiwata-kawabi-namba-I, ishiwata-kawabi-namba-II}, where certain limit theorems for centered and non-centered walks are also obtained.

A completely new method was brought forward by Varopoulos in the 80s (see e.g. \cite{varopoulos-saloff-coulhon92}) and further developed by Alexopoulos culminating in \cite{alexopoulos02,alexopoulos02-2,alexopoulos02-3}. This method provides very complete results on the asymptotics of convolutions powers $\mu^{*N}$ of a measure $\mu$ on a nilpotent group (either Lie or  discrete), including central and local limit theorems, even with the full Edgeworth expansion. However there are two caveats. The first is that this method requires the measure $\mu$ to have a density with respect to the Haar measure, an assumption which seems unavoidable. The second is that the limit theorems that have been obtained by this method always assumed the measure to be centered. 

The method used in our paper however (as well as that of \cite{breuillard05, diaconis-hough21, hough19, benard-breuillard-CLT}) is based on the analysis of characteristic functions and allows to handle probability measures without any assumption on their regularity or whether they are centered or not. While our results assume the measure to be non-lattice, it remains an open problem to handle general lattice measures (even  centered ones) on general nilpotent groups.

\subsection{The path swap principle behind the proof of \Cref{LLT}}
The approach to prove \Cref{LLT} relies on characteristic functions  as in the classical proof of the LLT for abelian groups.  Similarly to previous works, the main part of the proof consists in establishing super-polynommial decay of the Fourier transform outside of a shrinking domain around the origin: for every $Q\ge1$, $\gamma_{0}>0$,  \begin{equation}\label{foudecay}|\widehat{\mu^{*N}}(\xi)| \ll_{Q, \gamma_{0}} N^{-Q}\end{equation}  outside the range of  Fourier parameters $\xi \in \widehat{\kg}$ such that $\|\xi \circ D_{\sqrt{N}}\| \ll N^{\gamma_{0}}$.

In \cite{breuillard05}, a spectral gap estimate for the infinite dimensional irreducible unitary representations of the Heisenberg group was established for this purpose. In \cite{diaconis-hough21}, Diaconis and Hough proposed a different method to prove $(\ref{foudecay})$ in the centered Heisenberg case, which is reminiscent of the well-known \emph{van der Corput trick} in the proof of Weyl's equiditribution theorem for polynomials, where one differentiates the polynomial phase to produce cancellation in an exponential sum (see e.g. \cite[\S 4]{green-tao} for an example of its use in the context of equidistribution of ergodic flows on nilmanifolds). In  \cite{diaconis-hough21}, this discrete derivative is replaced by a difference of permutations acting on the variables, which after iteration, generates some independence. This path swap principle is also exploited by \cite{hough19} to prove  \eqref{foudecay} for centered walks on arbitrary nilpotent groups, and we use it as well for \Cref{LLT}.

We now give a brief outline of this principle (see \Cref{Sec-path swap} for details). Given a subgroup of permutations $F\leq \Sym(N)$ we may write $$\widehat{\mu^{*N}}(\xi) = \E_{\mu^{\otimes N}}\big(e_\xi(\Pi\xx)\big)= \E_{\mu^{\otimes N}}\big(\E_{\sigma \in F}(e_\xi(\Pi\sigma\xx))\big)$$ where $e_\xi(x)=\exp(-2i\pi \xi.x)$ and $\Pi\xx=x_{1}*\dots *x_{N}$ is the Lie product of the variables $\xx=(x_1,\ldots,x_N)$ and $\Sym(N)$ acts by permuting the variables. By Cauchy-Schwarz:
\begin{equation} \label{eq-1-pathswap-pple}
|\widehat{\mu^{*N}}(\xi)|^2 \leq \E_{\mu^{\otimes N}} \big( |\E_{\sigma \in F}(e_\xi(\Pi\xx))|^2\big)=\E_{\mu^{\otimes N}} \big(\E_{\sigma,\tau \in F}(e_\xi(\Pi\sigma\xx-\Pi\tau\xx))\big).
\end{equation}

The right hand side can be rewritten $\E_{\sigma,\tau \in F} \big(\E_{\mu^{*N}} (e_\xi((\sigma-\tau)\Pi\xx))\big)$ where $\sigma-\tau$ is now viewed as an element of the group algebra $\R[\Sym(N)]$ acting on the free Lie algebra. 

We can then iterate this procedure $s-1$ times, averaging successively over commuting subgroups $F_1,\ldots,F_{s-1}$. If  $F=F_1\times \ldots \times F_{s-1}$,
$$|\widehat{\mu^{*N}}(\xi)|^{2^{s-1}} \leq \E_{\sigma,\tau \in F} \big(\E_{\mu^{\otimes N}} (e_\xi(A_{\sigma,\tau}\Pi\xx))\big)$$
where $A_{\sigma,\tau}=\prod_1^{s-1}(\sigma_i-\tau_i) \in \R[\Sym(N)]$ and $\sigma=(\sigma_1,\ldots,\sigma_{s-1}) \in F$.

We now choose the $F_{i}$'s so that for any fixed $\sigma, \tau \in F$, the term $A_{\sigma,\tau}\Pi\xx$ is a sum of independent variables as $\xx$ varies with law $\mu^{\otimes N}$.  Assuming  $N=N'(2s-1)k$ for some $k,N'$, we may split the $N$ variables $x_i$ into $N'$ consecutive large blocks $(B_{j})_{j\leq N'}$ of length $(2s-1)k$. We see each large block as a concatenation of $(2s-1)$ blocks of length $k$. We then let $F_{i}\simeq (\Z/2\Z)^{N'}$ be the group generated by the  transpositions  exchanging the two blocks at distance $i$ from the central block in some $B_{j}$. As the $F_{i}$ have disjoint support, they do commute. The combinatorial properties of the Lie product on a nilpotent group imply that: $$A_{\sigma,\tau}\Pi\xx=\sum_{j=1}^{N'} A_{\sigma,\tau}u_j,$$  where each $u_j=\Pi \xx_j$ is a Lie product of the variables $x_i$ whose indices $i$ belong to the block $B_{j}$. It is therefore a sum of $N'$ uncorrelated terms. And hence:
$$|\widehat{\mu^{*N}}(\xi)|^{2^{s-1}} \leq \E_{\sigma,\tau \in F} \prod_{j=1}^{N'}\E_{\mu^{\otimes N}} (e_\xi(A_{\sigma,\tau}u_j))= \big( \E_{\sigma,\tau} \E_{\mu^{\otimes k(2a-1)}} (e_\xi(A_{\sigma,\tau}u_1) \big)^{N'}.$$
Furthermore the choice of $F$ will ensure that $A_{\sigma,\tau}u_1$ is a sum of $s$-brackets (i.e. pure commutators of length $s$), hence belongs to the center of $\kg$. Choosing $k \sim N^\eps$ and $N'\sim N^{1-\eps}$, this will eventually lead to the desired super-polynomial decay estimate $(\ref{foudecay})$.

\subsection{Organization of the paper} In \Cref{Sec-preliminary}, we  set up the notations to be used in the rest of the paper and we recall from \cite{benard-breuillard-CLT} several techniques of truncation, graded replacement and  Lindeberg replacement, that help  estimate the Fourier transform of the random product. \Cref{Sec-uniform-LLT} is devoted to the proof of the uniform local limit theorem and its effective version. In \Cref{Sec-LLT-ratio}, we prove the local limit theorem with variable recentering or controlled deviations. Finally, in \Cref{Sec-applications}, we apply the uniform local limit theorem to prove the Choquet-Deny property on connected nilpotent Lie groups, and the equidistribution of unipotent walks on finite volume homogeneous spaces.


\subsection{Acknowledgements} We are grateful to  Yves Benoist, Peter Friz and Laurent Saloff-Coste for very helpful discussions. We also thank the anonymous referee for their vigilant reading and relevant comments that, among other things, helped us correct an inaccuracy in \Cref{moyenne3} and improve the presentation of several arguments.

\tableofcontents


\addtocontents{toc}{\protect\setcounter{tocdepth}{2}}

\section{Notation and preliminary background} \label{Sec-preliminary}

The goal of this preliminary section is to collect the terminology and notation used throughout the paper, and to recall useful results from \cite{benard-breuillard-CLT} on which we will build the proof of the LLT.

\subsection{Set-up and notation} \label{Sec2} \label{Sec-cadre}

\subsubsection{  The ambient space and its bias extension} \label{weight-filtration}

Let $(\kg, [.,.])$ be a nilpotent Lie algebra of step $s\geq 1$ with associated Lie product $*$, and $\Xab \in \kg/[\kg,\kg]$. The   \emph{weight filtration} of $\kg$ associated to $\Xab$ is  the decreasing sequence of ideals defined recursively by setting $\kg^{(0)}=\kg^{(1)}=\kg$ and for $i\ge 1$
\begin{equation*}
    \kg^{(i+1)}=[\kg,\kg^{(i)}]+[x,\kg^{(i-1)}].
\end{equation*}
where $x\in \Xab$ is a fixed choice of representative of $\Xab$ in $\kg$. 
 
\begin{lemme} \label{basics-weight-filtration}
We have $\kg^{(2)}=[\kg, \kg]$ and for $k\geq 3$, the ideal $\kg^{(k)}$ is linearly spanned by iterated brackets $[x_{1}, [x_{2},\dots [x_{j-1}, x_{j}]\dots ]]$ with $j\ge k$, and also with $j<k$ provided at least $k-j$ of the $x_i$ are equal to $\Xab$ modulo $\kg^{(2)}$. In particular, the weight filtration depends only on $\Xab$ (and not on the choice of $x$) and  if $i,j \ge 1$, 
\begin{equation}\label{nestedg}
    [\kg^{(i)},\kg^{(j)}] \subseteq \kg^{(i+j)}.
\end{equation}
Furthermore $\kg^{[i]}\subseteq \kg^{(i)} \subseteq \kg^{[i']}$ where $i'=\lfloor i/2 \rfloor +1$, and $\kg^{(2s)}=0$.
\end{lemme}

 \begin{proof} Set $\mathfrak{f}^{(0)}=\mathfrak{f}^{(1)}=\kg$, $\mathfrak{f}^{(2)}= [\kg, \kg]$ and for $k\geq 3$, let $\mathfrak{f}^{(k)}$ be the linear span of all commutators as in the statement. One observes that $(\mathfrak{f}^{(i)})_{i}$ satisfies the same recurrence relation as that defining $(\kg^{(i)})_{i}$. Therefore  $\mathfrak{f}^{(i)}=\kg^{(i)}$ for all $i$. The other  statements follow directly.
 \end{proof}
 
For each  $i\leq 2s-1$, we choose once and for all a vector subspace $\km^{(i)}$ in $\kg$ such that 
 $$\kg^{(i)} = \km^{(i)}\oplus \kg^{(i+1)}. $$
It follows that $\kg$ is the (vector space) direct sum 
 $$\kg=\km^{(1)} \oplus \dots \oplus \km^{(2s-1)}.$$
 We call such a direct sum (with choices of $\km^{(i)}$) a \emph{weight decomposition} of $\kg$. We  may identify $\km^{(1)}$ with $\kg/[\kg,\kg]$. We denote by  $\Xm$  the unique element of $\km^{(1)}$ projecting to  $\Xab$.

 It is useful to introduce an extension of $\kg$, realized by adding to $\kg$ an independent copy $\chi$ of $\Xm$. Given a random walk on $(\kg,*)$, this procedure will allow to distinguish between the influence of the average of the increment and the variations around this average.   If  $\Xab=0$, we just set $\tkg=\kg$, $\chi=0$. Otherwise, we define $\tkg$ as the Lie algebra direct sum:
$$\tkg:=\kg \oplus \R$$
and set $\chi=(\Xm,1)$. In particular, $\tkg=\kg\oplus \R \chi$ (as vector spaces) and $[\Xm, .]=[\chi,.]$. We call $\tkg$ the \emph{bias extension} of $\kg$.   We extend the weight filtration and weight decomposition to $\tkg$ by setting : $\tkm^{(i)}=\km^{(i)}$ if $i \neq 2$, and $\tkm^{(2)}=\km^{(2)}+ \R\chi$, then $\tkg^{(i)}=\oplus_{j\geq i} \tkm^{(j)}$. We thus get the direct sum decomposition:
\begin{equation}\label{tilde-dec} \tkg = \km^{(1)} \oplus (\km^{(2)} \oplus \R\chi) \oplus \km^{(3)} \oplus \cdots \oplus \km^{(t)}.
\end{equation}

The large time behavior of a random walk on $\kg$ with increment average $\Xab$ is governed by a new \emph{graded Lie bracket} on $\kg$, or more generally on $\tkg$. It is defined by setting 
\begin{equation} \label{bracket-gr}
[x,y]'= \pi^{(i+j)}([x,y])
\end{equation}
if $x \in \tkm^{(i)}, y \in \tkm^{(j)}$ and $\pi^{(k)}: \tkg \to \tkm^{(k)}$ denotes the linear projection modulo the other $\tkm^{(k')}$, $k' \neq k$. 
\begin{lemme}
The bilinear map $[.,.]'$ is  a Lie bracket on $\tkg$.
\end{lemme}
\begin{proof}
Only the Jacobi relation requires some details. By \Cref{nestedg}, for all $x \in \tkm^{(i)}, y \in \tkm^{(j)}, z \in \tkm^{(k)}$, we have \begin{align*}[x, [y,z]']'&=\pi^{(i+j+k)}[x, [y,z]=\pi^{(i+j+k)}\left([[x,y],z]+[y,[x,z]]\right)\\ &=[[x,y]',z]'+[y,[x,z]']'.\end{align*} \end{proof}

Note that the vector space $\kg$ is an ideal of $(\tkg, [.,.]')$ and that $\Xm$ is central for $[.,.]'$. The  Lie product associated to $[.,.]'$ will be denoted by  $*'$. The operator $a_{\Xm}$ discussed in \cref{axdef} is just $[\chi, .]'$. 

Finally, we will make use of the \emph{dilations} defined in \cref{dilation}. They are a $1$-parameter  family of diagonal maps $(D_{r})_{r>0}$ on $\tkg$ defined by the formula $D_r(x)=\sum_{i}r^{i}x^{(i)}$. They are automorphisms of $(\tkg, [.,.]')$.

\subsubsection{Conventions: norm, Lebesgue measure,  dual, Fourier transform}\label{basis} We choose once and for all a basis $(e_j^{(i)})_{1\leq j \leq q_i}$ of each $\km^{(i)}$, and we add $e_0^{(2)}=\chi$ to $\tkm^{(2)}$. We denote by $\|.\|$ the associated Euclidean norm on $\kg$ and $\tkg$ making that basis orthonormal. The associated Lebesgue measure will be denoted by $dx$. Recall that $dx$ is also a left and right invariant Haar measure (for both $*$ and $*'$), see for instance \cite[1.2.10]{corwin-greenleaf90}.

We also fix a choice of left invariant Riemannian metric on $(\kg, *)$ and denote by $\norm{.}$ the induced distance to $0$.

We endow the dual space $\dkg$ of linear forms on $\kg$ with the operator norm and the Lebesgue measure $d\xi$ associated to the dual basis of the $(e_j^{(i)})_{i,j}$. Linear forms $\xi\in \dkg$ will be identified with their extension to $\tkg$ by setting $\xi(\chi)=0$. 

Given $\xi \in \dkg$, we set
$$e_{\xi} : \kg \rightarrow \C, x\mapsto e^{-2i\pi \xi(x)} $$
and we define the \emph{Fourier transform} of  a measure $\sigma$ or an integrable function $f$ on $\kg$ by 
$$\widehat{\sigma}(\xi)= \int_{\kg} e_{\xi}(x)d\sigma(x) \,\,\,\,\,\,\,\,\,\,\,\,\,\,\,\,\hf(\xi)= \int_{\kg} e_{-\xi}(x) f(x)dx.$$
If $f$ is continuous and $\hf$ is integrable, the Fourier inversion formula states that
\begin{align*}
\sigma (f)&=\int_{\dkg} \hf(\xi) \,\widehat{\sigma}(\xi) d\xi.
\end{align*}

\subsubsection{Polynomials and Lie product} \label{Sec-bi-grading}

Recall that there is a well-defined notion of polynomial function $f$ on a $d$-dimensional vector space $V$. Namely, it is independent of the choice of coordinates. So is its degree, which is defined as the maximal degree of a monomial $m:=x_1^{n_1}\ldots x_d^{n_d}$ appearing in $f$, where the degree of $m$ is $n_1+\ldots+n_d$.

On a vector space endowed with a filtration $V=V_1>V_2>\ldots >V_{s+1}=\{0\}$ of subspaces, there is another notion of degree of a polynomial function, which is tailored to the choice of filtration and we call \emph{w-degree}. It can be defined as follows. Choose supplementary subspaces $m_k$ so that $V_k=m_k\oplus V_{k+1}$ and a basis of each $m_k$ to form a full basis $(e_i)_{1\le i \le d}$ of $V$. We may assign the weight $w_i=k$ to each coordinate $x_i$ with $e_i \in m_k$. Now the w-degree of $m$ is set to be $w_1n_1+\ldots+w_dn_d$ and that of $f$ is defined as before.

It is straightforward to check that the w-degree $d_f$ depends only on the filtration and not on the choice of basis nor on the choice of $m_k$. To see this, one may for instance observe, for $f\neq 0$, that $\sup_{B_R} |f| \simeq c_f R^{d_f}$ for some $c_f>0$ as $R$ tends to infinity, where, for some Euclidean distance $d(\cdot,\cdot)$ on $V$ we define $$B_R:=\{x\in V\,:\,  \forall k \ge 1, \, d(x,V_{k+1})\leq R^k\}.$$

By definition, a \emph{polynomial map} $V \to \R^n$ is one whose coordinates are polynomial functions and its degree (or w-degree) if the maximum of the degree (or w-degree) of its coordinate functions. \emph{In this paper, we will apply this to $V=\kg$ or $\tkg$ endowed with the weight filtration defined in \Cref{weight-filtration}.}

\smallskip

We now pass to the description of the Lie product via the Baker-Campbell-Hausdorff formula. Write $\xx$ for the tuple $\xx=(x_1,\ldots,x_N)$. By virtue of this formula, in any nilpotent Lie algebra $\kh$, the Lie product map  
\begin{align}\Pi: \kh^N &\rightarrow \kh, \nonumber\\ \xx &\mapsto \Pi\xx:=x_{1}*x_{2}*\dots*x_{N}\label{prodmap}\end{align}
 is a polynomial map. 
To describe this map it is convenient to begin working formally in the free associative algebra on $N$ letters $\mathcal{A}_N:=\R\langle \oo_1,\ldots, \oo_N\rangle$. The algebra $\mathcal{A}_N$ admits a  Lie bracket given by $[\om,\om']=\om\om'-\om'\om$.  More generally, we may define the \emph{$r$-bracket} $L^{[r]}: \mathcal{A}^r_N\rightarrow \mathcal{A}_N$ recursively by $L^{[1]}=Id_{\mathcal{A}_{N}}$ and 
for all $\om_{1}, \dots, \om_{r} \in \mathcal{A}_N$, 
   $$L^{[r]}(\om_{1}, \dots, \om_{r})  =[L^{[r-1]}(\om_{1}, \dots,\om_{r-1}), \om_{r}].$$
To each monomial $m=a_1\cdot\ldots\cdot a_r$ in $\mathcal{A}_N$ (with $a_i \in \{\oo_1,\ldots,\oo_N\}$) we may assign the $r$-bracket $L(m):=L^{[r]}(a_1,\ldots,a_r)$ and extend linearly to obtain a well-defined map:
$$L:\mathcal{A}_N \to \mathcal{A}_N$$ whose image is a realization of the free Lie algebra on $N$ letters (see \cite [Theorem 0.5]{reutenauer93}). We note that if $\kh$ is a  Lie algebra and $\om\in L(\mathcal{A}_{N})$, then there is an \emph{evaluation map} $\kh^{N}\rightarrow \kh, \xx\mapsto \om\xx$ where $\om\xx$ is computed by writing $\om$ as a linear combination of iterated brackets in $\mathcal{A}_{N}$ then replacing each letter $\oo_{i}$ by $x_{i}$. By the universal property of free Lie algebras, the evaluation map is well-defined. Moreover if the monomials involved in $\om$ have length at most $s$, then $\om$ is characterized by the evaluation maps it induces on nilpotent Lie algebras of step (at most) $s$. For more details see \cite[Chapter 0]{reutenauer93}.

Fixing $s\geq 1$, the Baker-Campbell-Hausdorff formula (as proved by Dynkin in \cite{dynkin47}) guarantees the existence of a unique element $\Pi_{N}\in L(\mathcal{A}_{N})$ satisfying the following universal property: (1) for every $s$-step nilpotent Lie algebra $\kh$ and every $\xx=(x_{1}, \dots, x_{N})\in \kh^N$, one has $\Pi_{N}\xx=x_{1}*\dots *x_{N}$; (2) $\Pi_{N}$ only involves monomials of length at most $s$. 


We will use only few algebraic properties of the element $\Pi_N$. Among them is the following important feature formulated in the Claim below.

We use the standard shorthand notation $[N]$ for the set $\{1,\ldots,N\}$.  For a monomial $m \in \mathcal{A}_N$ we define its support $I_m \subseteq [N]$ to be the set of indices of the variables appearing in $m$. For $I\subseteq [N]$ we let $(\mathcal{A}_N)_I$ be the span of all monomials with support $I$, and  $\mathcal{A}_{N,t}$ the sum of all  $(\mathcal{A}_N)_I$ with $|I|=t$. For any $R \in \mathcal{A}_N$ we denote by $R_I$ its linear projection to $(\mathcal{A}_N)_I$ (modulo the other $(\mathcal{A}_N)_J$, $J \neq I$). Note that the projection $R\mapsto R_{I}$ commutes with $L$.  We may of course write:
$$\Pi_N =\sum_{t=1}^s \Pi_{N,t} \,\,\,\,\,\,\,\,\,\text{ where }\,\,\,\,\,\,\,\,\,\Pi_{N,t}:=\sum_{I \subseteq [N], |I|=t} (\Pi_{N})_I .$$ 
We note that $\mathcal{A}_N$ comes with a natural action of $\Sym(N)$ by permutation of the variables, i.e. $\sigma \oo_i=\oo_{\sigma i}$ for $\sigma \in \Sym(N)$. This action preserves $L(\mathcal{A}_N)$ and $\mathcal{A}_{N,t}$. For each $I\subseteq [N]$ with $|I|=t$, enumerate $I=\{i_1<\ldots<i_t\}$ and choose a permutation $\sigma_I$ of $[N]$ sending $j \in [t]$ to $i_j$.  
There is a natural \emph{periodization map} $\mathcal{P}_t:\mathcal{A}_{t,t} \to \mathcal{A}_N$ that sends $R$ to $\mathcal{P}_t R$:
$$\mathcal{P}_t R := \sum_{1\le i_1<\ldots<i_t \le N} R(\oo_{i_1},\ldots,\oo_{i_t})=\sum_{|I|=t} \sigma_I R.$$

\bigskip

\noindent {\bf Claim 1:} Let $t=1,\ldots,s$. Then $\Pi_{N,t}$ is in the image of $\mathcal{P}_t$. In fact:
$$\Pi_{N,t}=\mathcal{P}_t \Pi_{t,t}.$$

\begin{proof} The proof relies on the associativity of the Baker-Campbell-Hausdorff formula defining $\Pi_{N}$. 
In view of the decomposition  $\Pi_{N,t}:=\sum_{I \subseteq [N], |I|=t} (\Pi_{N})_I$ we only need to check that for $I=\{i_1<\ldots<i_t\}\subseteq [N]$, we have $(\Pi_{N})_{I}=\sigma_{I}\Pi_{t,t}$. Note that $(\Pi_{N})_{I}$ and $\sigma_{I}\Pi_{t,t}$ only involve brackets of length at most $s$, so we have to check that they induce the same evaluation map on every nilpotent Lie algebra $\kh$ of step at most $s$. Let $\xx=(x_1,\ldots,x_N)\in \kh^N$, and $\yy\in \kh^N$ obtained from $\xx$ by  setting $y_{i}=x_{i}$ if $i\in I$ and  $y_i=0$ if $i \notin I$.  We have $(\Pi_N)_{I}\xx= \Pi_N\yy=y_{i_1}*\ldots*y_{i_t}=(\sigma_I \Pi_{t,t}) \xx$, whence the result.
\end{proof}

Therefore, setting $P_t:=\Pi_{t,t}$,
\begin{equation} \Pi_N= \sum_{t=1}^s \sum_{|I|=t} \sigma_I P_t.
\end{equation}
For example, $\Pi_1=P_1=\oo_1$, $\Pi_2=\oo_1+\oo_2+P_2$, $\Pi_3=\oo_1+\oo_2+\oo_3+P_2(\oo_1,\oo_2)+P_2(\oo_1,\oo_3)+P_2(\oo_2,\oo_3)+P_3$. Also note that $P_2^{[2]}=\frac{1}{2}L(\oo_1\oo_2)$, $P_3^{[3]}=\frac{1}{6}L(\oo_1\oo_2\oo_3)+\frac{1}{6}L(\oo_3\oo_2\oo_1)$. 

\bigskip

In the proof of Theorem \ref{LLT} and the path-swapping method, it will be helpful to see $\mathcal{A}_N$ as a module for the group algebra of $\Sym(N)$. In particular, if $R \in \mathcal{A}_N$ and 
$$a=\sum_{\sigma \in \Sym(N)} a_\sigma \sigma$$
is an arbitrary element of the group algebra  $\R[\Sym(N)]$, we will write:
\begin{equation} \label{module} aR:= \sum_{\sigma \in \Sym(N)} a_\sigma \sigma R.\end{equation}
Given a Lie algebra $\kh$, the space of functions $\{P:\kh^N\rightarrow \kh\}$ is also as $\R[\Sym(N)]$-module, where the action is given by $$a P= \sum_{\sigma \in \Sym(N)} a_\sigma P\circ \sigma^{-1}.$$The two actions are compatible in the sense that $\xx\mapsto (aR)\xx$ and $a(\xx\mapsto R\xx)$ obviously coincide for $R \in L(\mathcal{A}_N)$. 
\bigskip

We end this subsection by introducing another piece of notation regarding the product $\Pi_N$. Recall that the free associative algebra $\mathcal{A}_N$ has a natural grading by monomial length. Namely given $R \in \mathcal{A}_N$ we may write
$$R=\sum_{r\ge 1} R^{[r]}$$
where $R^{[r]}$ is the part of $R$ which is a linear combination of monomials of length $r$. The projection $R\rightarrow R^{[r]}$ preserves $L(\mathcal{A}_{N})$. The polynomials $\Pi_N^{[a]}$, $a=1,\ldots,s$ will play a role below. We note that for any nilpotent Lie algebra $\kh$ of step at most $s$, the evaluation map $\kh^N\rightarrow \kh, \xx \mapsto \Pi_N^{[a]}\xx$ defines  a polynomial map whose coordinates are homogeneous polynomials of (ordinary) degree $a$ on the vector space $\kh^N$. In fact it is the $a$-homogeneous component of $\xx\mapsto \Pi_N\xx$.

Further to this, we now introduce a refined decomposition that takes into account the weight filtration and the w-degree. Suppose that the nilpotent Lie algebra $\kh$ comes with a filtration by ideals $(\kh^{(i)})_{i\leq t}$ and that supplementary subspaces $\kn^{(i)}$ are chosen so that $\kh^{(i)}=\kn^{(i)}\oplus\kh^{(i+1)}$. This is the situation encountered, for example, in \Cref{weight-filtration}. An element $x\in \kh$ can be written as a sum of its constituents $x^{(i)}\in \kn^{(i)}$ each of which will have weight $i$.  We may thus speak of the w-degree of a polynomial function $f$ on $\kh$ or $\kh^N$. And any such function can then be decomposed into a sum of homogeneous components for the w-degree 
\begin{align}\label{homodecompf}
f=\sum_{b\ge 0} f^{(b)}.
\end{align}

We denote by $\xx \mapsto \Pi^{(b)}\xx$ the $b$-homogeneous component of $\xx \mapsto \Pi\xx$ on $\kh^N$ for the w-degree. And for $b\ge a$, we write $\xx \mapsto \Pi^{[a,b)}\xx$ for the component of $\xx\mapsto (\Pi\xx)^{[a]}$ that is $b$-homogeneous for the w-degree, or equivalently the component of $\xx \mapsto \Pi^{(b)}\xx$ that is $a$-homogeneous for the ordinary degree. This produces a bi-grading of the ring of polynomial functions on $\kh^N$. We note that this new grading depends not only on the filtration $\kh^{(i)}$ but also on the choice of $\kn^{(i)}$. In particular, contrary to $\Pi^{[a]}$ which corresponds to an element in the associative algebra, the notation $\Pi^{[a,b)}\xx$ only refers to a map on $\kh^N$. For an alternative equivalent description of $\Pi^{[a,b)}\xx$ in terms of a free Lie algebra element, we refer the reader to our companion paper \cite{benard-breuillard-CLT}.

In  the remainder of this article, we will apply these considerations in the case where $\kh=\kg$ or $\kh=\tkg$ endowed with the weight filtration. We further note in this context that $\Pi^{[a,b)}\xx$ behaves in a homogeneous way with respect to the dilations introduced in \Cref{weight-filtration}. Namely $\Pi^{[a,b)} D_{r}\xx=r^b \Pi^{[a,b)}\xx$ for any $r>0$.

\subsubsection{The driving measure and its bias extension} \label{Sec-cadre-mu}

Let $\mu$ be a probability measure on $\kg$. Denote by $\mu_{ab}$ the projection of $\mu$ to the abelianization $\kg/[\kg, \kg]=:\kg_{ab}$. The subscript $ab$ stands for ``abelianized''. We assume that $\mu_{ab}$ has a finite first moment and write $\overline{X}=\E(\mu_{ab})$ for the expectation of $\mu_{ab}$. We also set $\mu^{(i)}:=\pi^{(i)}_{\star}\mu$, where the subscript $\star$ added to a map denotes the push-forward by that map. We set $\Xm= \E(\mu^{(1)})$ and notice $\Xm$ projects onto $\overline{X}$ under the abelianization map.

To prove the LLT, we will replace the measure $\mu$ by its \emph{bias extension} $\tmu$, defined as the image of $\mu$ under the map $\kg\rightarrow \tkg, x \mapsto (x,1)$. Note that for all $N\geq 1$, one has 
\begin{align} \label{formule-ext}
\tmu^{*N} = \mu^{*N}-N\Xm+N\chi =\mu^{*N}*-N\Xm*N\chi  
\end{align}
where the notation $ \mu^{*N}-N\Xm+N\chi $ stands for the image of $\mu^{*N}$ under the map $\tkg\rightarrow \tkg, x\mapsto x-N\Xm+N\chi$, and similarly for $\mu^{*N}*-N\Xm*N\chi$. The role of $\tmu$ is to convert the drift $\Xm \in \km^{(1)}$ into  $\chi\in \tkm^{(2)}$ while  keeping  the fluctuations around the drift in $\R\Xm \subseteq \km^{(1)}$. This separation aims to reflect the fact that the drift will contribute to  higher commutators with weight $2$ whereas the fluctuation around the mean will contribute  with weight $1$.



In proving the local limit theorem, we will  suppose that $\mu$ is \emph{aperiodic} (or \emph{non-lattice}). This means that the support of $\mu$ is not contained in a coset of some proper subgroup, i.e. some $x*H$ where $x\in \kg$, $H$ is a proper subgroup of $(\kg, *)$. For simply connected nilpotent Lie groups, it is equivalent to ask that $\mu_{ab}$ is non-lattice. This is because, as is well-known and easy to check by induction on the nilpotency class, a closed subgroup of $(\kg,*)$ is proper if and only if its image in $\kg_{ab}$ is proper.  Moreover the non-lattice condition on $\mu_{ab}$ is equivalent to requiring that $|\widehat{\mu_{ab}}(\xi)|<1$ for all $\xi \in \widehat{\kg}_{ab}\setminus\{0\}$ (indeed the latter fails precisely when $\mu_{ab}$ is supported on a coset of $\xi^{-1}(\Z)$).

We also need moment conditions.  For $m>0$, we say that $\mu$ has \emph{finite $m$-th moment} for the weight filtration if for all $b\leq 2s-1$, one has $$\int_{\kg} \|x^{(b)}\|^{m/b}d\mu(x)<\infty. $$
N.B. Unless explicitely stated otherwise, moment conditions on $\mu$ will always refer to moment for the weight filtration.

\subsubsection{Moderate deviations} \label{Sec-cadre-deviations}

Given $N\geq 1$, $g,h\in \tkg$, we denote by $g*\tmu^{*N}*h$, or $\tmu^{*N}_{g,h}$ for short, the image of $\tmu^{*N}$ under the map $\tkg\rightarrow \tkg, x\mapsto g*x*h$. Setting $p: \tkg:=\kg\oplus \R \rightarrow \kg$ the projection to the first factor (which is a Lie algebra homomorphism), we have $p_{\star}\tmu^{*N}_{g,h}=\mu^{*N}_{p(g),p(h)}$.

We shall prove  estimates for $\tmu^{*N}_{g,h}$ uniformly in a certain range of \emph{moderate deviations}: given  $\delta_{0} \in [0, 1)$, we set
$$\DN(\delta_{0})=\{x \in \tkg  \,:\, \forall i\leq 2s-1,\,\, \|x^{(i)}\| \leq  N^{i/2+s^{-1}\delta_{0}}\},$$ where $x^{(i)}=\pi^{(i)}(x)$ is the coordinate projection with respect to the grading $(\tkm^{(b)})_{b}$ of $\tkg$. It is worth pointing out that $\DN$ behaves well with respect to the product $*$ :  There exists $C>0$ such that for all $\delta_{0}>0$, for all large enough $N \geq 1$,  
 $$x,y\in \DN(\delta_{0}) \implies x*y\in \DN(C \delta_{0}).$$
 Indeed, this follows from the Baker-Campbell-Hausdorff formula: the coordinates of $(x*y)^{(i)}$ are polynomials in $(x,y)$ of w-degree  at most $i$.

\subsubsection{Asymptotic notation} \label{Sec-cadre-asympt}
In the rest of the text, we use the standard Vinogradov notation  $\ll$ and Landau notations $O(.)$, $o(.)$. Unless otherwise stated, the implied constants will depend only on the initial data $\kg, \XXab, \mu, (\km^{(b)})_{b}, (e^{(i)}_{j})_{i,j}, \norm{.}$.  In some statements, we will also use the notation $r_{1} \lll r_{2}$ to mean that the conclusions that follow are true up to choosing $r_{1}\leq c r_{2}$  where $c>0$ is a (small) constant depending only the initial data. Subscripts are added to indicate any additional dependency (for instance, on the parameter $\delta_{0}$ controlling the range of deviations). Furthermore $\N$ denotes the non-negative integers and $\R^+$ the non-negative reals, and for $n \in \N$, $[n]$ denotes the set $\{1,\ldots,n\}$.




\subsection{The three reductions} \label{3reductions} \label{Sec-3reductions}

To prove the local limit theorem \Cref{LLT} announced earlier, we perform a series of reductions on the random product  $X_{1}*\dots *X_{N}$ that help estimate its Fourier transform. First, we see that after a suitable truncation, the random product has  moment estimates similar to what can be proven in the case where $\mu$ has finite moment for every order. Second, we  note that the product structure $*$ can be replaced, up to a small error, by its graded counterpart $*'$ for which dilations act as automorphisms. Third, we show that $\mu$ can be replaced by any other driving measure  that share the same moments up to order $2$. Those reductions  were already used in \cite{benard-breuillard-CLT} in order to prove the central limit theorem.


The notations $(\kg, *, \tkg, \Xab, (\km^{(b)})_{b\leq 2s-1}, \Pi^{[a,b)},  *', \mu, \tmu, \DN(\delta_{0}))$ are those of  \Cref{Sec-cadre}.  Throughout \Cref{Sec-3reductions} we assume that \emph{$\mu$ has finite second moment for the weight filtration} induced by $\Xab$.

\subsubsection{Truncation and moment estimates} \label{Sec-3reductions-tronc}

We will need the moment estimates collected in \Cref{momentPi} below. They concern a truncation of the driving measure in order to allow weaker moment assumptions on $\mu$. If $\mu$ has finite moment for every order, such truncation is not necessary (see \cite[Proposition 4.1]{benard-breuillard-CLT}). As mentioned earlier, we will use the bias extension $\tmu$ instead of $\mu$ for technical convenience.



The \emph{truncation} $T_{N}\tmu$  is the image of $\tmu$ under the map 
\begin{align*}T_{N}: \tkg & \rightarrow \tkg,\\ x=\oplus x^{(b)} & \mapsto \oplus T_{N}x^{(b)}\end{align*} where 
 $$
T_{N}x^{(b)} = \left\{
    \begin{array}{ll}
       x^{(b)}  & \mbox{if } \|x^{(b)} \|\leq N^{b/2} \\
        0 & \mbox{if } b\geq 2 \text{ and } \|x^{(b)} \|> N^{b/2}\\ 
        c_{N} & \mbox{if }  b= 1 \text{ and } \|x^{(1)} \|> N^{1/2}
    \end{array}
\right.
$$
and $c_{N}:= - \tmu({\|x^{(1)} \|> N^{1/2}})^{-1}\E_{\tmu}(x^{(1)} \1_{{ \|x^{(1)} \|\leq N^{1/2}}})$ is chosen so that $T_{N}\tmu$ is centered in the $\km^{(1)}$-coordinate. 

The   error induced when replacing  the measure $\tmu$ by its truncation $T_{N} \tmu$ is controlled via the following lemma.

\begin{lemme}[Cost of truncation] \label{truncation-cost}
Let $q \in [2, +\infty)$. If $\mu$ has finite $q$-th moment for the weight filtration, then for all $N\geq 1$, $(N_{i})_{i\leq N}\in \R_{>0}^N$,   
$$\| \tmu^{\otimes N}-  \bigotimes_{i=1}^N T_{N_{i}}\tmu \| =o\left(\sum_{i=1}^N N_{i}^{-\frac{q}{2}}\right) $$
where $\|.\|$ is the total variation norm for measures on $\tkg^N$. 
\end{lemme}

\begin{proof} Note that $$\tmu \{T_{N_{i}}x\neq x\} \leq \tmu \{\max_{b} \|x^{(b)}\|^{1/b}>N_{i}^{1/2}\} =o(N_{i}^{-\frac{q}{2}})$$
so $\bigotimes_{i=1}^N T_{N_{i}}$ is the identity on a subset of $\tkg^N$ of $\tmu^{\otimes N}$-measure $(1-o(\sum_{i=1}^N N_{i}^{-\frac{q}{2}}))$.
\end{proof}

Given a monomial function $M:\tkg^{t}\rightarrow \R$ of w-degree $b\ge 1$, see \Cref{Sec-bi-grading}, we define its statistics $M^{\infty}$  on $\tkg^{\N^*}$  as the formal series
$$M^{\infty}(\xx)=\sum_{1\leq n_{1}<\dots <n_{t}<\infty}M(x_{n_{1}}, \dots, x_{n_{t}}).$$
For every $N\geq1$, we identify $\tkg^N$ as the subset of $\tkg^{\N^*}$ made of sequences that vanish after the  first $N$ terms.  In particular, $M^{\infty}$ defines a map on $\tkg^N$.

\begin{proposition}[Moment estimate \cite{benard-breuillard-CLT}] \label{momentPi} Assume that $\mu$ has a finite second moment for the weight filtration. For  $N\geq 1$, $(N_{i})_{i\leq N}\in  [0, N]^N$,  $m\geq 1$ an integer,
$$\E_{\bigotimes_{i=1}^N T_{N_{i}}\tmu}\left[{M^{\infty}}(\xx)^{2m} \right]\ll_{m} N^{bm},$$ and in particular for $\delta_{0}\geq 0$, $g,h\in \DN(\delta_{0})$, $\yy=(g,\xx,h)$, $a,b\geq 1$,
$$\E_{\xx\sim \bigotimes_{i=1}^N T_{N_{i}}\tmu }\left[\|\Pi^{[a,b)}\yy\|^{2m} \right]\ll_{m} N^{(b+2\delta_{0})m }.$$
\end{proposition}

In this last inequality, the symbol $\xx\sim \bigotimes_{i=1}^N T_{N_{i}}\tmu$ means that $\xx$ represents a random variable with law $\bigotimes_{i=1}^N T_{N_{i}}\tmu$ on $\tkg^N$. 

\begin{proof} This is \cite[Proposition 4.3]{benard-breuillard-CLT}.
 
 \end{proof}

\subsubsection{Graded replacement}

 We bound the error induced when replacing $*$ by the graded product $*'$ for which the dilations $(D_{r})_{r>0}$ act as automorphisms.  We denote by $\Pi'$ the $N$-fold product for $*'$. The result is formulated  in terms of Fourier transforms. 

\begin{proposition}[\cite{benard-breuillard-CLT}]
\label{grading-Fourier}
 For  a linear form $\xi$ on $\tkg$, $N\geq 1$, $(N_{i})_{i\leq N}\in [ 0, N]^N$, $g,h\in \DN(\delta_{0})$, $\yy=(g,\xx,h)$, 
$$\big|\E_{\xx \sim \bigotimes_{i=1}^N T_{N_{i}}\tmu } [e_{\xi} (\Pi\yy )-e_{\xi} (\Pi'\yy )] \big| \ll \,\| \DilN \xi \| \,N^{-1/2+\delta_{0}}.$$
\end{proposition}

\begin{proof} This is \cite[Corollary 4.5]{benard-breuillard-CLT}.

\end{proof}

\subsubsection{Lindeberg replacement}

We bound the error induced when replacing the driving measure with another measure that \emph{coincides with it up to order $2$}. This means that the two measures have the same evaluation on all polynomials of w-degree $1$ or $2$. The result is formulated  in terms of Fourier transforms.

\begin{proposition}[ \cite{benard-breuillard-CLT}] \label{gaussian}
Let $\eta$ be a probability measure on $\kg$ with finite second moment  for the weight filtration and that coincides with $\mu$ up to order $2$. For a linear form $\xi$ on $\tkg$, $N\geq 1$, $(N_{i})_{i\leq N}\in [ 0, N]^N$, $g,h\in \DN(\delta_{0})$, $\yy=(g,\xx,h)$, 
$$\big|\E_{\xx \sim \bigotimes_{i=1}^N T_{N_{i}}\tmu } [e_{\xi} (\Pi'\yy )]-\E_{\xx \sim \bigotimes_{i=1}^N T_{N_{i}}\teta } [e_{\xi} (\Pi'\yy )]\big| \ll_{\eta} \,(\| \DilN \xi\|+\| \DilN \xi\|^3)N^{3 \delta_{0}}\sum_{i=1}^N N^{-1}_{i} \,\eps_{N}$$
where $\eps_{N}=o_{\eta}(1)$. More generally, assuming $\mu$, $\eta$ both have finite $m$-th moment for some $m\in [2,3]$, we may take $\eps_{N}=o_{\eta}(N^{-\frac{m}{2}+1})$ if $m<3$ and $\eps_{N}=O_{\eta}(N^{-1/2})$ if $m=3$.
\end{proposition}

We will apply this proposition to a sequence  $(N_{i})$ that is gradually increasing and such that $\sum_{i=1}^N N^{-1}_{i}=O(1)$, so the right-hand side will be $O_{\eta}(1)(\| \DilN \xi\|+\| \DilN \xi\|^3)N^{3s \delta_{0}}\,\eps_{N}$.

\begin{proof} 
This is \cite[Proposition 4.7]{benard-breuillard-CLT}. 
\end{proof}


\section{Uniform local limit theorem} \label{Sec-uniform-LLT} 

This section is dedicated to the proof of the uniform local limit theorem, namely \Cref{LLT}.    The argument is based on the Fourier inversion formula and the analysis of characteristic functions. It exploits the reduction techniques recalled in \Cref{Sec-3reductions} but the treatment requires a new and crucial step: the domain reduction. 

We use the notations $(\kg, *, \tkg, \Xab, \dd, (\km^{(b)})_{b\leq 2s-1}, \Xm, \chi, *')$ from   \Cref{Sec-cadre}. We let $\mu$ be a probability measure on $\kg$, with finite moment of order $m_{\mu}>2$ for the weight filtration  defined by $\Xab$ and such that $ \mu_{ab}$ has  expectation $\Xab$ and is non-lattice. In particular $\E(\mu^{(1)})=\Xm$. We will  use the bias extension $\tmu=\mu-\Xm+\chi$  on $\tkg$, the notations for deviations  $\tmu^{*N}_{g,h}$, $\DN(\delta_{0})$ from \Cref{Sec-cadre-deviations} where $N\geq 1$, $g,h\in \tkg$, as well as the truncation $T_{N}\tmu$  introduced in \Cref{Sec-3reductions-tronc}. 

We will deduce \Cref{LLT} from a more general theorem, \Cref{LLT-tronc} below, that deals with truncated random products and works under a weaker moment assumption. \Cref{LLT} will then follow  from \Cref{LLT-tronc} by applying \Cref{truncation-cost}.  \Cref{LLT-tronc}  will be useful in the applications given later in \Cref{Sec-applications}. In order to state \Cref{LLT-tronc}, we need to set up some additional notation.

\bigskip

\noindent {\bf Claim:} The map $p:\tkg\rightarrow \kg, x\oplus t\chi \mapsto x+t\Xm$ is a group homomorphism for $*$. 

\begin{proof}Indeed, one may rewrite $p$ as the first factor projection map from $\tkg=\kg\oplus \R$ to $\kg$. It is a homomorphism because $*$ was extended from $\kg$ to $\tkg$ by asking that the factors $\kg$ and $\R$ commute.
\end{proof}

 We fix a (small) parameter $\gamma_{0} \in (0,1)$ to be specified below and set for $a\in \llbracket 1, s\rrbracket$,
$$\gamma_{a}= (32s)^{-a}\gamma_{0}\1_{a<s} \,\,\, \,\,\,  \,\,\, \,\,\, \,\,\,\,\,\,N_{a}=\lfloor N^{1-\gamma_{a}}\rfloor   \,\,\, \,\,\, \,\,\,\,\,\, \,\,\, \,\,\,\kappa_{N,a}=(T_{N_{a}} \tmu)^{\otimes N_{a}-N_{a-1}}$$
with the convention $N_{0}=0$.
We then define 
\begin{align*}  
&\Upsilon^N=\Pi_{\star}(\kappa_{N,1}\otimes \dots \otimes \kappa_{N,s})\,\,\,\,\,\,\,\,\,\,\,\,\,\,\,\,\,\,\,\,\,\,\,&&\Theta^N=p_{\star}(\Upsilon^N)\\
& \Upsilon^N_{g,h}=g*\Upsilon^N*h   \,\,\,\,\, \,\,\,\,\,\,\,\,\,\,\,\,\,\,\,\,\,&&\thetaNgh=p_{\star}(\Upsilon^N_{g,h}).
\end{align*}
So $\Upsilon^N$ is a measure on $\tkg$, which is the image under the Lie product of the $s$ measures $\kappa_{N,a}$, each one of which is a product of several copies of the truncated measure $T_{N_{a}} \tmu$. And $\Theta^N$ is the projection of $\Upsilon^N$ onto $\kg$. Note that if $\mu$ has compact support, then $\Upsilon^N=\tmu^{*N}$ and $\Theta^N=\mu^{*N}$ for all large enough $N$.  We say that $\Theta^N$ is  the  \emph{gradual truncated approximation of $\mu^{*N}$}. This truncation will play a role later in the proof of \Cref{reduction}. 

We are now ready to state \Cref{LLT-tronc}. Recall the standing assumption on $\mu$ is that it has a finite moment of order $m_\mu>2$ for the weight filtration (associated to $\Xab$). Note that this is weaker than the assumption made in \Cref{LLT}. 



\begin{theorem}[Local limit theorem for truncations] \label{LLT-tronc} 
Keep the  above notations.
Then for every $f \in C_{c}(\kg)$, for $0< \gamma_{0} \lll1$,  $0<\delta_{0}\lll_{\gamma_{0}}1$,   $N\geq 1$, $g,h\in \tkg$ one of which lies in $\DN(\delta_{0})$, we have 
$$|\thetaNgh (f) - (\DilN \nu * N\Xm )_{p(g),p(h)}(f)| =o_{f, \gamma_{0}}(N^{-\frac{\dd}{2}}).$$
\end{theorem}


 \Cref{LLT-tronc} does imply \Cref{LLT}:
\begin{proof}[Proof of \Cref{LLT}]
By \Cref{truncation-cost}, one has in total varation norm
$$\|\Theta^{N}_{p(g),p(h)}- \mu^{*N}_{p(g),p(h)}\| =  o(N^{1-\frac{1-\gamma_{0}}{2}m_{\mu} })$$
which is negligible compared to $N^{-\dd/2}$ as long as $\gamma_{0}$ is chosen small enough depending on $m_{\mu}>\dd+2$. Hence the LLT for $\Theta^{N}_{p(g),p(h)}$, namely \Cref{LLT-tronc}, implies the LLT for $\mu^{*N}_{p(g),p(h)}$, namely \Cref{LLT}.
\end{proof}


\bigskip
We now give an {\bf overview of the proof of \Cref{LLT-tronc}}. Let $f :\kg\rightarrow \R$ be a continuous integrable test function. We aim to estimate $\thetaNgh(f)$ for large $N$. If  $\hf$ is integrable, the Fourier inversion formula allows to write, for all $N\geq 1$, $g,h\in \tkg$, 
\begin{align*}
\thetaNgh (f)&=\int_{\dkg} \hf(\xi) \,\widehat{\thetaNgh}(\xi) d\xi \nonumber \\
&=\int_{\dkg} \hf(\xi) \,\huNgh(\xi) e^{-2i\pi \xi\circ p(g^{(\chi)}+N\chi+h^{(\chi)})}d\xi 
\end{align*} 
 where we identify $\xi$ with its extension to $\tkg$ defined by setting $\xi (\chi)=0$, and where $g^{(\chi)}, h^{(\chi)}$ denote the $\R\chi$ components  of $g,h$ in the decomposition \eqref{tilde-dec}.

We will then use the aperiodicity (non-lattice) assumption  on $\mu$ together with a combinatorial averaging argument (the \emph{path-swap}) inspired by that of Diaconis-Hough \cite{diaconis-hough21} to restrict the integration domain  to a small neighborhood of the origin, namely to the \emph{reduced domain} 
$$\UN(\gamma_{0}):=\{\xi \in \dkg  \,:\, \forall b\leq 2s-1,\,\,\|\xi_{|\kg^{(b)}}\| \leq  N^{-b/2+\gamma_{0}} \}.$$
The error term will be estimated in terms of the following  \emph{gap function} of $\hmu_{ab}$, which quantifies the level of aperiodicity of $\mu$. Given $c>0$, let $gap_{c}: \R_{>0}\rightarrow (0, c)$ be the function given by 
$$gap_{c}(R)= c\inf \{1-|\hmu_{ab}(\xi)| \,:\, \xi \in \widehat{\km^{(1)}}, \, c\leq \|\xi \|\leq c^{-1}(1+R)\}$$
where $\mu_{ab}$ is identified to a measure on $\km^{(1)}$ via the projection map. 

\begin{restatable*}[Domain reduction]{proposition}{domainreduction} \label{reduction}
For $\gamma_{0}\in (0,1)$, $Q>1$,  for some $\delta_{0}=\delta_{0}(\dim\kg, \gamma_{0})>0$, $L=L(\dim\kg, \gamma_{0}, Q)> 1$,  for $c\lll1$,  $N\geq 1$, $g,h\in \DN(\delta_{0})$, and $\xi \in \dkg \smallsetminus \UN(\gamma_{0})$, 
$$ |\huNgh(\xi)| \ll_{\gamma_{0}, Q} (1+\|\xi\|)^L gap_{c}(\|\xi\|)^{-L} N^{-Q}.$$
\end{restatable*}

This allows us to write:
\begin{align*}
\thetaNgh  (f) = &\int_{\UN(\gamma_{0})} \hf(\xi) \,\huNgh(\xi) e^{-2i\pi \xi\circ p(g^{(\chi)}+N\chi+h^{(\chi)})}d\xi  \,+\,\|f\|_{\mu_{ab}, c, L}O_{\gamma_{0},Q}( N^{-Q}) \end{align*}
where we set 
\begin{align}\label{norm-mucL}
\|f\|_{\mu_{ab}, c, L}:= \|f\|_{L^1}+ \int_{\dkg} |\hf(\xi)| (1+\|\xi\|)^L gap_{c}(\|\xi\|)^{-L} d\xi   
\end{align}
Note  that the additional $ \|f\|_{L^1}$ is not necessary for the estimate to hold here, but it will play a role  in \Cref{TLLFouriercompact} below. 

In the above integral on the reduced domain $\UN(\gamma_{0})$, we may  replace $*$ by its graded counterpart $*'$ (\Cref{grading-Fourier}) up to an error $O(N^{-1/2+2\gamma_{0}})$, then replace  $\mu$ by any order $2$ approximation $\eta$ with finite moments of all orders (\Cref{gaussian}) up to an error $O_{\eta}(N^{-\frac{1}{2}\min(1,m_{\mu}-2)+ 6\gamma_{0}})$. The total error, once integrated on $\UN(\gamma_{0})$, will be of order $O_{\eta, \gamma_{0}}\left(N^{ -\frac{1}{2}(\dd +\min(1, m_{\mu}- 2))+ C\gamma_{0} } \right)$ for some constant $C>1$ depending only on $\dim \kg$.
We will then choose   $\eta$ so that  $\teta^{*'N}=D_{\sqrt{N}}\nu*'N\chi$ for all $N$,  and infer the following local limit theorem with power saving and moderate bideviations for the test function $f$. 

\begin{restatable*}[Effective LLT with power saving for truncations]{proposition}{TLLFouriercompact} \label{TLLFouriercompact}
For  every continuous integrable function $f$ on $\kg$, for $\gamma_{0}\in (0, 1)$,  $c\lll1$, some $C=C(\dim \kg), L=L(\dim \kg, \gamma_{0})> 1$, for $N\geq 1$, $g,h\in \DN(1/L)$,    we have
$$|\thetaNgh  (f) - (\DilN \nu * N\Xm )_{p(g),p(h)}(f) | =  \|f\|_{\mu_{ab}, c, L} O_{\gamma_{0}}\left(N^{ -\frac{1}{2}(\dd +\min(1, m_{\mu}- 2))+C\gamma_{0}} \right). $$
\end{restatable*}

Notice \Cref{TLLFouriercompact} implies \Cref{eff-LLT}

\begin{proof}[Proof of \Cref{eff-LLT}] By \Cref{truncation-cost}, one has in total varation norm
$$\|\Theta^{N}_{p(g),p(h)}- \mu^{*N}_{p(g),p(h)}\| =  o(N^{1-\frac{1-\gamma_{0}}{2}m_{\mu} })$$
which is negligible compared to $N^{-(\dd+1-\eps)/2}$ as long as $\gamma_{0}$ is chosen small enough depending on $(\dim \kg, \eps)$. We also make sure  that $\gamma_{0}$ is small enough so that $C\gamma_{0}<\eps$ where $C$ is the constant in \Cref{TLLFouriercompact}. As $\|f\|_{L^\infty}\leq \|f\|_{\mu_{ab}, c,L}$ by Fourier inversion,   \Cref{eff-LLT} now follows from \Cref{TLLFouriercompact}.
\end{proof}

Next we will deduce an analogous result for continuous functions with compact support.

\begin{restatable*}[LLT with moderate bideviations]{proposition}{TLLfcompact}\label{TLLfcompact}
For every continuous function with compact support $f\in C_{c}(\kg)$,   $0<\gamma_{0} \lll 1$, some $\delta_{0}=\delta_{0}(\dim \kg, \gamma_{0})>0$, for $N\geq 1$, $g,h\in \DN(\delta_{0})$,  
$$|\thetaNgh   (f) - (\DilN \nu * N\Xm )_{p(g),p(h)}(f) | =   o_{f, \gamma_{0}}(N^{-\frac{\dd}{2}}).$$
\end{restatable*}

The proof of \Cref{TLLfcompact} relies on $L^1$-approximation of the test function by functions with compactly supported Fourier transform, to which we may apply  \Cref{TLLFouriercompact}. 



Finally the proof of \Cref{LLT-tronc} will be complete if we manage to allow one of the deviations to be arbitrary. For this it is sufficient to check that  the terms $\Theta^N_{p(g),p(h)}(f)$, $(\DilN \nu * N\Xm )_{p(g),p(h)}(f)$ will be negligible compared to $N^{-\frac{\dd}{2}}$. This is the role of the next lemma.

\begin{restatable*}{lemme}{toutedeviation}  \label{toutedeviation}
For every continuous function with compact support  $f\in C_{c}(\kg)$, for  $\gamma_{0}\in (0,1)$, for $\delta_{0} \ggg \delta_{1}>0$,  $N\geq 1$, $g,h \in \tkg$ with $p(g)\notin p(\DN(\delta_{0}))$, $h\in \DN(\delta_{1})$
$$|\Theta^N_{p(g),p(h)}  (f)|=   o_{f, \delta_{0}, \delta_{1}}(N^{-\frac{\dd}{2}})  \,\, \, \,\, \, \,\, \, \,\, \, \,\, \, |(\DilN \nu * N\Xm )_{p(g),p(h)}(f) | =   o_{f, \delta_{0}, \delta_{1}}(N^{-\frac{\dd}{2}}) $$
and this also holds if we switch the conditions on $g,h$. 
\end{restatable*}

This concludes the (overview of the) proof of \Cref{LLT-tronc}.

\subsection{Path swap} \label{Sec-path swap}

The main part of the proof of \Cref{LLT-tronc} is the domain reduction, namely  \Cref{reduction}. In this subsection, we establish \Cref{choisirk}, which gives an exponential bound on the Fourier transform of a random product restricted to its components involving only brackets of length bounded by an integer $a$. \Cref{choisirk}  will be an essential ingredient in the proof of \Cref{reduction}. 

\Cref{choisirk} relies on a path swap technique inspired by the work of  Diaconis and Hough in \cite{diaconis-hough21, hough19}.  The idea is to take advantage of the fact that the law of the product $X_1*\ldots*X_N$ is invariant under permutations of the $X_i$ by averaging over a certain subset of permutations in $Sym(N)$, in order to generate smoothness in the higher commutators. This is akin to a kind of \emph{discrete derivative} and its use in combination to the Cauchy-Schwarz inequality as effected below can be viewed as a kind of non-abelian analogue (for commutator calculus) of the van der Corput trick, which is ubiquitous in nilpotent dynamics \cite{leibman, host-kra, green-tao}. In what follows the average is taken \emph{over a certain elementary abelian $2$-subgroup of $Sym(N)$}.  In this subsection, as well as in  \cite{diaconis-hough21, hough19}, the path swap technique is tailored to deal with the central filtration of $\kg$. We will need extra work to obtain information regarding the weight filtration instead. 

\bigskip

We now present the path-swap technique.
Let us consider a subgroup $F\leq \Sym(N)$, which is a cartesian product $F=F_1\times \ldots \times F_{a-1}$ of commuting subgroups $F_i\leq \Sym(N)$. An element $\sigma\in F$ may be written uniquely $\sigma=\sigma_1\ldots\sigma_{a-1}$ with $\sigma_i \in F_i$. 


Recall from \Cref{Sec-bi-grading} that the set of maps from $\tkg^N \rightarrow \tkg$ is a $\R[\Sym(N)]$-module. Consider a measurable map $P: \tkg^N \rightarrow \tkg, \xx \mapsto P \xx$ and a probability measure  $\kappa$ on $\tkg^N$ such that the distributions $P_{\star}\kappa$ and $(\sigma P)_{\star}\kappa$ coincide for all $\sigma\in F$ (for instance $\kappa$ could be $\Sym(N)$-invariant). We write:
\begin{align*}
\widehat{P_{\star}\kappa}(\xi)&= \int_{\tkg^N} e_{\xi}(P\xx) \,\,d\kappa(\xx)= \int_{\tkg^N} \E_{\sigma \in F}e_{\xi}(\sigma P\xx) \,\,d\kappa(\xx)\\
&=\int_{\tkg^N}\mathbb{E}_{\sigma_{a-1}\in F_{a-1}}\dots \mathbb{E}_{\sigma_{1}\in F_1} e_{\xi}(\sigma_{a-1}\ldots \sigma_1 P\xx) \,\,d\kappa(\xx)
\end{align*}
where the expectation $\E$ refers to the uniform probability measure on the subgroup in subscript. Applying Cauchy-Schwarz then Fubini, we obtain with $\sigma=\sigma_{a-1}\cdots  \sigma_1=\sigma' \sigma_1$, $\sigma' \in F_{a-1}\ldots F_2$,
\begin{align*}
|\widehat{ P_{\star}\kappa}(\xi)|^2
& \leq\int_{\tkg^N}\mathbb{E}_{\sigma'} |\mathbb{E}_{\sigma_{1}} e_{\xi}(\sigma  P\xx)|^2 \,\,d\kappa(\xx)\\
& =\int_{\tkg^N}\mathbb{E}_{\sigma'} \mathbb{E}_{\sigma_{1}, \tau_{1}} e_{\xi}\left(\sigma'\sigma_1 P\xx - \sigma'\tau_1 P\xx)\right) \,\,d\kappa(\xx)\\
& =\int_{\tkg^N} \mathbb{E}_{\sigma_{1}, \tau_{1}}\mathbb{E}_{\sigma'} e_{\xi}\left((\sigma_1-\tau_1)\sigma' P \xx \right) \,\,d\kappa(\xx)
\end{align*}
where $\tau_{1}\in F_1$ is a new independent variable with uniform distribution on $F_1$.  We can repeat this process and obtain the \emph{path-swap inequality}:

\begin{align}\label{VanDerCorput}
|\widehat{ P_{\star}\kappa}(\xi)|^{2^{a-1}}&\leq \int_{\tkg^N} \mathbb{E}_{\sigma_{1}, \tau_{1}}\dots  \mathbb{E}_{\sigma_{a-1}, \tau_{a-1}} e_{\xi}\left(A_{\sigma, \tau}  P \xx\right) \,\,d\kappa(\xx)\nonumber \\ &\leq \int_{\tkg^N} \mathbb{E}_{\sigma, \tau} e_{\xi}\left( A_{\sigma, \tau}  P \xx \right) \,\,d\kappa(\xx)
\end{align}
where $\E_{\sigma,\tau}$ is over the uniform distribution on $F\times F$ and
$$
A_{\sigma, \tau} := \prod_{i=1}^{a-1} (\sigma_i-\tau_i) \in \R[\Sym(N)].
 $$
This inequality will  be applied later in the proof of \Cref{cor-moyenne} and $P$ will be taken to be the map $\xx\mapsto \Pi^{[<a]}\xx+\Pi^{[a]}\xx_{ab}$ where $\xx_{ab}$ refers to the coordinate-wise projection of $\xx$ to $\km^{(1)} \oplus \R\chi$ modulo $\kg^{[2]}$. First we study how $A_{\sigma, \tau}$ acts on $ \Pi^{[<a]}$ and $\Pi^{[a]}$ for an appropriate choice of permutation subgroup $F$. 
\bigskip

We  specify $F$ as follows. Fix $a\in \{2, \dots, s\}$ and integers $k,N'\ge 1$. Assume $N=k(2a-1)N'$. We divide $[N]=\{1,\ldots,N\}$ into $N'$ consecutive intervals. called \emph{large blocks}, of length $k(2a-1)$ and subdivide each large block into $2a-1$ consecutive intervals, or blocks, of size $k$. In each large block the central block is said to have type $0$, while the two blocks on each side of the central block have type $1$, the new blocks neighbouring a block of type $t$ have type $t+1$, and so on. The \emph{type of a block} is thus an integer in $[0,a-1]$.  The type of an element of $[N]$ is defined as the type of the block it belongs to.  For $I\subseteq [N]$ let $t(I)$ be the set of block types of blocks that meet $I$. Clearly $|t(I)|\leq |I|$.

For $j\in [N']$ and $i \in [a-1]$ we let $\eps_i^j$ be the permutation of $[N]$ that exchanges the two blocks of type $i$ in the $j$-th large block (preserving order within a block).  In other words, $\eps_i^j$ exchanges $r-ki\in [N]$ and $r+ki$ if $\lceil r/k \rceil = j(2a-1)-(a-1)$ and fixes every other element. We let $F_i$ be the subgroup generated by all $(\eps_i^j)_{j\in [N']}$ and $F$ the subgroup generated by all $F_i$. Since the $\eps_i^j$ all have disjoint support, $F$ is an elementary abelian $2$-group isomorphic to $(\Z/2\Z)^{N'(a-1)}$. Any $\sigma\in F$ can be written uniquely 
$$\sigma=\prod_{i \in [a-1], j\in [N']} \sigma_i^j$$
where $\sigma_i^j\in \{id, \eps_i^j\}$. We also set $\sigma_i:=\prod_j\sigma_i^j$, $\sigma^j:=\prod_i \sigma_i^j$. Observe that $F$ \emph{preserves type} by construction.

 The next lemmas are inspired from  \cite[Lemma 24]{diaconis-hough21} and \cite[Lemma 7]{hough19}. The first one shows that the operator $A_{\sigma, \tau}$ kills the components $\Pi^{[a']}$ for $a'<a$.

\begin{lemme}[path swap 1] \label{moyenne}
Let $\sigma, \tau \in F$. The kernel of $A_{\sigma,\tau}$ contains all monomials whose support misses some block type in $[a-1]$. It also contains $\Pi_{N,t}$ and $\Pi_N^{[t]}$ for every $t\leq a-1$.
\end{lemme}

\begin{proof}Note first that for each $i$, if $R \in \mathcal{A}_N$ is fixed by all elements of $F_i$, then $R \in\ker A_{\sigma,\tau}$. If $m\in \mathcal{A}_N$ is a monomial with support $I_m$ and $[a-1]\nsubseteq t(I_m)$, then there is $i \in [a-1]$ such that $F_i$ fixes $m$. This shows the first assertion and also that $\mathcal{A}_{N,t}$ and $\mathcal{A}_{N}^{[t]}$ lie in $\ker A_{\sigma,\tau}$ for all $t<a-1$. 

It remains to check that  $\Pi_N^{[a-1]}$ and  $\Pi_{N,a-1}$ are in $\ker  A_{\sigma,\tau}$. More generally, we check that any $R\in \mathcal{A}_{N,a-1}$ belonging to the image of the periodization map $\mathcal{P}_{a-1}$ (see \Cref{Sec-bi-grading}) lies in $\ker A_{\sigma,\tau}$.
There is $Q \in \mathcal{A}_{a-1,a-1}$ such that
$$A_{\sigma,\tau}R= A_{\sigma,\tau} \sum_{I}\sigma_IQ.$$ By the first paragraph of the proof, the sum can be restricted to those $I$ in $\mathcal{I}:=\{I\subseteq [N] \,:\,|I|=a-1,\,t(I)=[a-1]\}$. 
Note that $F_{1}$ permutes  $\mathcal{I}$, and acts on each $I\in \mathcal{I}$ in an order preserving way.
This implies that  the corresponding $\sigma_IQ$  are permuted by $F_1$, whence their sum $\sum_{I\in \mathcal{I}}\sigma_IQ$ is fixed by $F_1$ and thus lies in $\ker A_{\sigma,\tau}$. It follows that $A_{\sigma, \tau}R=0$ as desired.
\end{proof}


The second lemma shows that $A_{\sigma,\tau}$ desynchronizes the contributions of different large blocks. Let $(B_j)_{j\in [N']}$ be the sequence of large blocks.

\begin{lemme}[path swap 2] \label{moyenne2}For all $\sigma,\tau \in F$, we have: $$A_{\sigma,\tau} \Pi_{N,a} = \sum_{j=1}^{N'} \sum_{I\subseteq B_j} A_{\sigma,\tau} (\Pi_{N,a})_I.$$ 
Furthermore, the second sum can be  restricted to all $I\subseteq B_j$ of size $a$ with at most one element in each block. 
\end{lemme}
\begin{proof}For $i\in [N]$, let $t(i)$ be the type of $i$, and for $I \subseteq [N]$ with $|I|=a$, let $t_s(I)$ be the number of $i \in I$ with $t(i)=s$. Let also $s(I)$ be the largest $s\in [0,a-1]$ for which there is a large block $B_j$ such that  $t_{s'}(I\cap B_j)=0$ for all $s'<s$, while $t_s(I)=t_s(I\cap B_j)=1.$ Note that by  \Cref{moyenne} we may assume that $[a-1]\subseteq t(I)$.  If $s(I)=0$, then $I$ is contained in a single large block, because the elements of $I$ will all have distinct types. On the other hand, given $s$, the set $\{I, |I|=a, t(I)\supseteq[a-1], s(I)=s\}$ is invariant under $F$ and when $s\ge 1$, $F_s$ is order-preserving on each $I$ in this set. This means that  $F_s$ permutes the associated $\sigma_I\Pi_{a,a}$ (see Claim 1 in \Cref{Sec-bi-grading}). So their sum is fixed by $F_s$, hence lies in $\ker A_{\sigma,\tau}$. Finally, if $s(I)=-\infty$, i.e. $s(I)$ is the supremum of an empty set,  then $I$ has one element of each type $>1$ and two elements of type $1$ that are either in the same large block (case (a)) or in distinct large blocks (case (b)). Note that in case (a), $I$ lies in a single large block. Also $F_1$ acts in an order-preserving way on the subfamily of such $I$ in case (b) and also on the subfamily of such $I$ in case (a) whose type $1$ elements are in the same block. So the sum of the associated $\sigma_I\Pi_{a,a}$ is also in $\ker A_{\sigma,\tau}$.  This ends the proof of the displayed formula and final assertion. 
\end{proof}

Next, we make explicit the contribution of each large block under the condition that the variables in the blocks left of the central block in each large block are all zero.

\begin{lemme}[path swap 3] \label{moyenne3} For $j\in[N']$ and $\sigma,\tau \in F$ we have:
\begin{enumerate}[label=(\roman*)]
\item if $\sigma_i^j=\tau_i^j$ for some $i \in [a-1]$, then $A_{\sigma,\tau} (\mathcal{A}_{N})_I=0$ for all $I\subseteq B_j$, 

\item if $\sigma_i^j \neq \tau_i^j$ for each $i\in [a-1]$, if $\kh$ is a Lie algebra and $\xx \in \kh^N$ is such that  $x_{\ell}=0$ for every $\ell$ located to the left of the central block  in the large block $B_j$, then \begin{equation}\label{lastid}\sum_{I\subseteq B_j} A_{\sigma,\tau} (\Pi_{N,a}^{[a]})_I\xx = \eps(\sigma^j)L^{[a]}(\overline{x}_{j,0},\ldots,\overline{x}_{j,a-1})\end{equation}
where $\eps$ is the sign character and $\overline{x}_{j,k}$ is the sum of the $x_\ell$ for $\ell$ in the blocks of type $k$ in $B_j$. 
\end{enumerate} 
\end{lemme}

\begin{proof} Without loss of generality we may assume $j=N'=1$. The first assertion is obvious, because $A_{\sigma,\tau}=0$ in this case.  For $(ii)$, we have  $A_{\sigma,\tau}=\eps(\sigma)A_{1,\tau}$, where $A_{1,\tau}=\prod_{i=1}^{a-1}(1-\eps_i)$ (and $\eps_i:=\eps_i^1$). By  \Cref{moyenne2} we may further assume that the $I$ involved in the sum belong to $\mathcal{I}_1$, the subfamily of those $I$ with at most one element in each block. Let $r(I)$ be the $a$-tuple recording the labels in $[-a+1,a-1]$ of the blocks intersecting $I$ (in increasing order).  Now observe that $\Pi_{a,a}^{[a]}$ is made of monomials of length $a$ with support of size $a$. So it is linear in each variable, and by Claim 1 in \Cref{Sec-bi-grading}, we have $$\sum_{\rr(I)=\rr}(\Pi_{N,a}^{[a]})_I=\sum_{\rr(I)=\rr}\sigma_I\Pi_{a,a}^{[a]}=\Pi_{a,a}^{[a]}(\overline{u}_{r_1}, \ldots,\overline{u}_{r_a}),$$  where $\overline{u}_{r_i}$ denotes the sum of the letters $u_j$ for $j$ in the block  labelled $r_i$. Hence:
$$\sum_{I\in \mathcal{I}_1} (\Pi_{N,a}^{[a]})_I=\sum_{\rr} \Pi_{a,a}^{[a]}(\overline{u}_{r_1}, \ldots,\overline{u}_{r_a}) = \Pi_{2a-1,a}^{[a]}(\overline{u}_{-a+1}, \ldots,\overline{u}_{a-1})$$ where the second equality follows from Claim 1 in \Cref{Sec-bi-grading}.

We now consider the free associative algebra $\mathcal{A}'$ on $2a-1$ letters $v_{-a+1},\ldots,v_0,\ldots,v_{a-1}$ and claim that modulo the bilateral ideal $\mathcal{B}_a$  generated by the $v_i$, $i<0$ we have:
$$A_{1,\tau}\Pi^{[a]}_{2a-1,a} \equiv  L^{[a]}(v_0,\ldots,v_{a-1}).$$ Given the vanishing assumption made on the coordinates in $(ii)$, the claim implies \cref{lastid}. 

To prove the claim, proceed by induction on $a$. When $a=2$, we have $(1-\eps_1)\Pi_3=v_{-1}*v_0*v_1 - v_1*v_0*v_{-1}\equiv v_0*v_1 - v_1*v_0$ modulo $\mathcal{B}_2$, which coincides with $[v_0,v_1]$ modulo monomials of length at least $3$.   
Suppose now $a>2$. In view of \Cref{moyenne2}, $$A_{1,\tau}\Pi^{[a]}_{2a-1,a} \equiv A_{1,\tau} R_a$$ where $R_a$ is the part of $\Pi^{[a]}_{2a-1,a}$ consisting of monomials of full type (i.e. with one and only one letter of each type). Indeed the remaining monomials from \Cref{moyenne2}, as well as their images under  permutations in $F$, belong to $\mathcal{B}_a$. By associativity of the $*$ product one may write 
$$\Pi_{2a-1}(v_{-a+1},\ldots,v_{a-1})=\Pi_{2a-3}(\underline{w}) + r$$
where $\underline{w}:=(v_{-a+1},\ldots,v_{-2},w_0,v_2,\ldots,v_{a-1})$ with $w_0:=v_{-1}*v_0*v_1$, and $r\in \mathcal{A}'^{[\ge s+1]}$. Expanding $\Pi_{2a-3}(\underline{w})$ in the variables $v_i$, we see that the monomials coming from $\Pi_{2a-3}^{[\leq a-2]}(\underline{w})$, $\Pi_{2a-3}^{[\geq a+1]}(\underline{w})$ or $r$ will not be of full type in the variables $v_i$, hence will not contribute. It follows that $R_a$ coincides with the part of full type of the expression
\begin{equation}\label{eqsumaa}\sum_{I} (\Pi^{[a-1]}_{2a-3})_I(\underline{w})+ \sum_{I} (\Pi^{[a]}_{2a-3})_I(\underline{w})\end{equation}
where the $I$ in each sum have one letter of each type from $\{0\}\cup [2,a-1]$ and the second sum restricts to those monomials with the central letter appearing twice. The part of full type of the second sum coincides with that of 
$$ \sum_{I} (\Pi^{[a]}_{2a-3})_I(v_{-1}+v_0+v_1)$$
which is clearly invariant under $F_1$ and thus will disappear when applying $A_{1,\tau}$. We can thus focus on the first sum in \cref{eqsumaa}. Its image under $\prod_{i>1}(1-\eps_i)$ coincides with that of $\Pi_{2a-3,a-1}^{[a-1]}$. By the induction hypothesis this image is equal to $L^{[a-1]}(w_0,v_2,\ldots,v_{a-1})$ modulo the ideal generated by $v_{-2},\ldots,v_{-a+1}$.  Applying now $(1-\eps_1)$ the desired conclusion follows from the $a=2$ case.
\end{proof}

\bigskip
We now combine the  path swap lemmas  to show \Cref{cor-moyenne}, which   gives a first upper bound on the Fourier transform of the  product of random variables where we remove brackets of length at least  $a+1$.

\begin{lemme} \label{cor-moyenne}
 Let $\xi \in \dkg$. Let $a\in \{2, \dots, s\}$, $N\geq 1$ of the form  $N = kaN'$ where $k,N'\geq 1$,   let  $\xx_{-1}, \xx_{1}$ be tuples in $\tkg$ and for $\xx\in \tkg^N$ write $\yy=(\xx_{-1},\xx,\xx_1)$. For any probability measure $\omega$ on $\tkg$, we have  
$$\big| \E_{\xx\sim \omega^{\otimes N}}\left[e_{\xi}\left(\Pi^{[<a]}\yy+\Pi^{[a]}\yy_{ab}  \right)\right] \big|^{2^{a-1}}  \leq \left( 1-\frac{1}{2^{a-1}}+ \frac{1}{2^{a-1}} \E_{(\omega^{+k}_{ab})^{\otimes a}}\big[\cos(2\pi \xi L^{[a]} (v_{1}, \dots, v_{a})    )\big] \right)^{N'} $$
where $\omega_{ab}$, $\yy_{ab}$ denote the projections of $\omega$, $\yy$  to $\km^{(1)}\oplus \R\chi$ modulo $\kg^{[2]}$, and $\omega^{+k}_{ab}$ is the distribution of the sum of $k$ independent variables on $\km^{(1)}\oplus \R\chi$ with law $\omega_{ab}$.
\end{lemme}



\begin{proof} We first add some dummy zero variables to get a more suitable random vector to which we will apply the path-swap lemmas. 
More precisely, define $\kappa_0=\delta^{\otimes (a-1)k}_0\otimes \omega^{\otimes ak}$. Write in coordinates $\xx_{1}= (x_{1, 1}\dots x_{1, q})$ for some $q\geq 0$ and define $\kappa_{1}= \delta_{0}^{\otimes(2a-1)k-1}\otimes \delta_{x_{1,1}} \otimes\dots\otimes \delta_{0}^{\otimes(2a-1)k-1}\otimes \delta_{x_{1,q}}$. Similarly, define $\kappa_{-1}$ for $\xx_{-1}$. 
Finally, set $\kappa= \kappa_{-1}\otimes \kappa_{0}^{N'}\otimes\kappa_{1}$, which is a probability measure on  $\tkg^{N''}$ for some $N''\geq N$ in $(2a-1)k\N$. Write also $S\yy=\Pi^{[<a]}\yy+\Pi^{[a]}\yy_{ab} $.

As the variables $(S\yy)_{\xx\sim  \omega^{\otimes N}}$ and $(S \zz)_{\zz\sim \kappa}$ have same law,  the path-swap inequality \eqref{VanDerCorput} applied to the function $P=S$ gives
\begin{align*}
\big| \E_{\xx\sim \omega^{\otimes N} }\left[e_{\xi}(S\yy)\right]  \big|^{2^{a-1}}
=\big| \E_{\zz \sim \kappa}\left[e_{\xi}\left( S\zz \right)\right]  \big|^{2^{a-1}}
 \leq \int_{\tkg^{N''}} \mathbb{E}_{\sigma, \tau} \left[e_{\xi}\left( A_{\sigma, \tau} S \zz\right)\right]\,\,d\kappa(\zz).
\end{align*}
 \Cref{moyenne} tells us that $A_{\sigma, \tau}\Pi^{[<a]}=0$, hence $A_{\sigma, \tau} S \zz= A_{\sigma, \tau} \Pi^{[a]}\zz_{ab}$. Then, using \Cref{moyenne2} to desynchronize large blocks and  \Cref{moyenne3} to see the large blocks  with coordinates from $\xx_{-1}$ or  $\xx_{1}$ do not contribute, we get
\begin{align*}
\big| \E_{\xx\sim \omega^{\otimes N} }\left[e_{\xi}(S\yy)\right]  \big|^{2^{a-1}}
& \leq \int_{\tkg^{N''}} \mathbb{E}_{\sigma, \tau} \left[e_{\xi}\left( A_{\sigma, \tau} \Pi^{[a]}\zz_{ab}\right)\right] \,\,d\kappa(\zz)\\
& =  \mathbb{E}_{\sigma, \tau} \prod_{j=1}^{N'} \int_{\tkg^{(2a-1)k}}e_{\xi}\left( A_{\sigma^j, \tau^j} \Pi^{[a]}\uu_{ab}\right) \,\,d\kappa_0(\uu)\\
& = \left( \mathbb{E}_{\sigma^1, \tau^1} \int_{\tkg^{(2a-1)k}}e_{\xi}\left( A_{\sigma^1, \tau^1} \Pi^{[a]}\uu_{ab}\right) \,\,d\kappa_0(\uu)\right)^{N'}
\end{align*}
where the last identity exploits the independence of the $\sigma^j$, $\tau^j$, $j\in [N']$. By \Cref{moyenne3}, the last integral is equal to $1$ for a proportion $1-\frac{1}{2^{a-1}}$ of possibilities for $(\sigma^1, \tau^1)$, and for the other scenarios it is equal to 
$$ \int_{\tkg^{(2a-1)k}} e_{\xi} \left(\eps(\sigma^{1})  L^{[a]}(v_{1}, \dots, v_{a})\right) \,d(\omega^{+ k}_{ab})^{\otimes a}(v) $$
The result follows by observing that on the scenarios we consider, $\sigma^{1}$ is even or odd with the same probability. 
\end{proof}

To prove \Cref{reduction} (domain reduction), we will apply   \Cref{cor-moyenne} to a fixed $\xi$ and some integer $k$, chosen to guarantee that the right hand side in \Cref{cor-moyenne} decays exponentially in $N'$. The existence of such  $k$ is justified by the next proposition. It is the only result from \Cref{Sec-path swap} that will be used in the rest of the proof of \Cref{reduction}. 

For $a\geq 1$, $b\geq 1$, we denote by $\kg^{[a, b)}$ the subspace of $\kg$ spanned linearly by the brackets of the form $[x_{1}, [x_{2},\dots [x_{j-1}, x_{j}]\dots ]]$ with $j\geq a$, $x_{i}\in\kg$, and $j+|\{i:x_{i}=\Xab \mod \kg^{[2]}\}|\geq b$. For instance,  $\kg^{[1,2)}=\kg^{[2]}+\R \Xm$, $\kg^{[a,2)}=\kg^{[a]}$ if $a\geq 2$. If $b\geq 3$, then $\kg^{[a,b)}\subseteq \kg^{[a]}\cap \kg^{(b)}$ and $\kg^{[1,b)}=\kg^{(b)}$  (see \Cref{basics-weight-filtration}).

Recall that given, $c\in (0,1)$, we let $gap_{c}: \R_{>0}\rightarrow (0, c)$ be the function given by 
$$gap_{c}(R)= c\inf \{1-|\hmu_{ab}(\xi)| \,:\, \xi \in \widehat{\km^{(1)}}, \, c\leq \|\xi \|\leq c^{-1}(1+R)\}.$$

\begin{proposition} \label{choisirk}
 Let  $\xi \in \dkg\smallsetminus \{0\}$, $a\in \{1, \dots, s\}$ and $k\geq 1$ the smallest integer for which there exists $b_{0}\in \{a, \dots, 2a-1\}$ satisfying $\|k^{\frac{b_{0}}{2}}\,\xi_{|\kg^{[a,b_{0})}}\|\geq 1$. Assume $\|\xi_{|\kg^{[a+1, b_{0})}}\|\lll\| \xi_{|\kg^{[a,b_{0})}}\|$ for the smallest  such $b_{0}$. Then, for $N'\geq 1$,  $N = kaN'$, $N_{0}\geq N$, $\xx_{-1}, \xx_{1}$ finite length tuples in $\tkg$,  $\yy=(\xx_{-1},\xx,\xx_1)$, and $c\lll1$, we have 
$$\big| \E_{\xx\sim(T_{N_{0}}\tmu)^{\otimes N}}\left[e_{\xi}\left( \Pi^{[<a]}\yy+\Pi^{[a]} \yy_{ab} \right)\right] \big|^{2^{a-1}} \ll  (1-gap_{c}(\|\xi\|))^{N'} $$
where $\yy_{ab}$ denotes the coordinate-wise  projection of $\yy$ to $\km^{(1)}\oplus \R \chi$  modulo  $\kg^{[2]}$. 
\end{proposition}

\begin{proof}
We first define the implicit constant that will occur in the condition $\|\xi_{|\kg^{[a+1, b_{0})}}\|\lll\| \xi_{|\kg^{[a,b_{0})}}\|$ of the statement. For $a,b\geq 1$, let $V_{a,b}$ be the  subspace spanned by the $a$-brackets $L^{[a]}(\vv)$ with $\vv \in (\km^{(1)})^a$ and at least $b-a$ components $v_{i}$ equal to $X$.  
 Note that we have $\kg^{[a,b)}\subseteq V_{a,b}+\kg^{[a+1,b)}$. Hence, we may fix a constant  $D>1$ only depending  on the initial data $(\kg, \mu, (\km^{(b)})_{b}, (e^{(i)}_{j})_{i,j})$ such that
 for any $a,b\geq 1$, if a linear form $\xi\in \dkg$ satisfies $D\|\xi_{|\kg^{[a+1,b)}}\|\leq \|\xi_{|\kg^{[a,b)}}\| $, then we have $\|\xi_{|\kg^{[a,b)}}\| \leq  D\|\xi_{|V_{a,b}}\|$.

\bigskip

Let us now prove the proposition in the case $a=1$, which is elementary and helps build an intuition for the case $a\geq 2$. Let $\xi \in \dkg\smallsetminus \{0\}$, and  $k\geq 1$  the smallest integer for which $\|k^{\frac{1}{2}}\,\xi\|\geq 1$. Assume $D\|\xi_{|\kg^{[2]}}\|\leq \|\xi\|$, in particular $\|\xi\|\leq D\|\xi_{|\km^{(1)}}\|$. The quantity to bound can be rewritten
\begin{align*}
\big| \E_{ (T_{N_{0}}\tmu)^{\otimes N}}\left[e_{\xi}\left( \sum\yy_{ab}  \right)\right] \big|= \big| \widehat{(T_{N_{0}}\tmu)_{ab}}(\xi) \big|^N= \big| \widehat{(T_{N_{0}}\tmu)^{+k}_{ab}}(\xi) \big|^{N'}. \end{align*} 
For $N'$ (whence $N_{0}$) going to infinity, we have
\begin{align*}
\big| \widehat{(T_{N_{0}}\tmu)^{+k}_{ab}}(\xi) \big|
&\leq \big| \widehat{\tmu^{+k}_{ab}}(\xi) \big| + \|(T_{N_{0}}\tmu)^{\otimes k}-\tmu^{\otimes k}\|\\
&\leq \big| \widehat{\tmu^{+k}_{ab}}(\xi) \big|  +o(kN_{0}^{-m_{\mu}/2})\\
&\leq \big| \widehat{\mu^{+k}_{ab}}(\xi) \big|   +o(N'^{-1})
\end{align*}
where the second inequality relies on \Cref{truncation-cost}, and the third inequality uses $N_{0}\geq kN'$, $m_{\mu}\geq 2$ and the fact that $\tmu^{+k}_{ab}$ by $\mu^{+k}_{ab}$ only differ by a translation ($\mu^{+k}_{ab}=\tmu^{+k}_{ab}-k\chi+k\Xm$) so their Fourier transform have same modulus. It follows that as long as $c\lll1$ and $N'\geq 1/gap_{c}(\|\xi\|)$, 
\begin{align*}
\big| \E_{ (T_{N_{0}}\tmu)^{\otimes N}}\left[e_{\xi}\left( \sum\yy_{ab}  \right)\right] \big|\leq (\big| \widehat{\mu^{+k}_{ab}}(\xi) \big|   +gap_{c}(\|\xi\|))^{N'}.
\end{align*}
Note that the condition $N'\geq 1/gap_{c}(\|\xi\|)$ can be assumed without loss of generality to prove the proposition, because $\inf_{0<t<1/2}(1-t)^{1/t}>0$. Hence, we are reduced to show that 
 $$\big| \widehat{\mu^{+k}_{ab}}(\xi) \big| \leq 1-2gap_{c}(\|\xi\|)$$
  up to choosing $c>0$  small enough in terms of the initial data. 
To see this, let $r>0$ be a parameter depending only on the initial data, and that we will specify below. 
If $r\leq \|\xi\|$, then we saw that $\|\xi_{|\km^{(1)}}\|\geq r/D$, so the aperiodicity assumption on $\mu$ implies that $|\widehat{\mu_{ab}}(\xi)| \leq 1-2gap_{c}(\|\xi\|)$ as soon as $c<\min(r/D,1/2)$.
 But as $r$ tends to $0$, the parameter $k$ goes to infinity, so the distribution $k^{-\frac{1}{2}} \mu_{ab}^{+k}$ converges (after recentering) to a non-degenerate Gaussian distribution, whose Fourier transform outside any fixed neighborhood of $0$ is uniformly bounded away from  $1$. Thus we may choose $r$ small enough so that $\big| \widehat{\mu^{+k}_{ab}}(\xi) \big| =|\E_{k^{-\frac{1}{2}}\mu^{+ k}_{ab}} (e_{k^{\frac{1}{2}}\xi} ) \big| $ is strictly less than  $1$ uniformly in $\|\xi\|\leq r$, in particular less than $1-2c$ if $c$ is small enough, which concludes the case $a=1$. 

\bigskip
We now deal with the case $a\geq 2$. 
Let  $\xi \in \dkg \smallsetminus \{0\}$, $k\geq 1$ the smallest integer satisfying $\|k^{\frac{b_{0}}{2}}\,\xi_{|\kg^{[a,b_{0})}}\|\geq 1$
for some $b_{0}\in \{a, \dots, 2a-1\}$, and assume $2D \|\xi_{|\kg^{[a+1,b_{0})}}\| \leq \| \xi_{|\kg^{[a,b_{0})}} \|$ for the smallest such $b_{0}$. 
We are going to show that
 \begin{equation}\label{aperbound}\E_{(\mu_{ab}^{+k})^{\otimes a}}\big[\cos(2\pi \xi L^{[a]} (v_{1}, \dots, v_{a}))\big] \leq 1-2gap_{c}(\|\xi\|)\end{equation}
 up to choosing $c\lll1$. Applying \Cref{cor-moyenne} with $\omega=T_{N_{0}}\tmu$, arguing as above to remove the truncation, recalling that $[\chi,.]=[\Xm,.]$ to replace $\tmu$ by $\mu$,  this upper bound implies \Cref{choisirk}.
 

We let $r\in (0, 1)$ be a parameter only allowed to depend  on the initial data, and to be specified below. 

We first deal with the case where $\|\xi_{|\kg^{[a]}}\|\geq r$, in which case $k$ is bounded above by $k_{r}=\lceil r^{-2/a}\rceil$. For $r_{1}\in (0, 1)$, we introduce $E_{r_{1}}$ the set of  linear forms $\xi'$ on $\kg$  such that there exists $b\in \{a,\dots, 2a-1\}$ for which 
$$r_{1}^2 \leq \|\xi'_{|\kg^{[a]}}\| r_{1}\leq  \|\xi'_{|\kg^{[a,b)}}\|  \,\,\,\,\,\,\text{ and }   \,\,\,\,\,\, D \|\xi'_{|\kg^{[a+1, b)}}\| \leq  \|\xi'_{|\kg^{[a,b)}}\|.$$ 
Note we may choose $r_{1}$ small enough depending on $r$ and the initial data so that $\xi \in E_{r_{1}}$. Indeed, the claim is trivial if $\|\xi_{|\kg^{[a]}}\| \geq 1$ because then $k=1$, $b_{0}=a$, so one may take $b=a$ to check the conditions for $E_{r_{1}}$. For the case $\|\xi_{|\kg^{[a]}}\| \in [r,1]$, note that $\|\xi_{|\kg^{[a,b_{0})}}\|\geq k_{r}^{-s}$,  whence the claim if $r_{1}= k_{r}^{-s}$.  Now we show that the upper bound \eqref{aperbound} holds uniformly for linear forms in $E_{r_{1}}$ and $k$ bounded by  $k_{r}$, with some constant $c$ depending only on $r$ and the initial data.

Let $\xi'\in E_{r_{1}}$, $k\in [k_{r}]$. The definitions of $D$ and $E_{r_{1}}$ imply that $\|\xi'_{|V_{a,b}}\|\geq D^{-1} \|\xi'_{|\kg^{[a,b)}}\| \geq D^{-1}r_{1}^2$, in particular we have
$$\xi' L^{[a]}(\km^{(1)}, \dots, \km^{(1)})\neq \{0\}.$$
Since $\mu$ is assumed aperiodic, so is $\mu_{ab}^{+k}$, thus the subspace spanned by the support of $\mu_{ab}^{+k}$ is all of $\km^{(1)}$. We deduce  that there is a positive proportion of tuples $(v_{2}, \dots, v_{a})$ chosen with law $(\mu_{ab}^{+k})^{\otimes a-1}$  such that
$ v_{1}\mapsto \xi' L^{[a]}(v_{1}, v_{2}, \dots, v_{a})$ is a non-zero linear form on $\km^{(1)}$. In other words, there exists $r_{2}$ depending on $\xi', k$ and the initial data  such that for a $(\mu_{ab}^{+k})^{\otimes a-1}$-proportion at least $r_{2}$ of tuples $v_{\geq 2}=(v_{2}, \dots, v_{a})$, the linear form $ \varphi_{v_{\geq 2}} : v_{1}\mapsto \xi' L^{[a]}(v_{1}, v_{2}, \dots, v_{a})$ on $\km^{(1)}$ has its norm in the interval $\|\xi'_{|\kg^{[a]}}\|[r_{2}, r_{2}^{-1}]$. By compactness of $E_{r_{1}}$ in the projective space and the bound on $k$, we may actually suppose that $r_{2}$ only depends on $r_{1},r$ and the initial data. We deduce that for such $(v_{2}, \dots, v_{a})$, one has 
$|\hmu_{ab} (\varphi_{v_{\geq 2}})|\leq 1-gap_{r_{1} r_{2}}(\|\xi'_{|\kg^{[a]}}\|)$. Taking the real part of the Fourier transform and integrating in $(v_{2}, \dots, v_{a})$, we deduce \eqref{aperbound} with $c=\frac{1}{2} r_{1}r^2_{2}$, a quantity that  depends only on $r$ and the initial data.

We consider now the case $\|\xi_{|\kg^{[a]}}\|<r$. We write $v_{i}=v_{i}'+ k\Xm$ where $v'_{i}\in \km^{(1)}$ and expand the $a$-bracket using multilinearity to obtain 
$$L^{[a]} (v_{1}, \dots, v_{a}) =\sum_{b=a}^{2a-1}k^{b-a}M_{a,b}(v'_{1}, \dots, v'_{a}) $$
where $M_{a,b}(v'_{1}, \dots, v'_{a})$ is the sum of the $a$-brackets with exactly  $b-a$ occurrences of $\Xm$ (and $2a-b$ occurences of the variables $v_{i}'$). In particular $M_{a,b}$ takes values in $V_{a,b}\subseteq \kg^{[a,b)}$. Applying $\xi$, we then obtain 
\begin{align*}
\xi L^{[a]} (v_{1}, \dots, v_{a}) &=\sum_{b=a}^{2a-1}k^{\frac{b}{2}}\xi\,M_{a,b}(\frac{v'_{1}}{k^{1/2}}, \dots, \frac{v'_{a}}{k^{1/2}}).
\end{align*}

We observe that as $r$ goes to $0$, the parameter $k$ goes to infinity, so each term $M_{a,b}(\frac{v'_{1}}{k^{1/2}}, \dots, \frac{v'_{a}}{k^{1/2}})$ converges in distribution to $M_{a,b}(\mathscr{N}_{1}, \dots, \mathscr{N}_{a})$ where the $\mathscr{N}_{i}$'s are independent centred normal distributions on $\km^{(1)}$ with the same covariance matrix as $\mu_{ab}$. Moreover, the linear forms $k^{\frac{b}{2}}\xi_{|\kg^{[a,b)}} $, applied to the various terms, have norm less than $2$ (up to taking $r$ small enough depending on the initial data) and at least one of these has norm greater than $1$ while its restriction to the subspace $\kg^{[a+1,b)}$ has norm less than $D^{-1}$. As these conditions define a compact set, the upper bound \eqref{aperbound} (hence \Cref{choisirk})  then follows by choosing $r$ small enough and applying \Cref{caslimite} below.
\end{proof}

 \begin{lemme} \label{caslimite} There exists a constant $c_{1}>0$ depending only the initial data such that, for  $a\in \{2, \dots, s\}$, for all linear forms  $(\varphi_{b}\in \widehat{\kg^{[a,b)}})_{a\leq b\leq 2a-1}$ satisfying  $\max_{b}\|\varphi_{b}\|\leq 2$ and $D\|\varphi_{b_{0} | \kg^{[a+1,b_{0})}}\| \leq 1\leq \|\varphi_{b_{0}}\|$ for some $b_{0}$, one has
$$\E_{\mathscr{N}_{1}\otimes \dots \otimes \mathscr{N}_{a}}\left(\cos\left(2\pi \sum_{b=a}^{2a-1} \varphi_{b}M_{a,b}(x_{1}, \dots, x_{a})\right)\right) <1-c_{1}.$$
 \end{lemme} 

\begin{proof}[Proof of \Cref{caslimite}]
Fix $a$. The condition on the norms of the $\varphi_{b}$'s defines a compact subset of  $\prod_{b\geq a}\widehat{\kg^{[a,b)}}$. The expectation we aim to bound varies continuously with respect to the $\varphi_{b}$'s. So it is sufficient to check the lemma for a fixed family  $(\varphi_{b})_{b\geq a}$. This boils down to showing that $\sum_{b=a}^{2a-1} \varphi_{b}M_{a,b}$ is not constant, i.e. non-identically zero, on  $(\km^{(1)})^{a}$. 

To see this,  let $\mathcal{B}$ be the set of $b\in [a, 2a-1]$ such that $\varphi_{b}$ is non identically zero on $V_{a,b}$.  The set $\mathcal{B}$ is non-empty because it contains the element $b_{0}$ such that $D\|\varphi_{b_{0} | \kg^{[a+1,b_{0})}}\| \leq 1\leq \|\varphi_{b_{0}}\|$.

Let $b_{1}=\max \mathcal{B}$, and $L^{[a]}(\vv)$ with $\vv \in (\km^{(1)})^a$ and at least $b-a$ components of $v_{i}$ equal to $X$ such that $\varphi_{b_{1}}(L^{[a]}(\vv))\neq 0$. Choose $I\subseteq [a]$ such that $|I|=b_{1}-a$ and $v_{i}=\Xm$ for all $i\in I$. Define $\xx$ by $x_i= v_i \1_{i\notin I}$. Then $\varphi_{b}M_{a,b}=0$ for $b >b_1$ due to the maximality of $b_{1}$, also  $M_{a,b}(\xx)=0$ for $b < b_1$ because all the brackets appearing inside $M_{a, b}(\xx)$  involve then a null coordinate of $\xx$,  and finally $M_{a,b_1}(\xx)=L^{[a]}(\vv)$ for a similar reason. So
$$ \sum_{b=a}^{2a-1} \varphi_{b} M_{a,b}(\xx)=\varphi_{b_{1}}(L^{[a]}(\vv))$$
is non-zero and the proof is complete.
\end{proof}


\subsection{Reducing the domain of integration} \label{SecRed} \label{Sec6}

In this section, we apply \Cref{choisirk}  to show   
 \Cref{reduction} below.
  Recall that for  $N\geq1$,  $\delta_{0}\geq 0$, $\gamma_{0} \in (0, 1)$, we have defined the range of moderate deviations:
$$\DN(\delta_{0})=\{x \in \tkg \,:\, \forall b\leq 2s-1,\,\, \|x^{(b)}\| \leq  N^{b/2+s^{-1}\delta_{0}}\},$$
and the reduced domain in Fourier space:
$$\UN(\gamma_{0})=\{\xi \in \dkg  \,:\, \forall b\leq 2s-1,\,\,\|\xi_{|\kg^{(b)}}\| \leq  N^{-b/2+\gamma_{0}} \}.$$
At the beginning of \Cref{Sec-uniform-LLT}, we also introduced for $a\in [s]$, $g,h\in \tkg$,
$$\gamma_{a}= (32s)^{-a}\gamma_{0}\1_{a<s} \,\,\, \,\,\,  \,\,\, \,\,\, \,\,\,\,\,\,N_{a}=\lfloor N^{1-\gamma_{a}}\rfloor   \,\,\, \,\,\, \,\,\,\,\,\, \,\,\, \,\,\,\kappa_{N,a}=(T_{N_{a}} \tmu)^{\otimes N_{a}-N_{a-1}}$$
and $ \Upsilon^N_{g,h}=\Pi_{\star}(g\otimes \kappa_{N,1}\otimes \dots \otimes \kappa_{N,s}\otimes h) $.

\domainreduction

Given $a\in [s]$, we can write $\xx=(\xx'_{a}, \xx_{a})$ where $\xx'_{a}$ is the prefix of length $N_{a}$ in $\xx$. We will need the following lemma to guarantee that the main contribution to $\Pi^{[a,b)}\xx$ comes from $\Pi^{[a,b)}\xx_{a}$, i.e. from  the brackets that only involve variables from $\xx_{a}$. We also allow moderate deviations. The fact that $\Upsilon^N$ truncates  \emph{gradually} the variables in $\tmu^{* N}$ plays a role in the proof of this lemma.

\begin{lemme} \label{troncature}  Let $a\in\{1, \dots, s\}$, $b\geq a$, $N\geq 1$, $g,h\in \DN(\delta_{0})$, $m\geq1$,
$$\E_{\xx\sim\kappa_{N,1}\otimes \dots \otimes \kappa_{N,s} }\left[ \|\Pi^{[a,b)}(g,\xx,h) -\Pi^{[a,b)}(g,\xx_{a},h)\|^{2m} \right] \ll_{m} N^{m(b +2\delta_{0})}\left( \frac{N_{a}}{N}\right)^m$$
and the same holds if one replaces the variables $(g,\xx,h)$, $(g,\xx_{a}, h)$ by their projection  to $\km^{(1)}\oplus \R \chi$ modulo  $\kg^{[2]}$. 
\end{lemme}

\begin{proof} 
Note first that the map $\xx\mapsto \Pi^{[a,b)}(g,\xx,h)$ is a finite formal linear combination (depending on $a,b$ but not $N$) of statistics associated to  monomials of w-degree at most $b$, and such that each statistics of w-degree $b' \leq b$ has its coefficients depending on $g,h$ and bounded above by    $$O(1)\sum_{i_{1}+\dots+j_{l}=b-b' } \|g^{(i_{1})}\|\dots \|g^{(i_{k})}\| \,\|h^{(j_{1})}\|\dots \|h^{(j_{l})}\|$$
if $b'<b$, and $O(1)$ otherwise.

In view of the condition $g,h\in \DN(\delta_{0})$,  the proof of \Cref{troncature}  boils down to showing that for a statistics $M^\infty$ induced by a monomial $M:\tkg^t\rightarrow \R$ of  w-degree $b'$ (see \Cref{Sec-3reductions}),  we have
\begin{align} \label{tronc1}
\E_{\kappa_{N,1}\otimes \dots \otimes \kappa_{N,s} }\left[ |M^\infty(\xx) - M^{\infty}(\xx_{a})|^{2m} \right] \ll_{m} N^{mb'}\left( \frac{N_{a}}{N}\right)^m. 
\end{align}
Observe that the difference above can be rewritten as 
$$M^\infty(\xx) - M^{\infty}(\xx_{a})=\sum_{\substack{n_{1}<\dots <n_{t}\\ n_{1}\leq N_{a}}} M(x_{n_{1}},\dots, x_{n_{t}}).$$
Given $\alpha\in \N^{\dim\tkg}$,  denote by $x\mapsto x^\alpha$  the corresponding monomial on $\tkg$ (say for some ordering of the basis $e^{(i)}_{j}$ fixed in \Cref{Sec-cadre}) and write $d(\alpha)$ its w-degree. Expanding the power $2m$, we can write  for $N\geq 1$, $\xx\in \tkg^N$, 
  \begin{align*}
(M^\infty(\xx) - M^{\infty}(\xx_{a}))^{2m} &= \sum_{l=1}^{2mb} \sum_{\aalpha \in \mathfrak{A}^1_{l}}  \sum_{\nn\in \mathfrak{B}_{\aalpha}}  C_{\aalpha} x_{n_{1}}^{\alpha_{1}}\dots  x_{n_{l}}^{\alpha_{l}}
  \end{align*}
  where   $\sup_{\aalpha}|C_{\aalpha}| =O_{m}(1)$ and 
  $$\mathfrak{A}^{1}_{l} =\left\{\aalpha \in (\N^{\dim\tkg})^l \,:\, \sum_{i=1}^l d(\alpha_{i})=2mb \,\text{ and }\,\inf_{i}d(\alpha_{i})\geq 1 \right\}$$
    $$\mathfrak{B}_{\aalpha} =\left\{\nn \in [N]^l \,:\,  \sum_{i=1}^l d(\alpha_{i})\1_{n_{i}\leq N_{a}} \geq 2m \right\}.$$

Then by linearity and independence,
    \begin{align*}
\E_{\kappa_{N,1}\otimes \dots \otimes \kappa_{N,s} }\left[(M^\infty(\xx) - M^{\infty}(\xx_{a}))^{2m}\right] &= \sum_{l=1}^{2mb} \sum_{\aalpha \in \mathfrak{A}^1_{l}}  \sum_{\nn\in \mathfrak{B}_{\aalpha}}  C_{\aalpha}\E_{T_{N(n_{1})}\tmu}( x_{n_{1}}^{\alpha_{1}})\dots  \E_{T_{N(n_{l})}\tmu}(x_{n_{l}}^{\alpha_{l}})
  \end{align*}
where $N(n_{i}) = N_{j}$ if $n_{i} \in \{N_{j-1}+1, \dots, N_{j} \}$. 

The terms for which some $x_{i}^{\alpha_{i}}$ has w-degree $1$ satisfy $\E_{\tmu}(x_{i}^{\alpha_{i}})=0$, hence the $\aalpha$ that contribute belong to 
$$\mathfrak{A}^{2}_{l}=\left\{\aalpha \in \mathfrak{A}^{1}_{l} \,:\, \inf_{i}d(\alpha_{i})\geq 2 \right\}.$$
This forces $l$ to belong to $\{1, \dots, mb\}$:  
$$\E_{\kappa_{N,1}\otimes \dots \otimes \kappa_{N,s} }\left[(M^\infty(\xx) - M^{\infty}(\xx_{a}))^{2m}\right]= \sum_{l=1}^{mb} \sum_{\aalpha \in \mathfrak{A}^2_{l}}  \sum_{\nn\in \mathfrak{B}_{\aalpha}} C_{\aalpha}\E_{T_{N(n_{1})}\tmu}( x_{n_{1}}^{\alpha_{1}})\dots  \E_{T_{N(n_{l})}\tmu}(x_{n_{l}}^{\alpha_{l}}).$$
Now note that for $d(\alpha)\geq 2$, $N'\geq 1$,
$$\E_{T_{N'}\tmu}(|x^\alpha|)= \E_{\tmu}\left(|(T_{N'}x)^\frac{2\alpha}{d(\alpha)}|^{\frac{d(\alpha)}{2}-1} |(T_{N'}x)^\frac{2\alpha}{d(\alpha)}| \right)=O(N'^{\frac{d(\alpha)}{2}-1})$$
because $|T_{N'}x^\frac{2\alpha}{d(\alpha)}|=O(N')$ and  $ \E_{\tmu}(|(T_{N'}x)^\frac{2\alpha}{d(\alpha)}|)\leq  \E_{\tmu}(|x^\frac{2\alpha}{d(\alpha)}|)+o(1)$ is bounded independently of $N'$. In particular, for $n_{i}\leq N_{a}$ one has $\E_{T_{N(n_{i})}\tmu}( |x_{n_{i}}^{\alpha_{i}}|)\leq N_{a}^{\frac{d(\alpha_{i})}{2}-1}$, while for $n_{i}\geq N_{a}+1$ we have $\E_{T_{N(n_{i})}\tmu}( |x_{n_{i}}^{\alpha_{i}}|)\leq N^{\frac{d(\alpha_{i})}{2}-1}$. We deduce that for $l\leq mb$, $\aalpha \in \mathfrak{A}^2_{l}$, $n\in \mathfrak{B}_{\aalpha}$ one has 
$$\E_{T_{N(n_{1})}\tmu}( x_{n_{1}}^{\alpha_{1}})\dots  \E_{T_{N(n_{l})}\tmu}(x_{n_{l}}^{\alpha_{l}}) = O_{m}(1)N^{mb} \left(\frac{N_{a}}{N}\right)^{m} N^{-|\{i:n_{i}\leq N_{a}\}|}_{a}N^{-|\{i:n_{i}> N_{a}\}|}.$$
The result follows as the parameters $l, \aalpha$ belong to a bounded set independent of $N$, and because 
$\sum_{\nn\in [N]^l} N^{-|\{i:n_{i}\leq N_{a}\}|}_{a}N^{-|\{i:n_{i}> N_{a}\}|}=O_{l}(1)$.

\end{proof}

We now turn to the proof of \Cref{reduction}. Recall that for $a,b\geq 1$,  we denote by $\kg^{[a, b)}$  the subspace of $\kg$ spanned linearly by the brackets of the form $[x_{1}, [x_{2},\dots [x_{j-1}, x_{j}]\dots ]]$ with $j\geq a$, $x_{i}\in\kg$, and $j+| \{i\,:\, x_{i}=\Xm\}| \geq b$.

\begin{proof}[Proof of \Cref{reduction}]
Set 
$ \eps_{a}:=\frac{1}{8}\gamma_{a-1}+2s\gamma_{a}$.
We will use the property (whose verification is straightforward) that for small enough $\delta_{0}$  depending only on $\dim\kg, \gamma_{0}$, for all $1\leq a\leq s-1$, one has  
$$\frac{1}{4s}\eps_{a}> \gamma_{a} > 4(\eps_{a+1} +\delta_{0}),$$
 while for $a=s$ we still have $\frac{1}{4s}\eps_{s}> \gamma_{s}=0$. From now on we fix such $\delta_{0}$. 
 Set
$$\UN' =\{\xi \in \dkg : \,\forall a\in [s], \forall b \in \{a, \dots, 2a-1\},\,\|\xi_{|\kg^{[a,b)}}\|\leq N^{-\frac{b}{2}+\eps_{a}}\}$$
and note that $\UN' \subseteq \UN(\gamma_{0})$. Indeed, this follows from the inequality  $\sup_{a}\eps_{a}\leq \gamma_{0}$ and the fact that every $\kg^{(b)}$ coincides with  $\kg^{[a, b)}$ for some $a\geq 1$ such that $a\leq b\leq 2a-1$ (by \Cref{basics-weight-filtration}). It is thus sufficient to check \Cref{reduction} where  $\UN(\gamma_{0})$ is replaced by $\UN' $

Let $\xi \in \dkg$ such that $\xi \notin \UN'$. 
We must  bound the modulus of 
$$\huNgh(\xi)=\E_{\yy \sim g\otimes \kappa_{N,1}\otimes\dots \otimes \kappa_{N,s} \otimes h}\left[  e_{\xi}(\Pi\yy)\right]$$ 
uniformly in $g,h \in \DN(\delta_{0})$.

From now on we let $a\in [s]$ be maximal such that  $\|\xi_{|\kg^{[a,b)}}\|> N^{-\frac{b}{2}+\eps_{a}}$ for some $b \in \{a, \dots, 2a-1\}$. 
Decompose the product $\Pi \yy$ as
$$\Pi \yy= S \yy + E \yy$$
where
$$S \yy= \Pi^{[<a]}\yy+\Pi^{[a]}\yy_{ab}  \,\,\,\,\,\,\text{ and }\,\,\,\,\,\,E \yy= \Pi^{[a]}\yy-\Pi^{[a]}\yy_{ab}+\Pi^{[>a]}\yy$$
and as above $\yy_{ab}$ stands for  the coordinate-wise projection of $\yy$ to $\km^{(1)}\oplus \R\chi$ modulo $\kg^{[2]}$. 
An important feature of that decomposition is that 
$$E \yy=\sum_{b'\in \{a+1, \dots, 2s-1\}} E^{(b')}\yy$$
where  $E^{(b')}$ refers to the homogeneous component of $E$ with $w$-degree $b'$ (see \eqref{homodecompf}) and each $E^{(b')}$ takes values in $\kg^{[q(b'), b')}$ where $q(b')$ is an integer such that $a+1\leq q(b')\leq s$ and $q(b')\leq b'\leq 2q(b')-1$. We postpone the proof  to \Cref{lemmeE} below.


\bigskip

Our first step is to get partially rid of the term $E\yy$ in order to apply \Cref{choisirk}. To do so, write the variable $\yy$ as $\yy=(g,\xx,h)$ where $\xx$ is a variable in $\tkg^N$ with law $\kappa_{N,1}\otimes\dots \otimes \kappa_{N,s}$. Write  $\xx=(\xx'_{a}, \xx_{a})$ where $\xx'_{a}$ has length $N_{a}:=\lfloor N^{1-\gamma_{a}} \rfloor$, and set $\yy_{a}=(g,\xx_{a},h)$ the variable obtained by removing from $\yy$ the component $\xx'_{a}$.  
We also introduce a parameter $m\geq 1$ and set $T_{2m-1}(t)= \sum_{j=0}^{2m-1}\frac{(-2i\pi t)^j}{j!} $ the Taylor expansion up to order $2m-1$ of $t\mapsto \exp(-2i\pi t)$, so that
$$|T_{2m-1}(t)- \exp(-2i\pi t)| \leq \frac{(2\pi t)^{2m}}{(2m)!}.$$
 
We now perform the partial removal of $E$. Setting $\ukappaN=\kappa_{N,1}\otimes \dots\otimes \kappa_{N,s}$ we have
 \begin{align*}
\huNgh(\xi)
&= \E_{g\otimes \ukappaN \otimes h}\left[  e_{\xi}(S\yy) e_{\xi}(E\yy)\right]\\
&= \E_{g\otimes \ukappaN \otimes h}\left[   e_{\xi}(S\yy) e_{\xi}(E\yy_{a}) \exp\left(-2i\pi \xi(E\yy -E\yy_{a})\right)\right]\\
&= \E_{g\otimes \ukappaN \otimes h}\left[  e_{\xi}(S\yy) e_{\xi}(E\yy_{a}) T_{2m-1}(\xi E\yy -\xi E\yy_{a})\right]\\
&\,\,\,\,\,\,\,\,\,\,\,\,\,\,\,\,\,\,\,\,\,\,\,\,\,\,\,\,\,\,\,\,\,\,\,\,\,\,\,\,\,\,\,\,\,\,\,\,\,\,\,\,\,\,\,\,\,\,\,\,\,\,\,\,\,\,\,\,\,\,\,\,\,\,\,\,\,\,\,+O_{m}(1)\E_{g\otimes \ukappaN \otimes h}(\|\xi E\yy -\xi E\yy_{a}\|^{2m}).
 \end{align*}
Let us check that the error term is negligible for a good choice of $m$. Using the triangle inequality for the $L^{2m}$ norm,  the fact that $E^{(b')}$ takes values in $\kg^{[q(b'), b')}$ for all $a+1\leq b'\leq 2s-1$, the maximality of $a$, and  \Cref{troncature}, we may write 
 \begin{align*}
&E_{g\otimes \ukappaN \otimes h}(\|\xi E\yy -\xi E\yy_{a}\|^{2m})
\\
&\ll_{m} \sum_{a+1\leq b'\leq 2s-1} \E_{g\otimes \ukappaN \otimes h} (\| \xi  E^{(b')}\yy - \xi E^{(b')}\yy_{a}\|^{2m})\\
&\leq  \,\,\, \,\,\,  \sum_{a+1\leq b'\leq 2s-1} \|\xi_{|\kg^{[q(b'), b')}}\|^{2m} \E_{g\otimes \ukappaN \otimes h} (\|   E^{(b')}\yy - E^{(b')}\yy_{a}\|^{2m})\\
&\ll_{m} \,\sum_{a+1\leq b'\leq 2s-1} N^{-b'm+2\eps_{a+1}m} \,N^{(b'+2\delta_{0})m}\left( \frac{N_{a}}{N}\right)^m \\
&\ll \,\,\, \,\,\,  \,  N^{(2\eps_{a+1}+2\delta_{0} -\gamma_{a})m}  \\
&\leq \,\,\, \,\,\,  \,  N^{-\frac{m}{2}\gamma_{a}}  
 \end{align*}
because $\gamma_{a}\geq 4(\eps_{a+1}+\delta_{0})$.   Then we choose  $m=m(Q, \gamma_{a})$ so that $\frac{m}{2}\gamma_{a}>Q$. For this choice of parameter the error term  becomes negligible.  

\bigskip
The second step is now to bound the remaining expectation 
$$ \E_{g\otimes \ukappaN \otimes h}\left[  e_{\xi}(S\yy) e_{\xi}(E\yy_{a}) T_{2m-1}(\xi E\yy -\xi E\yy_{a})\right]$$
using the path-swapping technique on several sub-tuples of the variable $\xx'_{a}$ (via \Cref{choisirk}). Consider $g,h$ as fixed, and decompose
 $$T_{2m-1}(\xi E\yy -\xi E\yy_{a})= \sum_{M}c_{M}(g,h, \xi)M(\xx)$$
where $M$ varies in the set of monomials on $\tkg^{N}$ whose degree (in the classical sense) is at most  $(2m-1)s$, and $c_{M}(g,h, \xi)\in \C$. Recalling that $g,h\in \DN(\delta_{0})$ and $\delta_{0}\leq 1$, we see that  $\xi E\yy -\xi E\yy_{a}$ is a polynomial in $\xx$ with  coefficients bounded by $N^{C}\|\xi\|$  where $C>1$ is a constant depending only on $\dim \kg$. It follows that $|c_{M}(g,h, \xi)|\leq N^{C' m}(1+\|\xi\|)^{m}$ where $C'>1$ only depends on $\dim \kg$. 

 Consider one of these monomials $M$. We are going to apply the path-swapping technique to a sub-tuple of $\xx$  consisting of consecutive variables which do not appear in $\xx'_{a-1}$, $\xx_{a}$, nor $M(\xx)$. 
To do so, set $k=k(a,\xi)\geq 1$  the block length prescribed for $\xi$ by \Cref{choisirk}, and then $N'_{aa}:= \lfloor \frac{N_{a}-N_{a-1}}{2mska}\rfloor$.   
As the number of distinct variables in $M(\xx)$ is bounded by $(2m-1)s$, we may write  $\xx'_{a}=(\xx'_{a-1}, \zz_{0}, \zz, \zz_{1})$ where $\zz$ is a tuple of variables which do not appear in $M(\xx)$ and of length  $N_{aa}=ka N'_{aa}$. Then using Fubini's theorem, we write
 \begin{align*}
\E_{\yy\sim g\otimes \ukappaN \otimes h}\left[  e_{\xi}(S\yy) e_{\xi}(E\yy_{a}) M(\xx)\right] \,=\, \E_{\xx'_{a- 1},\zz_{0},\zz_{1}, \xx_{a}}\left[ \E_{\zz} [e_{\xi}(S\yy)] e_{\xi}(E\yy_{a}) M(\xx)\right].
 \end{align*}
   \Cref{choisirk} implies that 
  \begin{align*}
\big| \E_{\zz} [e_{\xi}(S\yy)] \big| & \ll (1-gap_{c}(\|\xi\|))^{\frac{1}{2^{a-1}} N'_{aa}}
   \end{align*}
 as long as $0<c\lll1$. To justify this last upper bound, we check that for $N$ large enough, the conditions of \Cref{choisirk} do hold. By definition, $k$ is the smallest integer such that there is ${b_0}\in \{a, \dots, 2a-1\}$ for which $k^{{b_0}/2}\|\xi_{|\kg^{[a,{b_0})}}\|\geq 1$. As  $\|\xi_{|\kg^{[a,b)}}\|> N^{-\frac{b}{2}+\eps_{a}}$ for some $b$, we must have $k\leq N^{1-\frac{2\eps_{a}}{b}}\leq N^{1-\frac{2\eps_{a}}{s}}$. In particular, this tells us that ${b_0}$ satisfies $\|\xi_{|\kg^{[a,{b_0})}}\|> N^{-\frac{{b_0}}{2}+\frac{{b_0}}{s}\eps_{a}}\geq N^{\frac{{b_0}}{s}\eps_{a}-\eps_{a+1}} \|\xi_{|\kg^{[a+1,{b_0})}}\|$. In this last inequality, we use the maximality of $a$, to guarantee that $\|\xi_{|\kg^{[a+1,{b_0})}}\|\leq N^{-{b_0}/2+\eps_{a+1}}$ which is valid even in the critical case where ${b_0}=a$ because then $\kg^{[a+1,{b_0})}= \kg^{[a+1,a+1)}$. This justifies $\|\xi_{|\kg^{[a,{b_0})}}\|\ggg \|\xi_{|\kg^{[a+1,{b_0})}}\|$ up to taking $N$ large enough in terms of $\eps_{a}, \eps_{a+1}$.
 Moreover, the upper bound on $k$ yields $N'_{aa}\geq N^{\frac{1}{s}\eps_{a}-2\gamma_{a}}> N^{\frac{1}{2s}\eps_{a}}$ due to  $\frac{1}{s}\eps_{a} >4 \gamma_{a}$, and up to assuming $N$ larger than a constant depending on $\gamma_{a-1}, \gamma_{a},m, s$. This discussion justifies the application of \Cref{choisirk}.  
   
We infer that
\begin{align*}
\huNgh(\xi)
& = \sum_{M} c_{M}\E_{g\otimes \ukappaN\otimes h}\left[  e_{\xi}(S\yy) e_{\xi}(E\yy_{a}) M(\xx)\right]  +O_{m}(N^{-Q})\\
& \ll_{\gamma_{a-1}, \gamma_{a}, m}(\sum_{M} |c_{M}|) \,  (1-gap_{c}(\|\xi\|))^{\frac{1}{2^{a-1}} N^{\frac{\eps_{a}}{2s}}}  + N^{-Q}.
 \end{align*}  
Using that the number of coefficients $c_{M}$ is bounded by $N^{2m \dim \kg}$, and that each of them is bounded by $N^{C'm}(1+\|\xi\|)^m$, we may write $\sum_{M} |c_{M}| \leq N^{C''m}(1+\|\xi\|)^m$ where $C''>1$ only depends on $\dim \kg$. Since the constant $m$ can be prescribed from the beginning as  functions of $(\gamma_{a})_{1\leq a\leq s}$, $Q$ and the initial data, and since  $(\gamma_{a})_{1\leq a\leq s}$ only depends on $\gamma_{0}$, we get

\begin{align*}
\huNgh(\xi)
& \ll_{\gamma_{0}, Q}  N^{C''m}(1+\|\xi\|)^m (1-gap_{c}(\|\xi\|))^{\frac{1}{2^{a-1}} N^{\frac{\eps_{a}}{2s}}}   +N^{-Q}.
  \end{align*}  
It follows that for some constant $L$ depending only on $(\dim \kg, \gamma_{0}, Q)$, 
we finally have the desired  bound  
  \begin{align*}
|\huNgh(\xi)| \ll_{\gamma_{0}, Q}  (1+\|\xi\|)^{L} gap_{c}(\|\xi\|)^{-L}N^{-Q}.
  \end{align*}
\end{proof}

We record the following lemma that we used in the proof above. 

\begin{lemme}\label{lemmeE}
Let $N\geq 1$, $a\in [s]$, and define a map $E:\tkg^N\rightarrow \tkg$ by
  $$E \yy= \Pi^{[a]}\yy-\Pi^{[a]}\yy_{ab}+\Pi^{[>a]}\yy.$$
Then $E \yy=\sum_{b'\in \{a+1, \dots, 2s-1\}} E^{(b')}\yy$, and each $E^{(b')}$ takes values in $\kg^{[q(b'), b')}$ where $q(b')$ is an integer such that $a+1\leq q(b')\leq s$ and $q(b')\leq b'\leq 2q(b')-1$.  
\end{lemme}

\begin{proof}
 Notice first that $E^{(b')}\yy$ is clearly zero if $b'\leq a$. Now let $b'\in \{a+1, \dots, 2s-1\}$. We can write $E^{(b')} \yy$ as a linear combination of iterated brackets of the form $L^{[j]}(\vv)$ where $j\geq a$, $v_{i}\in \tkm^{(b_{i})}$ with $\sum_{i}b_{i}=b'$, and at least one $v_{i}$ in $\kg^{[2]}$ if $j=a$. By \Cref{basics-weight-filtration}, we may write each $v_{i}$ as a combination of brackets $L^{[p]}(\underline{w})$ where $p+| \{n\,:\, w_{n}=\chi \}|\geq b_{i}$, and if $j=a$ and $v_{i}\in \kg^{[2]}$, we may also suppose $p\geq 2$. It follows that $L^{[j]}(\vv)$ is a combination of  brackets $L^{[q]}(\underline{u})$ where $ a+1\leq q$, and $q+ | \{n\,:\, u_{n}=\chi\}|\geq b'$ (in particular $2q-1\geq b'$). Take $q_{0}$ the minimal $q$ that appears in this decomposition. This justifies $E^{(b')} \yy\in \kg^{[q_{0}, b')}$ (here we use $[\chi,.]=[\Xm,.]$) with $a+1\leq q_{0}$ and $b'\leq 2q_{0}-1$. The claim follows by taking $q(b')=\min (b',q_{0})$. 
 \end{proof}


\subsection{Effective local limit theorem with power saving} \label{Sec-endproof}

We use the previous section to prove \Cref{TLLFouriercompact}, which is  a local limit theorem with explicit error term in the test function and better approximation rate than the LLT of \Cref{LLT} or \Cref{LLT-tronc}.  The price to pay is that the test function must be regular enough depending on $\mu$.

\TLLFouriercompact

Recall the notation $\|f\|_{\mu_{ab}, c, L}$ has been defined in \eqref{norm-mucL}.

\begin{proof}[Proof of \Cref{TLLFouriercompact}]
We may assume $\|f\|_{\mu_{ab}, c, L}<\infty$, in particular $\hf$ is integrable.  As we observed in the proof overview at the beginning of \Cref{Sec-uniform-LLT}, the Fourier inversion formula implies that for   $N\geq 1$, $g,h \in \tkg$, 
\begin{align*}
\thetaNgh (f)&=\int_{\dkg} \hf(\xi)\, \huNgh(\xi) e_{\xi,N,g,h} \,d\xi 
\end{align*}
where $e_{\xi,N,g,h}= e^{-2i\pi \xi\circ p(g^{(\chi)}+N\chi+h^{(\chi)})}$, and  $g^{(\chi)}, h^{(\chi)}$ are the  $\R\chi$ components of $g,h$ in the decomposition $(\ref{tilde-dec})$.

We fix $\delta_{0}=\delta_{0}(\dim \kg, \gamma_{0})>0$  as in  \Cref{reduction}.   
Supposing $g,h\in \DN(\delta_{0})$,  \Cref{reduction} applied to $Q=\frac{1}{2}(\dd+1)$, then yields
\begin{align*}
\thetaNgh (f)&=\int_{\UN(\gamma_{0})} \hf(\xi)\, \huNgh(\xi) e_{\xi,N,g,h}\,d\xi \, +\, \|f\|_{\mu_{ab}, c, L}O_{\gamma_{0}}\left(N^{-\frac{1}{2}(\dd+1)}\right).
\end{align*}
for any $0<c\lll 1$, and some $L>1$ depending only on $(\dim \kg, \gamma_{0})$ which we may assume greater than $\delta^{-1}_{0}$. 

By \cite[Lemma 5.1]{benard-breuillard-CLT}, there exists a smooth probability measure $\eta$ on $\kg$, which has finite moment of all orders, coincides with $\mu$ up to order $2$, and is linked to the limit measure $\nu$ in the central limit theorem via the relation $\teta^{*'N}=\DilN \nu*'N\chi$ for all $N$.
On the reduced domain  $\UN(\gamma_{0})$, we successively apply  Propositions \ref{grading-Fourier}, \ref{gaussian} to $\eta$, \ref{truncation-cost} to $\eta$, in order to replace $\uNgh$ by $g*'\DilN \nu*' N\chi*'h $, or even by $(\DilN \nu * N\chi)_{g,h}:=g*\DilN \nu* N\chi*h $ thanks to \Cref{remplacement*} below. Recalling the bound $|\hf|\leq \|f\|_{L^1}$, this gives for some $C>1$ depending only on $\dim \kg$,
\begin{align*}
&\thetaNgh (f)+\|f\|_{\mu_{ab}, c, L}O_{\gamma_{0}}\left(N^{ -\frac{1}{2}(\dd +\min(1, m_{\mu}- 2))+ C\gamma_{0} } \right)\\
&=\int_{\UN(\gamma_{0})} \hf(\xi)\, \widehat{(\DilN \nu * N\chi)_{g,h}}(\xi)\, e_{\xi,N,g,h}\,d\xi  \\
&=\int_{\UN(\gamma_{0})} \hf(\xi)\, \widehat{(\DilN \nu * N\Xm)_{p(g),p(h)}}(\xi)\,d\xi.
\end{align*}

 To make the term $(\DilN \nu * N\Xm)_{p(g),p(h)}(f)$ appear, we must check that the integral on the complementary domain $\dkg \smallsetminus \UN(\gamma_{0})$ is also negligible. This is proven in \Cref{compl} below.

Once  \Cref{compl} is established, Fourier inversion formula and the inequality $|\hf | \leq \|f\|_{L^1}$  imply that 
$$|\thetaNgh (f)-(\DilN \nu*N\Xm)_{p(g),p(h)}(f)| =  \|f\|_{\mu_{ab}, c, L }O_{\gamma_{0}}\left(N^{ -\frac{1}{2}(\dd +\min(1, m_{\mu}- 2))+ C\gamma_{0}}  \right)$$
 which concludes the proof.
\end{proof}

For the proof of \Cref{TLLFouriercompact} to be complete we need to check the next two lemmas. 
\begin{lemme}[Back to $*$] \label{remplacement*} For $N\geq 1$, $\gamma_{0}>\delta_{0}\geq 0$, $g,h\in \DN(\delta_{0})$, $\xi \in  \UN(\gamma_{0})$, we have
$$\big|   \widehat{(g*\DilN \nu*N\chi*h)}(\xi)    -   \widehat{(g*'\DilN \nu*'N\chi*'h)}(\xi)      \big|  \ll_{\gamma_{0}}  N^{-1/2+2\gamma_{0}}.$$
\end{lemme}
\begin{proof}
 Set  $\yy=(g,\DilN (x), N\chi, h)$ where $x\in \tkg$, and observe that
\begin{align*}
&\big|   \widehat{(g*\DilN \nu*N\chi*h)}(\xi)    -   \widehat{(g*'\DilN \nu*'N\chi*'h)}(\xi)      \big|  \\
&= \big|\E_{x\sim \nu}[ e_{\xi}(\Pi\yy)-e_{\xi}(\Pi'\yy) ]\big|\\
&\ll \E_{x\sim \nu}[|\xi \Pi \yy- \xi\Pi'\yy ]\big|\\
 &\ll \sum_{b\leq 2s-1}\|\xi_{|\kg^{(b+1)}}\|  \,\E_{x\sim \nu}\left(\|\Pi^{(b)}\yy\|\right).
 \end{align*}
where the last line uses the triangle inequality and the fact that $\Pi^{(b)}\yy - \Pi'^{(b)}\yy$ is the projection of $\Pi^{(b)}\yy $ to $\tkg^{(b+1)}$ modulo $\tkm^{(b)}$.

Hence, it is sufficient  to show that
$$\E_{x\sim \nu} \left[\|\Pi^{(b)}\yy \| \right]  \ll_{\nu} N^{b/2+\delta_{0}}.$$
This bound follows from the equality $\Pi^{(b)}\yy= N^{b/2} \Pi^{(b)}\zz$ where $\zz=(\DilsN g, x, \chi, \DilsN h)$ and the remark that  $\|\Pi^{(b)}\zz\|\ll N^{\delta_{0}} P(x)$ where $P$ is some polynomial.

\end{proof}

\begin{lemme}[Domain completion] \label{compl} Up to choosing $\delta_{0}>0$ smaller than a constant depending on $\gamma_{0}$ and $\dim \kg$, we have for $N\geq 1$,  $g,h\in \DN(\delta_{0})$, 
$$ \int_{\dkg\smallsetminus \UN(\gamma_{0})} |\widehat{(\DilN \nu * N\chi)_{g,h}}(\xi)| \,d\xi \,=\,O_{\gamma_{0}}(N^{-\frac{\dd +1}{2}}).$$
\end{lemme}

\begin{proof}
 Note that for $N\ggg1$, we have $N\chi *h\in \DN(C\delta_{0})$ for some $C>1$ only depending on $\dim \kg$. Hence  it is enough to check that for $\delta_{0}$ small enough in terms of $\gamma_{0}$ and $\dim \kg$, for $N\geq 1$, $g,h\in \DN(\delta_{0})$, we have
 \begin{align*}
 \int_{\dkg\smallsetminus \UN(\gamma_{0})} |\widehat{(\DilN \nu)_{g,h}}(\xi)| d\xi  =O_{\gamma_{0}}(N^{-\frac{\dd+1}{2}}).
\end{align*}
Replacing the variable $\xi$ by $\DilN \xi:=\xi\circ \DilN $, this can be rewritten: 
 \begin{align*}
 \int_{\dkg\smallsetminus \DilN \UN(\gamma_{0})} |\widehat{\sigma_{N,g,h}}(\xi)| d\xi  =O_{\gamma_{0}}(N^{-1/2}) 
\end{align*}
where $\sigma_{N,g,h}:= p_{\kg}\left(\DilsN  ( g * \DilN \nu *h)\right)$, with $p_{\kg}: \tkg\rightarrow \kg, x\oplus t\chi\mapsto x$. Moreover,  a direct computation yields the inclusion $B_{\dkg}(0,N^{\gamma_{0}})\subseteq {\DilN }\UN(\gamma_{0})$. Hence, it is sufficient to check that 
 \begin{align*}
 \int_{\dkg\smallsetminus B_{\dkg}(0,N^{\gamma_{0}})} |\widehat{\sigma_{N,g,h}}(\xi)| d\xi  = O_{\gamma_{0}}(N^{-1/2}). 
\end{align*}


We now give an explicit formula for the Radon-Nikodym derivative of $\sigma_{N,g,h}$ with respect to $dx$. By \cite[Proposition 3.8]{benard-breuillard-CLT}, we can write $\nu=v(x) dx$ where $v : \kg\rightarrow \R_{+}$ is a smooth function whose derivatives (at all orders) decay rapidly. 
For $y\in \kg$, we also write $L_{y}, R_{y}, S_{y}: \kg \rightarrow \kg$ for the operators of left multiplication, right multiplication, or sum, by $y$, i.e. $L_{y}(x)=y*x$, $R_{y}=x*y$, $S_{y}(x)=x+y$. A direct computation gives
$$\sigma_{N,g,h}= v_{N,g,h}(x)dx$$
where $v_{N,g,h}= v\circ s_{N,g,h}$ with 
$$s_{N,g,h}= \DilsN  \circ L_{-p(g)} \circ R_{-p(h)} \circ S_{p(g^{(\chi)}+h^{(\chi)})} \circ \DilN. $$

Let $e_{1}, \dots, e_{d}$ ($d=\dim\kg$) be an enumeration of the basis $(e^{(i)}_{j})$ of $\kg$ chosen in \Cref{Sec-cadre}. For each $i\in \{1, \dots, d\}$, denote by $\partial_{i}$ the derivation on $\kg$ with respect to the vector $e_{i}$ (for the addition).  Using integration by parts, we see that for all $k\geq 0$, $\xi \in \dkg\smallsetminus \{0\}$, 
$$|\widehat{v_{N,g,h}}(\xi)| \ll_{k} \frac{1}{\|\xi\|^k} \sum_{i=1}^d \|\partial^k_{i} v_{N,g,h} \|_{L^1}.$$
Hence, as $ \widehat{\sigma_{N,g,h}} = \widehat{v_{N,g,h}}$, 
 \begin{align*}
 \int_{\dkg\smallsetminus B_{\dkg}(0,N^{\gamma_{0}})} |\widehat{\sigma_{N,g,h}}(\xi)| d\xi & \,\ll_{k}\,  \int_{\dkg\smallsetminus B_{\dkg}(0,N^{\gamma_{0}})} \frac{1}{\|\xi\|^k} d\xi \,\sum_{i=1}^d \|\partial^k_{i} v_{N,g,h} \|_{L^1}\\
 & \,\ll\, \frac{1}{N^{\gamma_{0}(k-d)}}\,\sum_{i=1}^d \|\partial^k_{i} v_{N,g,h} \|_{L^1}.
\end{align*}
To conclude, it is enough to show that for all $\delta_{0}\lll\gamma_{0}$, there exists $k=k(\gamma_{0})$  (large) such that for  $N\geq1$, $g,h\in \DN(\delta_{0})$,  
$$\sum_{i=1}^d \|\partial^k_{i} v_{N,g,h} \|_{L^1}=O_{k} (N^{\gamma_{0}(k-d) -1/2}).$$

In order to estimate the $k$-th derivative of $v_{N,g,h}$ and its $L^1$-norm, we will need more information on the term $s_{N,g,h}$. The crucial observation is that $s_{N,g,h}(x)$ can be realised by dividing by a power of $\sqrt{N}$ some coefficients of a polynomial $P$ in the variables $(\DilsN g, \DilsN h, x)$ and of total degree (in the classical sense) at most  $s$. To justify this, observe first that one can simply delete the projections $p(\cdot)$ in the definition of $s_{N,g,h}$, i.e.  $s_{N,g,h}(x)=\DilsN \left( (-g)* (\DilN (x)+ g^{(\chi)}+h^{(\chi)}) *(-h) \right)$. Then apply Baker-Campbell-Hausdorff formula to write the product as a linear combination of iterated brackets of length at most $s$, and note that for any $a,b\geq 1$, any bracket $L^{[a]}(\vv)$ with $v_{i}\in \tkm^{(b_{i})}$, $\sum b_{i}=b$, we have 
$$D_{\frac{1}{\sqrt{N}}}L^{[a]}(\vv)= (N^{b/2} D_{\frac{1}{\sqrt{N}} | \kg^{(b)}})L^{[a]}(D_{\frac{1}{\sqrt{N}}}\vv)$$
 where $N^{b/2}D_{\frac{1}{\sqrt{N}} |\kg^{(b)} }$ is a diagonalizable endomorphism of $\tkg^{(b)}$ with eigenvalues given by powers of $1/\sqrt{N}$. 

The polynomial $P$ that we obtain only depends on the initial data fixed in \Cref{Sec-cadre}.  Hence, as $\DilsN \DN(\delta_{0}) \subseteq B_{\tkg}(0, O(N^{s^{-1}\delta_{0}}))$,  we have
$$\|s_{N,g,h}(x)\| \ll N^{\delta_{0}}\max(1, \|x\|^s).$$ 
All this also holds for the inverse transformation  $s_{N,g,h}^{-1} : \kg \rightarrow \kg$, and in particular, 
$$\|s_{N,g,h}^{-1}(x)\| \ll N^{\delta_{0}}\max(1, \|x\|^s).$$

We now prove the bound for $\|\partial^k_{i} v_{N,g,h} \|_{L^1}$ that we announced, for a fixed subscript $i\in \{1, \dots, d\}$.  Given a $d$-tuple $\alpha=(\alpha_{1}, \dots, \alpha_{d})\in \N^d$, write $|\alpha|=\sum_{j}|\alpha_{j}|$ and $\partial^{\alpha}= \partial^{\alpha_{1}}_{1}\dots \partial^{\alpha_{d}}_{d}$.  Arguing by induction, it is easy to check that for any $k\geq 0$, one has
$$ \partial^k_{i} v_{N,g,h}(x)  =\sum_{|\alpha|\leq k} P^{(\alpha)}_{N,g,h}(x) \,(\partial^{\alpha} v) \circ s_{N,g,h}(x) 
$$
where $P^{(\alpha)}_{N,g,h}(x)$ can be realised by dividing by a power of $\sqrt{N}$ some coefficients of a polynomial $P^{(\alpha)}$ in the variables $(\DilsN g, \DilsN h, x)$ and of total degree (in the classical sense) bounded above by $ks$. Moreover the coefficients of $P^{(\alpha)}$ are determined by the initial data fixed in \Cref{Sec-cadre}, and  $i,k, \alpha$. We infer an upper bound for the $L^1$-norm
 \begin{align*}
 \|\partial^k_{i} v_{N,g,h} \|_{L^1} &\leq  \sum_{|\alpha|\leq k} \int_{\kg} |P^{(\alpha)}_{N,g,h}(x) \,(\partial^{\alpha} v) \circ s_{N,g,h}(x)| \,dx\\
 &\ll_{i, k}  N^{\delta_{0}k}\sum_{|\alpha|\leq k} \int_{\kg}  \max (1, \|x\|^{ks}) \,|(\partial^{\alpha} v) \circ s_{N,g,h}(x)| \,dx\\
  &=  N^{\delta_{0}k}\sum_{|\alpha|\leq k} \int_{\kg}  \max (1, \|s_{N,g,h}^{-1}(x)\|^{ks}) \,|\partial^{\alpha} v (x)| \,dx
 \end{align*}
where the last line uses that $s_{N,g,h}$ preserves the Lebesgue measure on $\kg$. From the above-mentioned bound on $s_{N,g,h}^{-1}(x)$, we conclude that
 \begin{align*}
 \|\partial^k_{i} v_{N,g,h} \|_{L^1} 
 &\ll_{i, k}  
 N^{\delta_{0}k(1+s)}\sum_{|\alpha|\leq k} \int_{\kg}  \max (1, \|x\|^{ks^2}) \,|\partial^{\alpha} v (x)| \,dx.
 \end{align*}
Since the integral on the right-hand side is bounded, we may sum over  $i$ and conclude that for all $k\geq 0$, $N\geq 1$, $g,h\in \DN(\delta_{0})$, 
 \begin{align*}
\sum_{i=1}^d \|\partial^k_{i} v_{N,g,h} \|_{L^1}&=O_{k} (N^{\delta_{0}k(1+s)})\\
&=O_{k} (N^{\gamma_{0}(k-d) -1/2})
 \end{align*}
up to choosing $\delta_{0}<\frac{\gamma_{0}}{2(1+s)}$  and $k=k(\gamma_{0})$ large enough so that  $\delta_{0}k(1+s)<\gamma_{0}(k-d) -1/2$.
This finishes the proof of \Cref{compl}. 
\end{proof}

\subsection{Approximation by band-limited functions}

To convert \Cref{TLLFouriercompact} into a local limit theorem dealing with any compactly supported continuous function (\Cref{TLLfcompact}), we will need the following approximation result. It is certainly well-known, but in the absence of a convenient reference, we include a proof. 

\begin{proposition} \label{approx}
Let $f\in C_{c}(\kg)$. 
For all $\eps>0$, there exist $f^-, f^+\in C^\infty(\kg)$ which are integrable, bounded, and satisfy
\begin{itemize}
\item[$\bullet$] $f^-\leq f\leq f^+$.
\item[$\bullet$] $\widehat{f^-}$ and $\widehat{f^+}$ are compactly supported.
\item[$\bullet$] $\|f^+ - f^-\|_{L^1} <\eps$.
\end{itemize}
\end{proposition}



\bigskip

Let $\mathscr{C}$ be the cone of functions $\rho\in C^\infty(\kg)$ such that  $\rho\geq 0$, $\rho \in L^1(\kg)$ and $\hrho$ is compactly supported. We first show that $\mathscr{C}$ is non trivial, and contains  some functions decreasing fast, others decreasing slowly. 

\begin{lemme} \label{LL}
\begin{enumerate}[label=\roman*)]
\item For every $k\geq 2$, there is $\rho\in \mathscr{C}$ with $\|\rho\|_{L^1}=1$ and
$$0<\rho(x)\leq 1/\|x\|^k.$$
\item For all $\eps>0$, there is $\tau\in \mathscr{C}$ such that $\|\tau\|_{L^1}<\eps$ and 
$$\tau(x)\geq 1/\|x\|^{2 \dim \kg} $$
outside of a compact set. 
\end{enumerate}
\end{lemme}

\begin{proof}
Let $\kg \simeq \R^d$. It is enough to prove the result for $d=1$, since we may then consider the product $f(x_1)\cdot \ldots \cdot f(x_d)$ over distinct variables, where $f=\rho$ or $\tau$.  Let $(a_{n})_{n\in \Z}\in (\R_{\geq 0})^\Z$ be non identically zero with finite total sum. Let $p\geq 1$. We define the function  $\phi : \R\rightarrow \R$ by setting 
$$\phi(x)= \sin^{2p}(\pi x)\sum_{n\in \Z}\frac{a_{n}}{(x-n)^{2p}}.$$
Then $\phi$ is smooth, $\phi \geq 0$, $\|\phi\|_{L^1}\leq C_{p}\sum_{n\in \Z} a_{n}$ for some $C_{p}>0$. The Paley-Wiener theorem ensures that $\widehat{\phi}$ has compact support. Hence $\phi\in \mathscr{C}$. 
For the first item, we choose $p=k/2$, $a_{n}=\1_{n=0}$, and we set 
$$\rho_{1}(x) =\phi(x) +\phi(x+1/2)    \,\,\,\,\, \,\,\,\,\, \,\,\,\,\, \,\,\,\,\, \,\,\,\,\, \,\,\,\,\,\rho= \frac{1}{\|\rho_1\|_{L^1}}\rho_1.$$
For the second item, we first note that $\phi_{|[n-\frac{1}{2}, n+\frac{1}{2}]}\geq c_{p}a_{n}$  for some constant $c_{p}>0$. We choose $p=1$, $a_{n}=c^{-1}_{1}|n|^{-3/2} \1_{|n|\geq R}$ where $R>0$ is large enough so that
$\|\phi\|_{L^1}\leq C_{1}\sum_{n\in \Z} a_{n}<\eps$. We can then take $\tau=\phi$.

\end{proof}

\begin{proof}[Proof of \Cref{approx}]
We may assume $f\geq 0$ and non zero. We fix a Euclidean ball $K=B(0,R_0)$  of $\kg$ containing the support of $f$, and a positive integrable function $\rho \in \mathscr{C}$ as in the first item of \Cref{LL} for $k=2(\dim \kg +1)$. Given $t>0$, we set $\rho_{t}(x)=t^{\dim \kg}\rho(tx)$ the  $t$-contraction of $\rho$. We may choose $t$ large enough $(t=t(f, \eps))$ so that
\begin{itemize}
  \item[$\bullet$]  $ f-\eps \leq f* \rho_{t} \leq f+\eps $ on $K$
 \item[$\bullet$] $ \|(f*\rho_{t}) \1_{\kg\smallsetminus K}\|_{L^1}\leq \eps$.
\end{itemize}

We are going to show the existence of $\tau_{1}, \tau_{2}\in \mathscr{C}$ such that  $\tau_{1}> \eps \1_{K}$, $\tau_{2}> (f*\rho_{t}) \1_{\kg\smallsetminus K}$, and  $\|\tau_{i}\|_{L^1} = O_{K}(\eps)$. We then set 
$$f^-=f*\rho_{t} - \tau_{1}-\tau_{2}\,\,\,\, \,\, \,\, \,\, \,\, \,\, \,\, \,\, \,\, \,\, \,\, \,\, \,\,  f^+=f*\rho_{t} +\tau_{1}$$
to obtain \Cref{approx}.

The component $\tau_{1}$ can be constructed by fixing at the beginning of the proof  $\phi\in \mathscr{C}$ such that $\phi>\1_{K}$, then setting $\tau_{1}=\eps \phi$.

For the component $\tau_{2}$, we start by observing that $f*\rho_{t}$ decays rapidly at infinity: 
\begin{align*}
f*\rho_{t}(x)= &\int_{\kg} \rho_{t}(x+s) f(-s)ds\\
&\leq \|f\|_{L^\infty} \sup_{x+K}\rho_{t}\\
&= t^{\dim \kg}\|f\|_{L^\infty} \sup_{tx+tK}\rho.
\end{align*}

There exists $C=C(f,K,\rho, t)>R$ such that for $\|x\|\geq C$, we have 
$$\sup_{tx+tK}\rho \leq \frac{1}{ t^{\dim \kg}\|f\|_{L^\infty} \|x\|^{2\dim \kg +1}}$$
which implies for $\|x\|\geq C$,
\begin{align*}
f*\rho_{t}(x)\leq   \|x\|^{-2\dim \kg -1}.
\end{align*}

By \Cref{LL} (second item), there exists $\tau\in \mathscr{C}$ such that $\|\tau\|_{L^1}< \eps$ and $\tau > f*\rho_{t}(x)$ outside of a compact set $K_1=K_1(f,K,\rho,t) \supseteq K$. It remains to  upper bound $f*\rho_{t}$ in the intermediate  zone  $L:=K_1\smallsetminus K$. To do so, we fix $\phi'\in \mathscr{C}$ such that $\phi'\geq \1_{L}$, and observe that for  large enough $t'=t'(f, \rho, t, L, \phi')$
$$(f*\rho_{t})_{|L}< (f*\rho_{t})_{|L}*\rho_{t'}+\eps \frac{1}{\|\phi'\|_{L^1}} \phi'.$$
We conclude the construction of $\tau_{2}$ by setting 
$$ \tau_{2}= \tau + (f*\rho_{t})_{|L}*\rho_{t'}+\eps \frac{1}{\|\phi'\|_{L^1}} \phi'.$$
\end{proof}

\subsection{End of proof}

We   combine  Propositions \ref{TLLFouriercompact} and \ref{approx} to deduce Theorems \ref{LLT-tronc} and  \ref{LLT}.

\bigskip
We start with the case of moderate deviations:
\TLLfcompact

\begin{proof}
Fix $\gamma_{0}>0$ small enough depending only on $(\dim \kg, \mu)$  so that the  $C\gamma_{0}$ appearing as an exponent in the bound of  \Cref{TLLFouriercompact} is  bounded by  $\frac{1}{4}\min(1, m_{\mu}-2)$. Therefore, the total error term in this proposition can be bounded by $\|f\|_{\mu_{ab}, c, L}o_{\gamma_{0}}(N^{-\frac{\dd}{2}})$. Choose $\delta_{0}=\delta_{0}(\dim \kg, \gamma_{0})>0$ so that we may apply \Cref{TLLFouriercompact}.

Let $f\in C_{c}(\kg)$. For all  $\eps>0$, we fix $f^-\leq f\leq f^+$ some  $\eps$-approximations of $f$ as in \Cref{approx}. As the Fourier transforms of $f^\pm$ are compactly supported, we may write $\|f^\pm\|_{\mu_{ab}, c, L}=O_{f, \eps}(1)$. Given $g,h\in \DN(\delta_{0})$, we infer from  \Cref{TLLFouriercompact} that
\begin{align*}
&\thetaNgh(f)- (\DilN \nu*N\Xm)_{p(g)p(h)}(f)\\
& \leq \thetaNgh(f^+)- (\DilN \nu*N\Xm)_{p(g),p(h)}(f^-)\\ 
&= (\DilN \nu*N\Xm)_{p(g),p(h)}(f^+)- (\DilN \nu*N\Xm)_{p(g),p(h)}(f^-) +o_{f, \gamma_{0}, \eps}(N^{-\frac{\dd}{2} })\\
&= (\DilN \nu)(\psi) +o_{f,  \gamma_{0},\eps}(N^{-\frac{\dd}{2} })
\end{align*}
where $\psi(x)= (f^+ - f^-)(p(g)*x*N\Xm* p(h))$. Fourier inversion formula yields 
 \begin{align*}
 (\DilN \nu)(\psi)&= \int_{\dkg}\,\hpsi( \xi) \,\hnu(\DilN \xi)\, d\xi\\
&= N^{-\frac{\dd}{2}}\int_{\dkg}\hpsi(\DilsN  \xi) \,\hnu(\xi)\,  d\xi\\
&\leq N^{-\frac{\dd}{2}} \|\psi\|_{L^1}\|\hnu\|_{L^1}. 
 \end{align*}
As $$\|\psi\|_{L^1}= \|f^+ -f^-\|_{L^1}\leq \eps $$ 
we must have 
\begin{align*}
\thetaNgh(f)- (\DilN \nu*N\Xm)_{p(g),p(h)}(f) & \leq N^{-\frac{\dd}{2}} \eps \|\hnu\|_{L^1}   +o_{f, \gamma_{0}, \eps}(N^{-\frac{\dd}{2} })\end{align*}

One can argue in the same way, reversing the roles of  $f^+$ and $f^-$, to obtain a lower bound, so in the end
\begin{align*}
\big|  \thetaNgh(f)- (\DilN \nu*N\Xm)_{p(g),p(h)}(f) \big| & \leq N^{-\frac{\dd}{2}} \eps \|\hnu\|_{L^1}    +o_{f,  \gamma_{0},\eps}(N^{-\frac{\dd}{2} }).
\end{align*}

This implies
$$\limsup_{N\to +\infty} \sup_{g,h\in \DN(\delta_{0})}N^{\frac{\dd}{2}} |\thetaNgh (f)- (\DilN \nu*N\Xm)_{p(g),p(h)}(f)| \leq \eps \|\hnu\|_{L^1}.$$
As $\eps$ is arbitrary, the limit must be zero, which is exactly  \Cref{TLLfcompact}. 
\end{proof}

\bigskip

To conclude the proof of \Cref{LLT-tronc}, it is now sufficient to check that when exactly one of the deviations is non moderate, the terms  $\Theta^{N}_{p(g),p(h)}(f)$ and $(\DilN \nu*N\chi)_{p(g),p(h)}(f)$ are both negligible compared to $N^{-\dd/2}$, and this holds uniformly in $g,h$. This is the content of the next lemma.

\toutedeviation

\begin{proof}[Proof of \Cref{toutedeviation}]
Let $K\subseteq \kg$ be a compact set containing the support of $f$. We may replace $f$ without loss of generality by $\1_{K}$. For any $g,h\in \tkg$, one has
\begin{align*}
\Theta^{N}_{p(g),p(h)}(K)=\Upsilon^{N}_{g,h}(p^{-1}(K))=\Upsilon^N((-g)*p^{-1}(K)*(-h)).
\end{align*}

Let $\delta_{0}>\delta_{1}>0$ and assume $p(g)\notin p(\DN(\delta_{0}))$, $h\in \DN(\delta_{1})$. If $\delta_{1}\lll \delta_{0}$, then for large $N$ (depending on $K$, $\delta_{0}$), we have 
$$K* p(\DN(\delta_{1}))*p(\DN(\delta_{1}))  \subseteq p(\DN(\delta_{0}))$$
and this implies
$$(-g)*p^{-1}(K)*(-h)\subseteq \tkg\smallsetminus \DN(\delta_{1}).$$
Indeed, we would otherwise have $k'*(-h)*x=g$ where $k'\in p^{-1}(K)$, $x\in \DN(\delta_{1})$, hence $K* p(\DN(\delta_{1}))^{*2}\ni p(g)$, which contradicts our assumption on $g$. 

We deduce that for $N$ large, any $g,h$ as above, 
$$\Theta^{N}_{p(g),p(h)}(K) \leq\Upsilon^N(\tkg\smallsetminus \DN(\delta_{1})).$$
By the moment estimate \Cref{momentPi}, it follows that $\Theta^{N}_{p(g),p(h)}(K)$ must be bounded by $o_{K, \delta_{0}, \delta_{1}}(N^{-\frac{\dd}{2}})$. The other cases are similar.

\end{proof}

\begin{proof}[Proof of \Cref{LLT-tronc}]
Combine \Cref{TLLfcompact} and \Cref{toutedeviation}.
\end{proof}


\section{Local limit theorem for ratios } \label{Sec-LLT-ratio}


In this section we prove \Cref{mu/Leb} and \Cref{mu/nu}.
We keep the notations $(\kg, *, \tkg, \Xab, \dd, (\km^{(b)})_{b\leq 2s-1}, \Xm, \chi, *')$ from   \Cref{Sec-cadre}. We assume that $\mu$ is an aperiodic probability measure on $\kg$ with finite $(\dd+2)$-th moment for the weight filtration induced by $\Xab$, and expectation $\E(\mu_{ab})=\Xab$. In particular $X=\E(\mu^{(1)})$. We recall that $\nu=v(x)dx$ denotes the limit density in the central limit theorem  \eqref{CLT} for the $\mu$-walk on $(\kg,*)$.

Our argument relies  on the uniform local limit theorem established in \Cref{Sec-uniform-LLT} and a good control of the limit density $\nu$. Accordingly, all the estimates of this section are still valid if we only assume $\mu$ to have a finite moment of order strictly greater than $2$ and replace $\mu^{*N}$ by its truncated approximation $\Theta^N$, see  \Cref{Sec-uniform-LLT}. We leave this verification to the reader.

\subsection{Comparison with Lebesgue}

We prove \Cref{mu/Leb}. It follows from the following more general result. 
\begin{proposition}[Asymptotic with variable recentering] \label{muvsleb} 
For every $f\in C_{c}(\kg)$, $C>0$ $N\geq 1$, $g,h\in \DilN (B_{\tkg}(0,C))$ we have 
$$ (\mu^{*N}*-N\Xm)_{p(g), p(h)}(f) = v(x_{N,g,h}) N^{-\dd/2}\int_{\kg}f dx \,\,+\,\, o_{ f, C}(N^{-\dd/2})$$
where $v$ is the density of the limit measure $\nu=v(x)dx$, and  
$$x_{N,g,h}= p\left(\DilsN(g)\right)^{-1} *p\left(\DilsN(h)^{-1}\right)$$ 
for $p:\tkg\rightarrow \kg$  the projection $x\oplus t\chi\mapsto x+t\Xm$.
\end{proposition}

\noindent\emph{Remark}. Recall from \cite[Section 3.3.2]{benard-breuillard-CLT} that $v(0)$ may vanish. In particular, the sequence $ (\mu^{*N}*-N\Xm)(f)$ may be $o_{ f, C}(N^{-\dd/2})$. This never happens in the centered case  (i.e. $\Xab=0$) for which $v>0$ (see \Cref{fact-gaussian}).

\begin{proof}

 By \Cref{TLLfcompact},
\begin{align*}
(\mu^{*N}*-N\Xm)_{p(g), p(h)}(f) = (\DilN \nu)_{p(g), p(h)}(f)+o_{f,C}( N^{-\dd/2}).
\end{align*}

Let $\gamma_{0} <1/2$. Set $e_{\xi,N,g,h}= e^{-2i\pi \xi\circ p(g^{(\chi)}+h^{(\chi)})}$ and $g_{N}=\DilsN g$, $h_{N}=\DilsN h$. Applying Fourier inversion formula, \Cref{compl} and \Cref{remplacement*}, we obtain
\begin{align*}
(\DilN \nu)_{p(g), p(h)}(f) &=\int_{\dkg} \hf(\xi) \widehat{(\DilN \nu)_{p(g), p(h)}}(\xi) d\xi\\
 &=\int_{\dkg} \hf(\xi) \widehat{(\DilN \nu)_{g, h}}(\xi) e_{\xi,N,g,h} d\xi\\
 &=\int_{\UN(\gamma_{0})} \hf(\xi) \widehat{(\DilN \nu)_{g, h}}(\xi) e_{\xi,N,g,h} d\xi   + \|f\|_{L^1}O_{\gamma_{0}}(N^{-\frac{\dd+1}{2}})\\
&= \int_{\UN(\gamma_{0})} \hf(\xi) \widehat{(g*'\DilN \nu *'h)}(\xi) e_{\xi,N,g,h} d\xi   + \|f\|_{L^1}O_{\gamma_{0}}(N^{-\frac{\dd+1}{2}+O(\gamma_{0}) })\\
&= \int_{\UN(\gamma_{0})} \hf(\xi) \widehat{(\DilN (g_{N}*'\nu *'h_{N}))}(\xi) e_{\xi,N,g,h} d\xi   + \|f\|_{L^1}O_{\gamma_{0}}(N^{-\frac{\dd+1}{2}+O(\gamma_{0}) })\\
&= \int_{\UN(\gamma_{0})} \hf(\xi) \widehat{(\DilN (g_{N}*\nu *h_{N}))}(\xi) e_{\xi,N,g,h} d\xi   + \|f\|_{L^1}O_{\gamma_{0}}(N^{-\frac{\dd+1}{2}+O(\gamma_{0}) })\\
&=[\DilN (p(g_{N})*\nu *p(h_{N}))](f)+ \|f\|_{L^1}O_{\gamma_{0}}(N^{-\frac{\dd+1}{2}+O(\gamma_{0}) }).
 \end{align*}
Pursuing the computation
\begin{align*}
[\DilN (p(g_{N})*\nu *p(h_{N}))](f)
&= \int_{\kg} f\left(\DilN (p(g_{N})* x*p(h_{N}))\right) v(x)dx\\
&= N^{-\dd /2} \int_{\kg} f(x)\, v \left(p(g_{N})^{-1}*\DilsN x*p(h_{N})^{-1} \right) dx\\
&= v \left(p(g_{N})^{-1}*p(h_{N})^{-1} \right) N^{-\dd /2} \int_{\kg} f(x) dx  +r
 \end{align*}
 where the difference $r$ is bounded by
 \begin{align*}
r \leq  N^{-\dd /2} \int_{\kg}  \big|v \left(p(g_{N})^{-1}*\DilsN x*p(h_{N})^{-1} \right) -  v \left(p(g_{N})^{-1}*p(h_{N})^{-1} \right)\big| \,\,|f(x)|dx
 \end{align*}
whence is $o_{C}(N^{-\dd/2})$  by dominated convergence. 
This concludes the proof. 
\end{proof}

\subsection{Comparison with $\nu$} \label{Sec-nu/mu}

The goal of the section is to show \Cref{mu/nu}. We  assume that $\mu$ is {\bf centered}, aperiodic, with finite moment of order $m_{\mu}>\dg+2$ for the central filtration, and we show that the asymptotic $\mu^{*N}(f) \sim (\DilN \nu)(f)$ implied by \Cref{mu/Leb}, actually holds uniformly for a certain range of deviations, namely those in 
 $$W_{N}(c):=\{x\in \kg,\,|x^{(i)}|\leq c\left(N\log N\right)^{i/2} \}$$
 for $c>0$ small.

 The assumption that $\mu$ is centered is used to guarantee that the limit density $\nu=v(x)dx$ has a  Gaussian lower bound. Indeed, we record the following fact \cite{varopoulos-saloff-coulhon92}.
\begin{fact}\label{fact-gaussian}
Assume $\Xab=0$. Let $\norm{.}$ denote the distance to $0$ for a left-invariant distance on $(\kg,*)$. There exists $A>1$ such that for all $x\in \kg$
$$A^{-1} e^{-A \norm{x}^2} \leq v(x) \leq A e^{-A^{-1} \norm{x}^2}.$$
\end{fact}

Using the  lower bound, we show  
\begin{proposition} \label{nu-min-Fcompact}
Assume $\Xab=0$. Let $\delta>0$ and $c\lll_{\delta} 1$. Let $f\in C^\infty(\kg)$ be a non-negative integrable function  with compactly supported Fourier transform, such that $f\neq 0$. For all $N\geq 1$ large enough,  $g,h \in W_N(c)$
$$(\DilN \nu)_{g,h}(f)\geq  N^{-\frac{\dg}{2} -\delta}. $$
\end{proposition}

Combining with \Cref{TLLFouriercompact}, we get 

\begin{corollary} \label{mu/nuFcompact}
Assume $\Xab=0$, $m_{\mu}>\dg+2$ and $c\lll 1$. Let $f\in C^\infty(\kg)$ be a non-negative integrable function  with compactly supported Fourier transform, such that $f\neq 0$. For $N\geq 1$,  $g,h\in W_{N}(c)$, 
we have
$$\frac{\mu^{*N}_{g,h}(f)}{(\DilN \nu)_{g,h}(f)} = 1 +o_{f}(1).$$
\end{corollary}

We will then deduce  \Cref{mu/nu}  by applying a refinement of the approximation technique used in the proof of the uniform local limit theorem.

\bigskip

\begin{proof}[Proof of \Cref{nu-min-Fcompact}]. 
 We can write $(\DilN \nu)_{g,h}=N^{-\dg/2}w_{N,g,h}(x) dx$ 
where $$w_{N,g,h}(x)=v\circ r_{N,g,h}(x) \,\, \,\, \,\, \,\, \,\, \,\,\,\,\,\,\,\,\,\,\,\,r_{N,g,h}(x)=\DilsN \left( g^{-1}*x *h^{-1} \right).$$
We need to show that for $c\lll_{\delta}1$, $N\geq 1$ large enough (possibly depending on $f$), $g,h\in W_{N}(c)$, 
$$\int_{\kg} f(x) w_{N,g,h}(x)dx \geq N^{-\delta}. $$
As $f$ is non-negative non-identically zero, it is enough to show that for $c\lll_{\delta}1$, any compact subset of $\kg$ is included in the set $\{w_{N,g,h} \geq N^{-\frac{\delta}{2}}\}$ for  $N$  large enough and $g,h\in W_{N}(c)$. To analyse this set, we use the lower bound in \Cref{fact-gaussian}

Consider $x$  such that $w_{N,g,h}(x) < N^{-\frac{\delta}{2}}$. Using  \Cref{fact-gaussian}, we infer
\begin{align*}
A^{-1} e^{-A \norm{r_{N,g,h}(x)}^2} < N^{-\frac{\delta}{2}}
\end{align*}
which can be rewritten 
$$ \frac{\delta}{2A} \log N - \frac{\log A}{A} \leq \norm{r_{N,g,h}(x)}^2$$
so for $N\ggg_{A, \delta} 1$, 
$$ \sqrt{\frac{\delta \log N}{3A }} \leq \norm{r_{N,g,h}(x)}.$$
This discussion justifies the inclusion
$$\{w_{N,g,h} \geq N^{-\frac{\delta}{2}}\} \supseteq \left\{x\in \kg\,:\, \norm{r_{N,g,h}(x)} < \sqrt{\frac{\delta \log N}{3A }} \right\}$$
for any $N\ggg_{A, \delta} 1$.
However, by the ball-box principle, 
$$\norm{r_{N,g,h}(x)} \ll \max_{i\leq s}\|(r_{N,g,h}(x))^{(i)}\|^{1/i}= N^{-1/2} \,\max_{i\leq s}\|(g^{-1}*x*h)^{(i)}\|^{1/i}$$
so
\begin{align*}
\{w_{N,g,h} \geq N^{-\frac{\delta}{2}}\}& \supseteq \left\{x\in \kg\,:\,\forall i\leq s, \,\|(g^{-1}*x*h^{-1})^{(i)}\| \lll_{A, \delta} (N\log N)^{i/2}  \right\}.\\
\end{align*}
Noticing that for any $\delta>0$, we have $W_{N}(\delta)*W_{N}(\delta)\subseteq W_{N}(O(\delta))$, we conclude that for $c\lll_{A, \delta}1$, $N\ggg_{A, \delta} 1$, $g,h\in W_{N}(c)$, 
\begin{align*}
\{w_{N,g,h} \geq N^{-1/4}\}& \supseteq W_{N}(c).
\end{align*}
As $A$ can be fixed as a function of the initial data, our conventions allow to remove it from the subscript, see \Cref{Sec-cadre-asympt}.

\end{proof}

We can now conclude the proof of \Cref{mu/nu}. 

\begin{proof}[Proof of \Cref{mu/nu}]
Fix $\delta \in (0,1/2]$ such that $\mu$ has a finite moment of order $\dg+2+2\delta$. We choose $\delta$ is a function of $\mu$, so it does not appear in the subscripts of our Landau and Vinogradov notations (see \Cref{Sec-cadre-asympt}). Fix $c\lll 1$ as in \Cref{nu-min-Fcompact} and  \Cref{mu/nuFcompact}.  Let $\eps>0$ and fix  $f^-\leq f \leq f^+$  an $\eps$-approximation of $f$ by functions with compactly supported Fourier transforms, as in \Cref{approx}. As  $f$ is non negative and non identically zero, we have for $N\geq 1$, $g,h\in W_N(c)$
\begin{align*}
\frac{\mu^{*N}_{g,h}(f)}{(\DilN \nu)_{g,h}(f)}&\geq \frac{\mu^{*N}_{g,h}(f^-)}{(\DilN \nu)_{g,h}(f^+)} \\   
&= \frac{(\DilN \nu)_{g,h}(f^-) +o_{f^-}(N^{-\frac{\dg}{2}-\delta}) }{(\DilN \nu)_{g,h}(f^+)} \\ 
&= \frac{(\DilN \nu)_{g,h}(f^-)}{(\DilN \nu)_{g,h}(f^+)} +o_{f^\pm}(1)   
\end{align*}
where the second line uses the local limit theorem with power saving from \Cref{TLLFouriercompact}, and the last line uses \Cref{nu-min-Fcompact}.

Moreover, 
\begin{align*}
\frac{\mu^{*N}_{g,h}(f)}{(\DilN \nu)_{g,h}(f)}&\leq \frac{\mu^{*N}_{g,h}(f^+)}{(\DilN \nu)_{g,h}(f)} \\   
&= \frac{\mu^{*N}_{g,h}(f^+)}{(\DilN \nu)_{g,h}(f^+)} \frac{(\DilN \nu)_{g,h}(f^+)}{(\DilN \nu)_{g,h}(f)} \\ 
&= (1+o_{f^+}(1)) \frac{(\DilN \nu)_{g,h}(f^+)}{(\DilN \nu)_{g,h}(f)} 
\end{align*}
where the last line uses \Cref{mu/nuFcompact}.
To conclude, it is enough to prove that  
$$\frac{(\DilN \nu)_{g,h}(f^-)}{(\DilN \nu)_{g,h}(f^+)}=1+O_{f}(\eps) +o_{f^{\pm}}(1).$$
Only the lower bound needs a proof. We can reformulate it as
\begin{align} \label{appII}
(\DilN \nu)_{g,h}(f^+-f^-)\leq \left(O_{f}(\eps)+o_{f^{\pm}}(1)\right)(\DilN \nu)_{g,h}(f^+).
\end{align}
To prove \eqref{appII}, we  write as above $(\DilN \nu)_{g,h}=N^{-\dg/2}w_{N,g,h}(x) dx$ 
where $$w_{N,g,h}(x)=u\circ r_{N,g,h}(x) \,\,\,\,\,\,\,\,\,\,\,\,r_{N,g,h}(x)=\DilsN (g^{-1}*x*h^{-1}).$$
We also introduce the domain 
$$\Omega_N :=\{x\in \kg\,:\, \|x\|\leq N^{\frac{1}{4s}}\}.$$

\begin{lemme} \label{var}
For $c\lll1$, there exists $V>1$ such that for all $N\geq1$, $g,h \in W_N(c)$, 
$$\frac{\sup_{\Omega_N}w_{N,g,h}}{\inf_{\Omega_N}w_{N,g,h}}\leq V. $$
\end{lemme}

Let us see how the lemma allows to conclude. Write $\psi=f^+-f^-$. 
\begin{align*}
(\DilN \nu)_{g,h}(f^+-f^-) &= N^{-\dg/2}\int_{\kg} \psi(x) w_{N,g,h}(x) dx\\   
 &=\underbrace{ N^{-\dg/2}\int_{\Omega_N} \psi(x) w_{N,g,h}(x) dx}_{T_N} \,+\,\underbrace{N^{-\dg/2}\int_{\kg\smallsetminus \Omega_N} \psi(x) w_{N,g,h}(x) dx}_{S_N}.
\end{align*}

Asking that $c$ also satisfies \Cref{var}, and taking $N$ large enough so that $\int_{\Omega_N}f\geq \frac{\|f\|_{L^1}}{2}$, we can control the first term:
\begin{align*}
T_N &\leq N^{-\dg/2} \int_{\Omega_N} \psi(x) \sup_{\Omega_N}w_{N,g,h}\ dx\\
&\leq N^{-\dg/2} \|\psi\|_{L^1} V\inf_{\Omega_N}w_{N,g,h}\\
&\leq N^{-\dg/2} \|\psi\|_{L^1} V \frac{2}{\|f\|_{L^1}}\int_{\Omega_N} f(x) w_{N,g,h}(x)dx \\
&\leq \|\psi\|_{L^1} V \frac{2}{\|f\|_{L^1}}   (\DilN \nu)_{g,h}(f^+). \\
\end{align*}
As $\|\psi\|_{L^1}\leq \eps$, we finally have for $N\geq 1$, $g,h\in W_N(c)$,
\begin{align*}
T_N=  O_{f}(\eps) \, (\DilN \nu)_{g,h}(f^+).
\end{align*}

Let us now deal with $S_N$. Using that $w_{N,g,h}$ is bounded by $\sup v$, and that $f^+, f^-$ are rapidly decreasing, we obtain for every $Q>0$
\begin{align*}
S_N &\ll N^{-\dg/2}\int_{\kg\smallsetminus \Omega_N} \psi(x)  dx \\
&= O_{f^+, f^-, Q}(N^{-Q}).
\end{align*}
Choosing $Q=\frac{\dg}{2}+1$ and using the lower bound $(\DilN \nu)_{g,h}(f^+)\geq  N^{-\frac{\dg}{2} -\delta}$ from \Cref{nu-min-Fcompact} and our choice for $c$, we deduce that 
\begin{align*}
S_N = o_{f^{\pm}}(1) (\DilN \nu)_{g,h}(f^+).
\end{align*}

This concludes the proof of \Cref{mu/nu} modulo \Cref{var}, which we prove below. 
\end{proof}

\begin{proof}[Proof of \Cref{var}]
We need to show that for $c\lll1$, the maximum variation
$$\sup_{x,y\in \Omega_N}\big|\log(w_{N,g,h}(x)) -\log(w_{N,g,h}(y))\big|$$ 
is bounded  by some constant $V'>0$ independent of $N\geq1$, $g,h\in W_N(c)$. Enumerate as $e_1, \dots, e_d$ the basis $(e^{(i)}_j)$ of $\kg$ fixed in \Cref{Sec2} and write $\partial_i$ for the derivation in the variable $e_i$ (in the classical sense). Fixing $i\in \{1, \dots, d\}$, it is sufficient to show that for $c\lll 1$, $x\in \kg$
$$|(\partial_i \log w_{N,g,h})(x)| \leq |P_{N,g,h}(x)|$$
where $P_{N,g,h}$ is a polynomial on $\kg$ of degree at most $s$, and whose coefficients are $O(\frac{(c\log N)^{O(s)}}{N^{1/3}})$ uniformly in $N\geq2$, $g,h\in W_N(c)$. 
 
Let us analyse this derivative. We can rewrite
$$(\partial_i \log w_{N,g,h})(x) = (\partial_i r_{N,g,h})(x) \,\frac{\partial_i v }{v}(r_{N,g,h}(x)).$$

We study each term in the product.  

\bigskip

For the first term $\partial_i r_{N,g,h}$, recall that 
$r_{N,g,h}(x)=\DilsN (g^{-1}*x*h^{-1})$. As in the proof of \Cref{TLLFouriercompact} (\Cref{Sec-endproof}), we can write $r_{N,g,h}$ as a sum 
$$r_{N,g,h}(x)= Q_N(\DilsN g, \DilsN x, \DilsN h)$$ 
where   $Q_N$ is a polynomial on $\kg^3$ of degree at most $s$,  only\footnote{$Q_N$ also depends on the initial data $\kg$, $\mu$, etc. fixed once and for all in  \Cref{Sec-cadre}.} depending on $N$, and with coefficients $O(1)$. 

As $\|\DilsN g\|,  \|\DilsN h\| =O( c(\log N)^{s/2})$, and as any term in $Q_N$ with non-zero $\partial_i$-derivative in $x$ must involve $\DilsN x$, we infer that $\partial_i r_{N,g,h}(x)$ is  a polynomial whose coefficients  are $O(\frac{(c\log N)^{O(s)}}{\sqrt{N}})$ uniformly in $N\geq2$, $g,h\in W_N(c)$.

\bigskip

For the second term $\frac{\partial_i v }{v}(r_{N,g,h}(x))$, we bound the numerator by a constant and use the lower bound on $v$ from \Cref{fact-gaussian} to get
$$|\frac{\partial_i v }{v}(r_{N,g,h}(x))|\ll A e^{ A \norm{r_{N,g,h}(x)}^2}. $$
However,  we have for all $N\geq 1$, $x\in \Omega_N$, $g,h\in W_N(c)$ that 
$$\norm{r_{N,g,h}(x)} \leq O(c) \sqrt{\log N}.$$
It follows that 
$$|\frac{\partial_i v }{v}(r_{N,g,h}(x))|\ll A N^{A O(c)^2} \ll N^{\frac{1}{6}}$$
up to taking $c\lll_A 1$. As $A$ can be fixed as a function of the initial data, our conventions allow to remove it from the subscript, see \Cref{Sec-cadre-asympt}.  

\bigskip

The result follows by combining the discussions for each term.

\end{proof}


\section{Consequences for Choquet-Deny and equidistribution.} \label{Sec-applications}

\subsection{Bounded harmonic functions on nilpotent groups}

\label{Sec12} 

In this paragraph, we prove \Cref{Choquet-Deny}, namely the Choquet-Deny property. While doing so, we also establish the proposition below, which is an important consequence of our truncated uniform local limit theorem \ref{LLT-tronc}, and will be crucial in the next section.


\begin{proposition} \label{pixel}
Let $(\kg,*)$ be a simply connected nilpotent Lie group. Let $\mu$ be  an aperiodic probability measure on $\kg$ with a finite moment  of order strictly greater than $2$ for the weight filtration induced by $\XXab=\E(\mu_{ab})$. Let $\nu$ be the limit measure arising in the central limit theorem \eqref{CLT} for some choice of weight decomposition $\kg=\oplus \km^{(b)}$. Fix a right invariant distance  on $(\kg,*)$ and let 
$$\mathscr{L}=\{F: \kg\rightarrow \R \,\,|\,\, F\, \,1\emph{-lipschitz},\, \|F\|_{\infty}\leq 1\}.$$ Then $$\sup_{F\in \mathscr{L}}\, |\mu^{*N}(F)-(\DilN \nu*N\XX)(F)|\underset{N\to +\infty}{\longrightarrow} 0.$$

\end{proposition}


\noindent\emph{Remarks}.
1) Note that $F$ may not decay at infinity. The idea of the proof is to pixelize $F$, and then apply the uniform local limit theorem to each pixel. 

2) If we replace the right-invariant metric by a left invariant metric, then the statement still holds, following the same line of argument. We will need the metric to be right-invariant below in order to regularize bounded harmonic functions for the right-random walk using convolution on the left. 

\begin{proof} 
We fix  $0<\gamma_{0}\lll 1$ so that the truncation $\Theta^N$ defined in  \Cref{Sec-uniform-LLT} satisfies the uniform local limit theorem \ref{LLT-tronc}. As $\mu$ has in particular finite second moment, \Cref{truncation-cost} yields that 
 $$\| \mu^{*N} -\Theta^N\|\to0 $$ in total variation. Hence, we may replace $\mu^{*N}$ by $\Theta^N$ in order to prove \Cref{pixel}.

Let $\eps>0$.  We fix a ``bump function'' $\chi_{\eps} : \kg \rightarrow \R^+$, i.e.  a continuous function supported on the ball  $B_{\kg}(0, \eps)$ with total sum $1$. We introduce the convolution $F_\eps:=\chi_\eps*F$, given by 
$$F_{\eps}(x)= \int_{\kg}\chi_{\eps}(u^{-1}) F(u x)du=\int_{\kg}\chi_{\eps}(xu^{-1}) F(u)du$$
where, here and below, we write $xy=x*y$ and $x^{-1}=-x$ to lighten notations. 
As $F$ is $1$-Lispchitz we have $\|F-F_{\eps}\|_{\infty}\leq \eps$, so for all $N\geq 1$, \begin{equation}\label{fee} |\Theta^N(F)-(\DilN \nu*N\XX)(F)|  \leq |\Theta^N(F_{\eps})-(\DilN \nu*N\XX)(F_{\eps})|   +2\eps. \end{equation}
On the other hand
\begin{align} \label{muNFeps}
\Theta^N(F_{\eps})&=\int_{\kg}(\Theta^N*\delta_{u^{-1}}) (\chi_{\eps})\,  F(u) \,du.
 \end{align}
   
 The uniform local limit theorem \ref{LLT-tronc} with arbitrary deviations by right multiplication implies that uniformly in $u$
 \begin{align} \label{TLLchi}
 \Theta^N*\delta_{u^{-1}} (\chi_{\eps}) = \DilN \nu *N\XX* \delta_{u^{-1}} (\chi_{\eps}) + o_{\gamma_{0}, \chi_{\eps}}(N^{-\frac{\dd}{2}})
\end{align}
where the $o_{\gamma_{0}, \chi_{\eps}}(.)$ does not depend on $u$. However, we may not replace directly $\Theta^N$ by $\DilN \nu *N\XX$ in the integral because the error term is a priori not integrable. To remedy this, we truncate the domain of integration in a part on which the error integrates to   $o_{\gamma_{0}, \chi_{\eps}}(1)$, and a part that is negligible thanks to the central limit theorem. More precisely, for $C>0$ we set 
$$Q_{C, N}=\{u\in \kg \,:\, \forall i,\,\,\|u^{(i)}\|\leq C N^{i/2}\}.$$
By the central limit theorem \eqref{CLT}, and the observation that $Q_{C,N}= \DilN (Q_{C,1})$, we may choose $C=C(\mu, \eps)$ large enough so that for $N\geq 1$, 
$$(\Theta^N*-N\XX)(Q^c_{C,N}) <\eps\,\,\,\,\,\,\,\,\,\,\,\,\, \,\,\,\,\,\,\DilN \nu(Q_{C,N}) < \eps $$
where $Q^c_{C,N}:=\kg \smallsetminus Q^c_{C,N}$. Then we split the integral in \eqref{muNFeps} into two parts corresponding to the integration domains $Q_{2C,N}*N\XX$ and $Q^c_{2C,N}*N\XX$. Since $\|F\|_\infty \leq 1$, the second part is bounded above by $(\Theta^N*-N\XX) (B_{\kg}(0,\eps)*Q^c_{2C,N})$, which is itself at most $(\Theta^N*-N\XX)(Q^c_{C,N}  )<\eps$. The same applies to $\DilN \nu$ in place of $\Theta^N*-N\XX$, so in the end in view of \eqref{TLLchi}, we get:
\begin{align*}
|\Theta^N(F_{\eps}) - \DilN \nu *N\XX(F_{\eps}) | \leq 2\eps + o_{\gamma_{0}, \chi_{\eps}}(1). 
\end{align*}

Finally, for large $N$  (depending on $\mu$, $\gamma_0$, $\eps$, $\chi_{\eps}$ but not on $F\in \mathscr{L}$), the left-hand side in \eqref{fee} is at most $5\eps$, and the final result follows letting  $\eps$ go to $0$. 
\end{proof}

Finally, we show  \Cref{Choquet-Deny}.

\begin{proof}[Proof  \Cref{Choquet-Deny}]
 We  identify $G$ with its Lie algebra $\kg$ via the exponential map, and denote by $*$ the induced group structure. Up to replacing $\mu$ by $\frac{1}{2}\mu+\frac{1}{2}\delta_0$, which has the same harmonic functions, we may assume that $\mu$ is aperiodic. 

We endow $\kg$ with a $*$-right invariant Riemannian metric. Without loss of generality we may further restrict attention to harmonic functions $F$ that are Lipschitz. Indeed, we may replace $F$ by the convolution $\chi *F$, where $\chi$ is any smooth bump function with integral $1$. This new function remains harmonic and bounded. It is furthermore Lipschitz.  Finally if $\chi *F$ is a.e. constant for all such $\chi$, then $F$ too is a.e. constant. 

Arguing by induction on the dimension of $\kg$ (applying the result to $\kg/\mathfrak{z}(\kg)$), it is enough to prove that $F$ is invariant under translation by the center $\mathfrak{z}(\kg)$ of $\kg$. 

So, assume that $F$ is Lipschitz, pick $x_0 \in \kg$ and $z \in \mathfrak{z}(\kg)$. The rest of the proof is very similar to Nelson's proof of Liouville's theorem \cite{nelson61}.  The left translate $F_{x_{0}}(x)=F(x_0x)$ is again bounded, harmonic and it remains Lipschitz\footnote{The Lipschitz constant of $F_{x_{0}}$ depends on $x_{0}$ but this is not an issue.}. Thus by \Cref{pixel} we have
\begin{equation}\label{Fx0F}F(x_{0}) =\mu^{*N}(F_{x_{0}}) = (\DilN \nu*N\XX)(F_{x_{0}})+o_{F_{x_{0}}}(1)
\end{equation} 
and the same holds with $x_{0}z$ in place of $x_{0}$. 
But since $z$ lies in the center $\mathfrak{z}(\kg)$ we have:
$$(\DilN \nu*N\XX)(F_{x_{0}z})= \int_{\kg} F(x_{0} z x *N\XX) \,d\DilN \nu(x) = \int_{\kg} F(x_{0}x *N\XX) \,d(\DilN \nu*\delta_z)(x).$$
In view of \eqref{Fx0F}, we can conclude that $F(x_0)=F(x_0z)$ provided we show that
$$\int_{\kg} F(x_{0}x *N\XX) \,d(\DilN \nu*\delta_z)- \DilN \nu )(x)\longrightarrow 0.$$
But this is indeed the case, since setting $\DilN \nu= u_{N}(x) dx$ where $v_{N} =N^{-\frac{\dd}{2}} v \circ \DilsN $, the above integral is bounded above by:
\begin{align*}
\|F\|_{\infty}\int_{\kg} |v_{N}(x+z)-v_{N}(x)| dx =\|F\|_{\infty}\int_{\kg} |v(x+\DilsN z)-v(x)| dx
\end{align*}
which tends to $0$ as $N\to+\infty$.

\end{proof}


\subsection{Equidistribution in homogeneous spaces} \label{Sec-equid-homogene}

In this section we prove Theorem \ref{equid-homogene} regarding equidistribution of unipotent random walks on homogeneous spaces of finite volume. In a nutshell, we will use our LLT (twice!) in the guise of Lemma \ref{pixel} above in order to reduce to a version of Ratner's equidistribution theorem proved by Shah, Theorem \ref{Ratner-Shah} below. A similar strategy was employed in \cite{breuillard05}. 

We first reduce to the case of aperiodic walks thanks to the following general lemma.

\begin{lemme}[Reduction to aperiodic case] \label{reduction-aper}
Let $\mu$ be a probability measure on a group $G$ and  $\mu'=\frac{1}{2}\mu + \frac{1}{2}\delta_{1}$. Then for all $N>1$,
$$\|\frac{1}{2N}\sum_{k=1}^{2N}\mu'^k- \frac{1}{N}\sum_{k=1}^N\mu^k\| \leq C \frac{\log N}{\sqrt{N}},$$
where $\|.\|$ is the total variation norm and $C$ an absolute constant.
\end{lemme}

\begin{proof}
Expanding the power $\mu'^k$, and setting $\binom{k}{i}=0$ if $i>k$,  we may write
\begin{align*}
\sum_{k=1}^{2N}\mu'^k = \sum_{i=1}^{2N}  \sum_{k=1}^{2N} \frac{1}{2^k} \binom{k}{i}  \mu^i = \sum_{i=1}^{2N} \E( N_{2N}(i)) \mu^i
\end{align*}
where $N_p(i)$ is the occupation time $N_p(i):=|\{k\leq p, S_{k}=i\}|$ and $(S_{k})_{k\geq0}$ is the Markov chain on $\N_{\ge 0}$ starting at $0$ with transition probabilities $p(i,i)=p(i,i+1)=1/2$, for $i \ge 0$. We can then write:
\begin{equation}\label{bbmar}
\|\frac{1}{2N}\sum_{k=1}^{2N}\mu'^k- \frac{1}{N}\sum_{k=1}^N\mu^k\| \leq \frac{1}{N}\sum_{i=1}^{2N} |\frac{1}{2}\E( N_{2N}(i)) -\1_{i\leq N}|.
\end{equation}
We split the right hand side in three terms $A_1+A_2+A_3$ according to the position of $i$ with respect to $p_{N}:=\lfloor N- \sqrt{N}\log N \rfloor$ and $q_{N}=\lfloor N+ \sqrt{N}\log N \rfloor$. 

Recall, either from the general theory of discrete Markov chains or in this case by a direct calculation, that $\E(N_\infty(i))=2$ for all $i$. In particular, if $i \in [p_N,q_N]$, each term in the sum $A_2$ is bounded by $2$, so $A_2 \leq 4\log N/\sqrt{N}$.  Next:\begin{align*}A_3 & = \frac{1}{N} \sum_{i \ge q_N} \E(N_{2N}(i)) = \frac{1}{N} \sum_{k\leq 2N}\sum_{i \ge q_N} \PP(S_k=i) = \frac{1}{N} \sum_{k\leq 2N}  \PP(S_k\ge q_N)\\ & \leq  2 \PP(S_{2N} \ge q_N). \end{align*} However $S_{2N}$ is a  binomial  random variable and by classical tail bounds $\PP(S_{2N} \ge q_N)$ decays faster than any inverse power of $N$. So $A_3 \leq \log N/\sqrt{N}$ for $N$ large enough.

Finally if $i\le p_N$,

\begin{align*}
 |\frac{1}{2}\E( N_{2N}(i)) -1| = \frac{1}{2} |\E( N_{2N}(i)) -\E( N_{\infty}(i)) | \leq \frac{1}{2} \sum_{k>2N} \PP(S_k=i).
\end{align*}
Summing over $i\le p_N$, we get:
$$A_1\leq \frac{1}{2}\sum_{k>2N} \PP(S_k \leq p_N) \leq \frac{1}{2}\sum_{k>2N} \PP(S_k \leq k/2- \sqrt{k/2}\log(k/2)).$$
Again each term of the latter sum decays faster than any inverse power of $k$, hence so does the sum as a function of $N$. This ends the proof.
\end{proof}




We keep the notation of Theorem \ref{equid-homogene} and consider the Alexandrov compactification $\Omega\cup \{\infty\}$ of the homogeneous space $\Omega=H/\Lambda$. The next proposition ensures that the sequence of measures has no escape of mass and that any limit measure is absolutely continuous with respect to the homogeneous measure $m_\omega$ given by Ratner's theorem:

\begin{proposition}[Domination by $m_{\w}$] \label{domination}
There exists $C>0$ such that every weak limit $m$ of the sequence of probability measures $(\frac{1}{N}\sum_{k=1}^N\mu^k*\delta_{\w})_{n\geq 1}$ on $\Omega \cup \{\infty\}$ satisfies $m\leq C m_{\w}$.
The result holds without  Ces\`aro if $\mu$ is centered.   
\end{proposition}

Before giving the proof, we recall the Ratner-Shah equidistribution theorem. Let $\ku$ be the Lie algebra of the connected $Ad$-unipotent subgroup $U \leq H$, so $(\ku,*)$ is the simply connected covering of $U$.  We say that a vector space  basis  $(e_{1}, \dots, e_{d})$ of $\ku$ is  \emph{triangular} if it satisfies $[e_{i}, e_{j}]\subseteq \text{Vect}\{e_{k} , \, k>\max(i,j)\}$ for all $i,j$. A basis is \emph{regular} if it can be obtained by permuting  the elements in a triangular basis.

 \begin{theorem}[Ratner-Shah] \label{Ratner-Shah}
 Let $(e_{1}, \dots, e_{d})$ be a \emph{regular} basis  of $\ku$.  For $\ut=(t_{1}, \dots, t_{d})\in \R^d$, set $R_{\ut}:=[0, t_{1}] e_{1}*  \dots *[0, t_{d}]e_{d}$. Then for  $\w\in \Omega$, 
 $$\frac{\leb_{|R_{\ut}}}{\leb(R_{\ut})}*\delta_{\w} \underset{\inf t_{i}\to +\infty}{\longrightarrow}  m_{\w}$$
 where $\leb$ stands for the Lebesgue measure on $\ku$. 
 \end{theorem}

This theorem is due to Ratner when $U$ is one-dimensional. It was subsequently generalized by Shah \cite[Corollary 1.3]{shah94} in the present form. Note that Shah states it for simply connected $U$, but his proof does not require this assumption.

Applying this result to a finite union of rectangles, it clearly holds with $R_{\ut}$ replaced by $B_{\ut}:=[-t_{1}, t_{1}]e_{1}* \dots * [-t_{d}, t_{d}]e_{d}$, as $\inf t_{i} \to +\infty$.  

We now assume that  $e_{1}, \dots, e_{d}$ enumerates the basis $(e^{(i)}_{j})$ introduced  in \Cref{Sec2}. 
It is indeed a regular basis (thanks to the relation $[\kg^{(i_{1})}, \kg^{(i_{2})}]\subseteq \kg^{(i_{1}+i_{2})}$). In the non-centred case, we also assume $e_{d}=\XX$.   


\begin{proof}[Proof of \Cref{domination}]
First we observe that \Cref{pixel} allows us to replace $\mu^k$ by the measure $D_{\sqrt{k}}\nu*k\XX$ without loss of generality. Indeed, endow $H$ with a right invariant Riemannian metric, and $\Omega$  with the quotient metric. For every Lipschitz function $f$ on $\Omega$, we set $F:\ku \to \R$, $F(u)=f(\exp(u)\omega)$. Then $F$ is Lipschitz on $\ku$ for say a right-invariant Riemannian metric on $(\ku,*)$, and \Cref{pixel} guarantees that $$(\mu^k-D_{\sqrt{k}}\nu*k\XX)(F) \underset{k\to+\infty}{\longrightarrow} 0.$$
 Hence, the set of weak limits of  $\frac{1}{N}\sum_{k=1}^N\mu^k*\delta_\omega$ (or $\mu^N*\delta_{\w}$ in the centered case)  does not change if we consider $D_{\sqrt{k}}\nu*k\XX$ in place of  $\mu^k$. 

\bigskip
  \Cref{domination} now follows from the following claim.  

\begin{affirmation} There is $C>0$, such that for all   $N,\eps >0$ there are positive measures $m_{N}$ and $\sigma_{N}$ on $\Omega$, with  $\|m_{N}\|=1$ and $\|\sigma_{N}\|\leq \eps$, such that
$$\frac{1}{N}\sum_{k=1}^N D_{\sqrt{k}}\nu*k\XX \leq C m_{N} + \sigma_{N}$$
and  $m_{N}\to_{N \to +\infty} m_{\w}$. The statement also holds without  Ces\`aro average if $\XX=0$.
\end{affirmation}

Let $K_1$ be a compact neighbourhood of the origin in $\ku$ and let $K_r=D_{\sqrt{r}}K_1$. Since $\nu$ has a continuous density, there is $C_1>0$ such that $\nu\leq C_{1}\leb_{|K_r}+ \nu_{\ku\smallsetminus K_r}$ for all $r\ge 1$.
For all $N\geq1$, 
\begin{align*}
\DilN \nu &  \leq \DilN  C_{1}\leb_{K_r}+\DilN \nu_{\ku\smallsetminus K_r}= C_{1}N^{-\dd/2}\leb_{D_{\sqrt{rN}}K_1}+\DilN \nu_{\ku\smallsetminus K_r}\\
&\leq C_{1}M\frac{\leb_{B_{\ut_{Nr}}}}{\leb(B_{\ut_{Nr}})}+\DilN \nu_{\ku\smallsetminus K_r} = Cm_N + \sigma_N
\end{align*}
setting $C=C_1M$, where the constant $M$ and the box $B_{\ut_{rN}} \supset D_{\sqrt{rN}}K_1$ are given by   \Cref{boules-eucl-mult} below applied to $K_1$. Choosing $r$ such that $\|\nu_{\ku\smallsetminus K_r}\|<\eps$, the result follows in case $\mu$ is centered. 
In the general case when $\XX \neq 0$, we translate by $k\XX$ on the right and average over $k$:
$$\frac{1}{N}\sum_{k=1}^N D_{\sqrt{k}}\nu*k\XX \leq C\frac{1}{N}\sum_{k=1}^N \frac{\leb_{B_{\ut_{rN}} *k\XX}}{\leb(B_{\ut_{rN}})} + \sigma_{N}$$
where $\sigma_{N}=\frac{1}{N}\sum_{k=1}^N(D_{\sqrt{k}}\nu_{\ku\smallsetminus K_r})*k\XX$ has mass less than $\eps$. 
As we took care of choosing $e_{d}=\XX$,  the diffeomorphism $t_{1}e_{1}*\dots *t_{d}e_{d}\mapsto (t_{1}, \dots, t_{d})$ sends the Lebesgue measure $\leb_{B_{\ut_{rN}} *k\XX}$ to the product measure
$$\leb_{[-\ut_{rN,1}, \ut_{rN,1}]}\otimes \dots \otimes \leb_{[-\ut_{rN,d-1}, \ut_{rN,d-1}]}\otimes \leb_{[-\ut_{rN,d}, \ut_{rN,d}]+k}.$$
By construction of $\ut_{rN}$, we have  $\ut_{rN,d}=M \sqrt{rN}$, so for $N$ large enough
$$ \frac{1}{N}\sum_{k=1}^N\leb_{[-\ut_{rN,d}, \ut_{rN,d}]+k}\leq  \frac{2M\sqrt{Nr}+1}{N}\leb_{[-2N,2N]}.$$
Hence
$$\frac{1}{N}\sum_{k=1}^N D_{\sqrt{k}}\nu*k\XX \leq C(2M+1)\frac{\sqrt{r}\leb_{B_{\ut'_{rN}}}}{\sqrt{N}\leb(B_{\ut_{rN}})} + \sigma_{N}$$
 where $\ut'_{rN}$ is obtained from $\ut_{rN}$ by replacing the coordinate $\ut_{rN,d}$ par $2N$. As  $\leb(B_{\ut'_{rN}})$  is comparable to $\frac{2N}{\sqrt{rN}}\leb(B_{\ut_{rN}})$ the claim follows by \Cref{Ratner-Shah}.
\end{proof}

In the proof we used the following well-known comparison principle between Malcev coordinates of the first and second kind \cite[1.2]{corwin-greenleaf90}.

\begin{lemme}[comparison principle]  \label{boules-eucl-mult}
Let $K\subseteq \ku$ be a compact set. There exists $M>0$ such that for all $r>1$, there exists $\ut_{r}\in \R^d$ satisfying
$$D_{\sqrt{r}}K \subseteq B_{\ut_{r}}\,\,\,\,\,\,\,\,\,\,et \,\,\,\,\,\,\,\,\,\,\leb(B_{\ut_{r}}) \leq  M r^{\frac{\dd}{2}}.$$
\end{lemme} 


\begin{proof}
It is sufficient to check that for all $C>0$, there exists $M>0$ such that for all $r>1$, 
\begin{align} \label{eq-incl}
\left\{\sum_{i,j} s^{(i)}_{j}e^{(i)}_{j}\,: \,\forall i,j, \,|s^{(i)}_{j}|\leq C \right\} \subseteq D_{\frac{1}{\sqrt{r}}} \left\{ \prod_{i,j} t^{(i)}_{j}e^{(i)}_{j}\,: \,\forall i,j, \,|t^{(i)}_{j}|\leq Mr^{i/2}  \right\}
\end{align}
where the product follows the order prescribed the enumeration $e_{1}, \dots, e_{d}$.  But we can write  $D_{\frac{1}{\sqrt{r}}} \prod_{i,j} t^{(i)}_{j}e^{(i)}_{j}= \prod^{*'}_{i,j} r^{-i/2}t^{(i)}_{j}e^{(i)}_{j} + P_{r}((r^{-i/2}t^{(i)}_{j})_{i,j})$ where $*'$ is the graded product coming from $*$ for the weight filtration, and $P_{r}$ is obtained by dividing by a  (non zero)  power of $\sqrt{r}$ each coefficent of a polynomial function $P$ with $d$ variables. Then, we choose $r_{0}>0$ such that for $r>r_{0}$, $$\|P_{r}((r^{-i/2}t^{(i)}_{j})_{i,j})\|\leq 1$$
and $M>0$ large enough so that
$$\left\{\sum_{i,j} s^{(i)}_{j}e^{(i)}_{j}\,: \,\forall i,j, \,|s^{(i)}_{j}|\leq C+1 \right\} \subseteq  \left\{ {\prod_{i,j}}^{*'} t^{(i)}_{j}e^{(i)}_{j}\,: \,\forall i,j, \,|t^{(i)}_{j}|\leq M \right\}$$
This proves \eqref{eq-incl} for $r>r_{0}$. Up to increasing the value of $M$ even more, we get the result for all $r>1$. 
\end{proof}

We can finally conclude the proof of \Cref{equid-homogene}.

\begin{proof}[Proof of \Cref{equid-homogene}]
\Cref{domination} guarantees that every weak limit of the sequence
$(\frac{1}{N}\sum_{k=1}^N \mu^N*\delta_{\w})_{N\geq1}$
is a probability measure on $\Omega$ that is absolutely continuous with respect to   $m_{\w}$. It is also $\mu$-stationary, so it must be $U$-invariant by \Cref{Choquet-Deny} and hence $m_\w$ itself, because $m_\w$ is $U$-ergodic by Ratner's theorem \cite{ratner91}.

When $\mu$ is centered, we may apply the same argument to the sequence $(\mu^N*\delta_{\w})_{N\geq1}$, but we need to justify why limit measures are $\mu$-stationary. It is sufficient to check that for every compactly supported smooth function  $f$  on $\Omega$, we have $|\mu^{N+1}*\delta_\w(f)- \mu^{N}*\delta_\w(f)|\rightarrow 0$ as $N\to +\infty$.  Setting $F(u)=f(\exp(u)\w)$ we obtain a bounded Lipschitz function on $\ku$ and by \Cref{pixel}, it  boils down  to checking that  $|D_{\sqrt{N+1}}\nu(F)- \DilN \nu(F)|\rightarrow 0$. To see this, we observe that for the  norm in total variation, 
$$\|D_{\sqrt{N+1}}\,\nu-\DilN \nu\|= \|\nu -D_{\sqrt{\frac{N}{N+1}}}\,\nu\|\longrightarrow 0$$
whence the result.

\end{proof}

















\bigskip

\bigskip

\noindent\textsc{Mathematics Institute, University of Warwick, Coventry CV4 7AL, United Kingdom}

\noindent\textit{Email address}: \texttt{timothee.benard@warwick.ac.uk},

\bigskip

\noindent\textsc{Mathematical Institute, University of Oxford, Woodstock Rd Oxford OX2 6GG, United Kingdom}

\noindent\textit{Email address}: \texttt{breuillard@maths.ox.ac.uk}





\end{document}